\documentclass[11pt]{amsart}
\usepackage{amsmath}
\usepackage{amsfonts}
\usepackage{amssymb}
\usepackage{graphicx}
\usepackage[all,cmtip]{xy} % Use it for \xymatrix

\usepackage{aliascnt} % This seems necessary if I want both the same counters for defs, thms, lemmas, etc. and to use the \hyperref package. Without aliascnt, using a common counter leads \autoref to call everything a Theorem! I don't know if there is a better solution, but I found this one online at www.tug.org/applications/hyperref/ftp/doc/manual.html .

\theoremstyle{plain}
\newtheorem{theorem}{Theorem}[section]
\newtheorem{hypothesis}{Hypothesis}[section]
\newaliascnt{corollary}{theorem}

\aliascntresetthe{corollary}

\newaliascnt{lemma}{theorem}
\newtheorem{lemma}[lemma]{Lemma}
\aliascntresetthe{lemma}
\newaliascnt{proposition}{theorem}
\newtheorem{proposition}[proposition]{Proposition}
\aliascntresetthe{proposition}

\newaliascnt{hypotheses}{theorem}

\aliascntresetthe{hypotheses}

\theoremstyle{definition}
\newaliascnt{definition}{theorem}
\newtheorem{definition}[definition]{Definition}
\aliascntresetthe{definition}

\newaliascnt{example}{theorem}

\aliascntresetthe{example}

\newaliascnt{remark}{theorem}
\newtheorem{remark}[remark]{Remark}
\aliascntresetthe{remark}

\newaliascnt{remarks}{theorem}

\aliascntresetthe{remarks}

\numberwithin{equation}{section}

\usepackage{}
\usepackage[pdftex,breaklinks,colorlinks,
citecolor=blue,
linkcolor=blue,
urlcolor=blue]{hyperref}

\DeclareMathOperator{\Gr}{Gr}
\DeclareMathOperator{\image}{Im}

\newcommand{\mynote}[1]{}
\everymath{\displaystyle}

\begin{document}
\title[Kudla--Millson forms and degenerations of Hodge structure]{Kudla--Millson forms and one--variable degenerations of Hodge structure}
\begin{abstract}
We consider arbitrary polarized variations of Hodge structure of weight two and $h^{2,0}=1$ over a non--singular complex algebraic curve $S$ and analyze the boundary behaviour of the associated Kudla--Millson theta series using Schmid's theorems on degenerations of Hodge structure. This allows us to prove that this theta series is always integrable over $S$ and to describe explicitly the non-holomorphic part of the Kudla--Millson generating series in terms of the mixed Hodge structures at infinity.
\end{abstract}
\author{Luis E. Garc\'ia}
\maketitle
\setcounter{tocdepth}{1}
\tableofcontents

\section{Introduction}

The goal of this paper is to study the behaviour under degeneration of Hodge structure of certain theta series introduced by Kudla and Millson to study special cycles on Shimura varieties. Unlike previous work that analyzes the case where $S$ is a special subvariety in a toroidal compactification of a Shimura variety, here we consider polarized variations of Hodge structure with $h^{2,0}=1$ over an arbitrary non--singular complex curve $S$. 

\subsection{Main results} Let $\overline{S}$ be a connected compact Riemann surface and denote by $S$ the Riemann surface obtained by removing a finite number of points from $\overline{S}$. Consider an integral polarized variation of Hodge structure ($\mathbb{Z}$--PVHS)
\begin{equation} \label{intro:eq_1}
\xymatrix{
\mathbb{V}  = (\mathcal{V}_\mathbb{Z},Q,\mathcal{F}^\bullet) \ar[d] \\ S
}
\end{equation}
of weight two with $h^{2,0}=1$. Here $\mathcal{V}_\mathbb{Z}$ denotes the local system underlying $\mathbb{V}$ and we write $Q$ for the polarization and $\mathcal{F}^\bullet$ for the Hodge filtration. Let us write $\mathcal{L}$ for the line bundle $\mathcal{F}^2$ over $S$ and $\mathcal{V}^\vee_\mathbb{Z} \supseteq \mathcal{V}_\mathbb{Z}$ for the dual lattice of $\mathcal{V}_\mathbb{Z}$, that is,
$$
\mathcal{V}^\vee_{\mathbb{Z}s} = \left\{ v \in \mathcal{V}_{\mathbb{Q}s} \ | \ Q(v,v') \in \mathbb{Z} \text{ for all } v \in \mathcal{V}_{\mathbb{Z}s} \right\}.
$$
In order to state our main results succintly we will assume the following mild condition on $\mathbb{V}$ (cf. Remark \ref{remark:hyp_intro} below).

\begin{hypothesis} \label{hypothesis:hyp_on_V}
For any $s \in S$, the lattice $(\mathcal{V}_{\mathbb{Z}s},Q)$ is even and $\mathcal{V}_{\mathbb{Q}s}$ is a simple $\pi_1(S,s)$--module. Moreover, the fundamental group of $S$ acts trivially on $\mathcal{V}^\vee_\mathbb{Z}/\mathcal{V}_\mathbb{Z}$ and the monodromy of $\mathcal{V}_\mathbb{Z}$ around each $P \in \overline{S} \, \backslash \, S$ is unipotent and non--trivial.
\end{hypothesis}

We will be interested in the Noether--Lefschetz loci of $\mathbb{V}$: for a positive rational number $m$ and $\mu \in \mathcal{V}^\vee_\mathbb{Z}/\mathcal{V}_\mathbb{Z}$, define
\begin{equation} \label{def:NL_locus_mu_0}
\mathrm{NL}_\mathbb{V}(m)_\mu = \{ s \in S \ | \ \exists v \in (\mu + \mathcal{V}_{\mathbb{Z} s})  \cap \mathcal{F}^1_s \text{ with } Q(v,v)=2m \}.
\end{equation}
The locus $\mathrm{NL}_\mathbb{V}(m)_\mu \subset S$ has a natural complex analytic space structure \cite[\textsection 5.3.1]{VoisinII}; in fact, by the celebrated theorem of Cattani--Deligne--Kaplan it is a proper algebraic subset of $S$ with a natural scheme structure. Let us write $\mathrm{deg} \ \mathrm{NL}_\mathbb{V}(m)_\mu$ for the degree of the divisor naturally associated with $\mathrm{NL}_\mathbb{V}(m)_\mu$ and $\overline{\mathcal{L}}$ for Deligne's canonical extension of $\mathcal{L}$ to a line bundle over $\overline{S}$ and form the generating series
\begin{equation}
Z_\mathbb{V}^+(\tau)_\mu = -\deg (\overline{\mathcal{L}})\delta_{\mu,0} + \sum_{m >0} \deg \mathrm{NL}_\mathbb{V}(m)_\mu \cdot q^m, \quad q=e^{2\pi i \tau}.
\end{equation}

When $S$ is compact, the series $Z_\mathbb{V}^+(\tau)$ are known to be modular forms of possibly half--integral weight. More precisely, let $\mathrm{Mp}_2(\mathbb{Z})$ denote the metaplectic double cover of $\mathrm{SL}_2(\mathbb{Z})$ and let $\rho_{\mathcal{V}_\mathbb{Z}}$ be the Weil representation of $\mathrm{Mp}_2(\mathbb{Z})$ on the group algebra $\mathbb{C}[\mathcal{V}^\vee_\mathbb{Z}/\mathcal{V}_\mathbb{Z}]$, which has a standard basis $e^\mu$ indexed by $\mu \in \mathcal{V}^\vee_\mathbb{Z}/\mathcal{V}_\mathbb{Z}$. The work of Kudla and Millson \cite{KudlaMillson3} implies that the generating series
\begin{equation} \label{eq:gen_series}
Z_\mathbb{V}^+(\tau) = \sum_\mu Z^+_\mathbb{V}(\tau)_\mu \cdot e^\mu
\end{equation}
is a $\rho_{\mathcal{V}_\mathbb{Z}}$--valued modular form of weight $\mathrm{rk}(\mathcal{V}_\mathbb{Z})/2$. Their proof proceeds by constructing first certain theta series
\begin{equation}
\Theta_{\mathbb{V}}(\tau)_\mu \in \Omega^{1,1}(S), \quad \mu \in \mathcal{V}^\vee_\mathbb{Z}/\mathcal{V}_\mathbb{Z},
\end{equation}
depending on $\tau \in \mathbb{H}$, that transform like non--holomorphic modular forms. When $S$ is compact one can consider the integral $\smallint_S \Theta_\mathbb{V}(\tau)_\mu$, which inherits the transformation properties of $\Theta_{\mathbb{V}}(\tau)_\mu$, and so the modularity of $Z_\mathbb{V}^+(\tau)$ follows from the identity (cf. \cite[Theorem 2]{KudlaMillson3})
\begin{equation}
Z_\mathbb{V}^+(\tau)_\mu = \int_S \Theta_\mathbb{V}(\tau)_\mu.
\end{equation}

The goal of this paper is to generalize these results to the setting of \eqref{intro:eq_1} with $S$ non-compact. In this case the differential forms $\Theta_\mathbb{V}(\tau)_\mu$ often have singularities around the points in $\overline{S} \backslash S$. Our first result is the theorem below showing that these are always mild enough that $\Theta_\mathbb{V}(\tau)_\mu$ is integrable on $S$; note that we do not impose Hypothesis \ref{hypothesis:hyp_on_V}. 

\begin{theorem} \label{thm:intro_1}
Let $S$ be a smooth complex algebraic curve and  $\mathbb{V} \to S$ be a $\mathbb{Z}$--PVHS over $S$ of weight two with $h^{2,0}=1$ such that the action of $\pi_1(S,s)$ on $\mathcal{V}^\vee_{\mathbb{Z}}/\mathcal{V}_\mathbb{Z}$ is trivial. Then the integral
$$
Z_\mathbb{V}(\tau)_\mu = \int_S \Theta_\mathbb{V}(\tau)_\mu
$$
converges for every $\mu \in \mathcal{V}^\vee_\mathbb{Z}/\mathcal{V}_\mathbb{Z}$ and the expression
$$
Z_\mathbb{V}(\tau) = \sum_{\mu \in \mathcal{V}^\vee_\mathbb{Z}/\mathcal{V}_\mathbb{Z}} Z_\mathbb{V}(\tau)_\mu \cdot e^\mu
$$
defines a (possibly non-holomorphic) modular form of weight $\mathrm{rk}(\mathcal{V}_\mathbb{Z})/2$ valued in $\rho_{\mathcal{V}_\mathbb{Z}}$.
\end{theorem}

Our second main result gives the precise relation between the non-holo\-morphic modular form $Z_\mathbb{V}(\tau)$ and the generating series $Z_\mathbb{V}^+(\tau)$: one can write
$$
Z_\mathbb{V}(\tau) - Z_\mathbb{V}^+(\tau) = \sum_{P \in \overline{S} \, \backslash \, S} Z_{\mathbb{V},P}^-(\tau),
$$
where the term $Z_{\mathbb{V},P}^-(\tau)$ indexed by a given point $P \in \overline{S} \, \backslash \, S$ is determined explicitly from the polarized mixed Hodge structure defined by the degeneration of $\mathbb{V}$ at $P$. More precisely, let us write $V_\mathbb{Z}$ for the space of global multivalued sections of $\mathcal{V}_\mathbb{Z}$, that is, the space of global sections of the pullback of $\mathcal{V}_\mathbb{Z}$ to a universal cover of $S$. The pair $(V_\mathbb{Z},Q)$ is an even lattice of signature $(h^{1,1},2)$. A point $P \in \overline{S} \backslash S$ determines then an endomorphism $N(P)$ of $V_\mathbb{Z} \otimes \mathbb{Q}$ (the local monodromy logarithm) and an ascending filtration $W(P)_\bullet$ of $V_\mathbb{Z} \otimes \mathbb{Q}$ (the shifted weight filtration) such that the quotients
$$
\mathrm{Gr}_k^{W(P)} V_\mathbb{Z} := (W(P)_k \cap V_\mathbb{Z}) / (W(P)_{k-1} \cap V_\mathbb{Z})
$$
are free abelian groups of finite rank. The pair $(Q,N(P))$ determine bilinear forms $Q_k$ on $\mathrm{Gr}_k^{W(P)} V_\mathbb{Z}$ that define a structure of positive definite even lattice on $\mathrm{Gr}_4^{W(P)} V_\mathbb{Z}$ and on a certain sublattice $\mathrm{Gr}_{2,\mathrm{prim}}^{W(P)}V_\mathbb{Z} \subseteq \mathrm{Gr}_2^{W(P)}V_\mathbb{Z}$; the elements of $\mathrm{Gr}_{2,\mathrm{prim}}^{W(P)}V_\mathbb{Z}$ can be thought of as classes that become Hodge ``at infinity''.

Associated with these data are positive integers
$
r_k(V_\mathbb{Z},N(P))$ ($k=1,2$),  $\deg(Q_3)$ and $\mathrm{Vol}(\mathrm{Gr}_4^{W(P)} V_\mathbb{Z})$
as well as holomorphic theta series
$$
\Theta_{\mathrm{Gr}_{2,\mathrm{prim}}^{W(P)}V_\mathbb{Z}}(\tau), \quad \Theta_{\mathrm{Gr}_4^{W(P)}V_\mathbb{Z}}(\tau),
$$
valued in finite--dimensional representations $\rho_{\mathrm{Gr}_{2,\mathrm{prim}}
^{W(P)}V_\mathbb{Z}}$ and $\rho_{\mathrm{Gr}_4^{W(P)} V_\mathbb{Z}}$. The representations $\rho_{\mathrm{Gr}_2
^{W(P)}V_\mathbb{Z}}$ and $\rho_{\mathrm{Gr}_{2,\mathrm{prim}}
^{W(P)}V_\mathbb{Z}} \otimes \rho_{\mathrm{Gr}_4^{W(P)} V_\mathbb{Z}}$ admit intertwining maps to $\rho_{\mathcal{V}_\mathbb{Z}}$ that we denote by $\iota$.

\begin{theorem}\label{thm:main_thm_intro}
Assume that $\mathbb{V}$ satisfies \ref{hypothesis:hyp_on_V}. For $P \in \overline{S} \ \backslash \ S$, denote by $N(P)$ the local monodromy logarithm and by $W(P)_\bullet$ the corresponding (shifted) weight filtration and define
$$
Z_{\mathbb{V},P}^-(\tau) = \frac{r_1(V_\mathbb{Z},N(P))}{\deg(Q_3)} \frac{1}{4 \pi \mathrm{Im}(\tau)} \iota(\Theta_{\mathrm{Gr}_2^{W(P)} V_\mathbb{Z}})
$$
if $N(P)^2=0$ and
\begin{equation*}
\begin{split}
Z^-_{\mathbb{V},P}(\tau) = &  \frac{r_2(V_\mathbb{Z},N(P))}{\mathrm{Vol}(\mathrm{Gr}_4^{W(P)} V_\mathbb{Z})^{1/2}} \frac{1}{4\pi i} 
 \int_{-\overline{\tau}}^{i\infty} \frac{\iota(\Theta_{\mathrm{Gr}_4^{W(P)} V_\mathbb{Z}}(z) \otimes \Theta_{\mathrm{Gr}_{2,\mathrm{prim}}^{W(P)} V_\mathbb{Z}}(\tau))}{((z+\tau)/i)^{3/2}} dz
\end{split}
\end{equation*}
if $N(P)^2 \neq 0$. Then
$$
Z_\mathbb{V}(\tau) = Z^+_\mathbb{V}(\tau) + \sum_{P \in \overline{S} \, \backslash \, S} Z^-_{\mathbb{V},P}(\tau).
$$
In particular, the right hand side is 
 a $\rho_{\mathcal{V}_\mathbb{Z}}$-valued modular form of weight $rk(\mathcal{V}_\mathbb{Z})/2$. 
\end{theorem}

\begin{remark}\label{remark:hyp_intro}
Hypothesis \ref{hypothesis:hyp_on_V} is very mild: let $\mathbb{V}$ be an arbitrary $\mathbb{Z}$--PVHS of weight two with $h^{2,0}=1$ such that $\mathcal{V}_{\mathbb{Q}s}$ is a simple $\pi_1(S,s)$--module (recall that the category of polarizable $\mathbb{Q}$--VHS over $S$ is semisimple \cite[Cor. 13]{PetersSteenbrink}). Since the monodromy of a $\mathbb{Z}$-PVHS on the punctured disk is quasi--unipotent \cite[Lemma (4.5)]{Schmid}, one can guarantee that \ref{hypothesis:hyp_on_V} holds by picking a finite index even sublattice  $\mathcal{V}_\mathbb{Z}$--module and passing to an appropriate finite cover of $S$ so that the local monodromies around $\overline{S} \backslash S$ are unipotent (note that $\mathbb{V}$ extends across any point in $\overline{S} \backslash S$ with trivial monodromy by \cite[Cor. (4.11)]{Schmid}). 
\end{remark}

\begin{remark}
If $\mathbb{V}$  is the PVHS associated with a polarized family $\mathcal{X}$ of K3 surfaces parametrized by $S$, we can (after replacing $S$ by a finite cover if necessary) interpret the non-holomorphic terms $Z_{\mathbb{V},P}^-$ in terms of Hodge classes in the irreducible components of the singular fibers of a semistable model of $\mathcal{X}$. This follows from the Clemens--Schmid exact sequence (see e.g. \cite{Morrison}). A similar remark applies if $\mathbb{V}_\mathbb{Q} = (\mathcal{V}_\mathbb{Q},Q,\mathcal{F}^\bullet)$ appears as a direct summand of the PVHS naturally attached to a polarized family of non--singular projective surfaces parametrized by $S$.
\end{remark}

\begin{remark}
Let 
$$
\mathbb{G}_2(q) = -\frac{1}{24}+\sum_{n \geq 1} \sigma_1(n) q^n, \quad \sigma_1(n):=\sum_{d|n} d.
$$
Then $\mathbb{G}_2^*(\tau):=\mathbb{G}_2(q)+(8\pi y)^{-1}$ is a (non--holomorphic) modular form of weight $2$ for the full modular group $\mathrm{SL}_2(\mathbb{Z})$ (see \cite[eqs. (17) and (21)]{Zagier123}); moreover, the operator
$$
f \mapsto q \frac{d}{dq}f + 2k \mathbb{G}_2(q) \cdot f
$$
sends modular forms of weight $k$ to modular forms of weight $k+2$ 

\cite[\textsection 5.1]{Zagier123}. Thus in Theorem \ref{thm:main_thm_intro} we still obtain a $\rho_{\mathcal{V}_\mathbb{Z}}$-valued modular form if we replace any term $Z_{\mathbb{V},P}^-(\tau)$ associated with a point $P$ of type II with any of the holomorphic expressions
$$
- 2\frac{r_1(V_\mathbb{Z},N)}{\deg(Q_3)} \mathbb{G}_2(q) \Theta_{\mathrm{Gr}_2^W V_\mathbb{Z}}(\tau)
$$
or
$$
\frac{r_1(V_\mathbb{Z},N)}{\deg(Q_3)}  \frac{2}{\mathrm{rk}(\mathcal{V}_\mathbb{Z})-4} \cdot q \frac{d}{dq}\Theta_{\mathrm{Gr}_2^W V_\mathbb{Z}}(\tau).
$$
Similarly, the Eichler integral of $\Theta_{\mathrm{Gr}_4^{W(P)} V_\mathbb{Z}}$ appearing in the contribution $Z^-_{\mathbb{V},P}$ of a type III degeneration is the non--holomorphic part of a weight $3/2$--Eisenstein series defined in \cite{Zagier} and so we may replace any term $Z_{\mathbb{V},P}^-(\tau)$ associated with a degeneration of type III with a holomorphic expression involving the holomorphic part of Zagier's Eisenstein series.
\end{remark}

\subsection{Relation with other works}

In the setting of the PVHS parametrized by Shimura varieties of orthogonal or unitary type, several recent works address the explicit computation of correction terms to the generating series of special divisors coming from an appropriate toroidal compactification: the case of modular curves was treated by Funke in \cite{Funke} and related computations for toroidal compactifications of unitary Shimura varieties where only type II degenerations appear are in \cite{KudlaHoward1}. Recently Bruinier and Zemel \cite{BruinierZemel} have proved a result for special divisors on orthogonal Shimura varieties that is similar to the modularity statement in Theorem \ref{thm:main_thm_intro}. Their proof involves studying the asymptotic behaviour of Borcherds lifts along components of a toroidal compactification. A different proof (and refinement) using more geometric methods has been very recently obtained by Engel, Greer and Tayou in  \cite{https://doi.org/10.48550/arxiv.2301.05982}. Our paper contributes the explicit description of boundary terms in terms of limiting mixed Hodge structures and, like \cite{https://doi.org/10.48550/arxiv.2301.05982}, it also clarifies the rationality properties of coefficients along type III contributions.

\subsection{Strategy of proof} In contrast to the above works, this paper does not rely on the theory of toroidal compactifications of Shimura varieties. Instead, our proofs are analytic in nature and use Schmid's results on degenerations of Hodge structure, particularly his  characterization of the weight filtration by growth of the Hodge norm and his nilpotent and $\mathrm{SL}_2$--orbit theorems. For a fixed point $P \in \overline{S} \backslash S$ these results imply that in a neighbourhood of $P$ the variation $\mathbb{V}$ is well-approximated by a special type of nilpotent orbit $\tilde{\mathbb{V}}^{\mathrm{nilp}}$. We prove Theorem \ref{thm:intro_1} by showing that $\Theta_\mathbb{V}(\tau)_\mu - \Theta_{\tilde{\mathbb{V}}^{\mathrm{nilp}}}(\tau)_\mu$ and $\Theta_{\tilde{\mathbb{V}}^{\mathrm{nilp}}}(\tau)_\mu$ are both locally integrable around $P$. The proof of Theorem \ref{thm:main_thm_intro} reduces to the computation of the residue of certain canonical Green functions for $\mathrm{NL}_\mathbb{V}(m)_\mu$ along $P \in \overline{S} \backslash S$. We show that the residue agrees with that of the corresponding Green function for $\tilde{\mathbb{V}}^{\mathrm{nilp}}$, which can be computed exactly thanks to the explicit nature of nilpotent orbits.

One advantage of our methods over the use of toroidal compactifications that originally motivated the author's interest is that the theta series $\Theta_\mathbb{V}(\tau)$ and Schmid's theorems are available and in definitive form for arbitrary PVHS of weight two over the complement of a normal crossing divisor in a higher--dimensional base. Schmid's several variables $\mathrm{SL}_2$-orbit theorem \cite{CKS} approximates a degeneration of Hodge structure in $n$ variables by an $(n-1)$--dimensional pencil of one--dimensional nilpotent orbits; in particular, to use his theorem to study the boundary behaviour of $\Theta_\mathbb{V}(\tau)$ one must first understand $\Theta_{\mathbb{V}}(\tau)$ along one--variable degenerations. In future work, the author intends to develop the methods in this paper to address the conjectural (mock) modularity of generating series of Noether--Lefschetz loci for certain VHS with $h^{2,0}>1$ that are not naturally parametrized by Shimura varieties; for a particularly interesting example see \cite{https://doi.org/10.48550/arxiv.2007.03037}.

\subsection{Acknowledgements} The author would like to thank Nicolas Bergeron and Keerthi Madapusi for their interest and feedback, Salim Tayou for conversations regarding boundary behaviour of special divisors  and Richard Thomas for his questions regarding modularity of Noether--Lefschetz loci for variations with $h^{2,0}>1$.

\section{Weight two PVHS over a complex algebraic curve}  In this section we briefly review some relevant facts on variations of Hodge structure over a one--dimensional base. We will only consider integral polarized variations of weight two with Hodge numbers $(1,n,1)$ for some positive integer $n$.

Throughout the paper we fix a connected compact Riemann surface $\overline{S}$ and a finite collection of points $P_1,\ldots, P_r \in \overline{S}$, and write 
$$
S=\overline{S}-\{P_1,\ldots,P_r\}.
$$
Sections \ref{subsection:VHS_definitions} and \ref{subsection:local_monodromy} collect definitions and known facts on degenerations of Hodge structure and approximation by nilpotent orbits. We refer the reader to \cite{Schmid, CKS} for proofs; our exposition follows closely Hain's account \cite{Hain}. Sections \ref{subsection:type_ii_degenerations} and \ref{subsection:type_iii_degenerations} compute some nilpotent orbits explicitly. The formulas in these Sections will be used later to understand the behaviour of Kudla--Millson forms around the points $P \in \overline{S} \backslash S$.

\subsection{Definitions}\label{subsection:VHS_definitions} Consider an integral polarized variation of Hodge structure ($\mathbb{Z}$-PVHS) $\mathbb{V} \to S$ of weight two over $S$.  Here $\mathbb{V}$ is a triple $(\mathcal{V}_\mathbb{Z},Q,\mathcal{F}^\bullet)$ consisting of:
\begin{itemize}
\item a local system $\mathcal{V}_\mathbb{Z}$ of free and finite rank abelian groups over $S$,
\item a (locally constant) non-degenerate symmetric bilinear form
$$
Q: \mathcal{V}_\mathbb{Z} \times \mathcal{V}_\mathbb{Z} \to \underline{\mathbb{Z}},
$$
\item a descending filtration 
$$
\mathcal{V} = \mathcal{F}^0 \supset \mathcal{F}^1 \supset \mathcal{F}^2
$$
of the flat complex vector bundle $\mathcal{V}:=\mathcal{V}_\mathbb{Z} \otimes \mathcal{O}_S$ by holomorphic vector bundles $\mathcal{F}^k$ that are locally direct summands,
\end{itemize}
such that the fiber $\mathbb{V}_s = ((\mathcal{V}_\mathbb{Z})_s,Q_s,\mathcal{F}^\bullet_s)$ over any $s \in S$ defines a polarized Hodge structure of weight two. We assume that
$$
h^{2,0} = \mathrm{rk} \mathcal{F}^2 = 1,
$$
so that $\mathcal{F}^2$ is a holomorphic line bundle over $S$ that we denote by $\mathcal{L}$. 

For each $s \in S$ we have the Hodge decomposition
\begin{equation}
\mathcal{V}_{s} = \oplus_{p+q=2} \mathcal{V}_s^{p,q}.
\end{equation}
Let $C_s \in \mathrm{End}(\mathcal{V}_{s})$ denote the Weil operator: it acts as the identity on $\mathcal{V}^{1,1}_s$ and as $-1$ on $\mathcal{V}_s^{2,0} \oplus \mathcal{V}_s^{0,2}$. The polarization induces a hermitian metric $\|\cdot\|_{\mathcal{V},s}$ on $\mathcal{V}_s$. This is the Hodge metric, defined by
\begin{equation}
\|v\|_{\mathcal{V},s}^2 = Q(C_s v, \overline{v}), \quad v \in \mathcal{V}_s.
\end{equation}
When only one PVHS is being considered, we will suppress $\mathcal{V}$ from the notation and denote the Hodge norm of $v \in \mathcal{V}_s$ simply by $\|v\|_s^2$.

The polarization $Q$ induces an isomorphism $\mathcal{V} \simeq \mathcal{V}^\vee$ sending a vector $v \in \mathcal{V}_s$ to the linear functional $v' \mapsto Q(v',v)$. Composing this isomorphism with the canonical surjection $\mathcal{V}^\vee \to \mathcal{L}^\vee$ dual to the inclusion $\mathcal{L} \subset \mathcal{V}$ gives an isomorphism $\mathcal{V}/\mathcal{F}^1 \simeq \mathcal{L}^\vee$. In particular, to a section $v \in \mathrm{H}^0(U,\mathcal{V})$ defined over $U \subset S$ corresponds a section $s_v \in \mathrm{H}^0(U,\mathcal{L}^\vee)$. It will be convenient to define
\begin{equation}
h(s_v) = 2\|s_v\|^2_{\mathcal{L}^\vee}.
\end{equation}
Writing $v_z=\Sigma v^{p,q}_z$ for the Hodge decomposition of $v_z \in \mathcal{V}_z$, the value of $h(s_v)$ at $z \in U$ is given by
\begin{equation}
h(s_v)_z = 2\|v_z^{2,0}\|^2_{\mathcal{V}}=  -2Q(v^{2,0}_z,\overline{v^{2,0}_z}).
\end{equation}

%Local monodromy and mixed Hodge structure
\subsection{Local monodromy and limit mixed Hodge structure} \label{subsection:local_monodromy}

The asymptotic behaviour of $\mathbb{V} \to S$ around each of the points $P \in \overline{S} \backslash S$ can be described precisely in terms of limit mixed Hodge structures using the results of Schmid in \cite{Schmid}. We briefly recall the results that we will use.

Let $\Delta=\{t \in \mathbb{C} \ | \ |t|<1\}$ denote the open unit disk in $\mathbb{C}$ and let $\Delta^*=\Delta - \{0\}$ be the punctured open unit disk. Consider a polarized variation of Hodge structure $\mathbb{V}=(\mathcal{V}_\mathbb{Z},Q,\mathcal{F}^\bullet)$ of weight two over $\Delta^*$ with $h^{2,0}=1$. 

\subsubsection{Local monodromy and weight filtration} For $s \in \Delta^*$, let $\mathcal{V}_{\mathbb{Z}s}$ be the fiber of $\mathcal{V}_\mathbb{Z}$ over $s$. This fiber carries an action of the fundamental group $\pi_1(\Delta^*,s)$. We denote by
$$
T \in \mathrm{O}(\mathcal{V}_{\mathbb{Z}s},Q) \subset \mathrm{GL}(\mathcal{V}_{\mathbb{Z} s})
$$
the monodromy operator, that is, the image in $ \mathrm{GL}(\mathcal{V}_{\mathbb{Z} s})$ of the generator of $\pi_1(\Delta^*,s)$ defined by the loop $t \mapsto s e^{2\pi it}$ for $t \in [0,1]$. Then (\cite[Thm. 6.1]{Schmid}) $T$ is quasi-unipotent, i.e. there exist positive integers $e$ and $M$ such that
$$
(T^e-1)^M=0.
$$
Passing to a cover of $\Delta^*$ of degree $e$, we may assume that $e=1$. Moreover, we can take $M \leq 3$ (\cite[Thm. 6.1]{Schmid}) and, if $T=1$, then the polarized variation of Hodge structure $(\mathcal{V}_\mathbb{Z},Q,\mathcal{F}^\bullet)$ can be extended to the open unit disk $\Delta$ (\cite[Cor. 4.11]{Schmid}). 

Let
$$
N=\log T =\sum_{k=1}^{M-1} (-1)^{k-1} \frac{(T-1)^k}{k}.
$$
Then $N^3=0$. If $T \neq 1$, this leaves two possibilities:  
\begin{itemize}
\item $N^2 = 0$ (Type II degeneration), and
\item $N^2 \neq 0$ (Type III degeneration).
\end{itemize}

To the nilpotent endomorphism $N$ of $\mathcal{V}_{\mathbb{Q} s}=\mathcal{V}_{\mathbb{Z}s} \otimes \mathbb{Q}$ corresponds an increasing filtration $W_\bullet(N)$ of $\mathcal{V}_{\mathbb{Q}s}$ by $\mathbb{Q}$-vector spaces called the weight filtration. It is the unique filtration
$$
\cdots \subseteq W_{k}(N) \subseteq W_{k+1}(N) \subseteq \cdots
$$
of $\mathcal{V}_{\mathbb{Z}s} \otimes \mathbb{Q}$ satisfying
$
N \cdot W_{k}(N) \subseteq W_{k-2}(N)
$
and such that 
$$
N^{k} : W_{k}(N)/W_{k-1}(N) \to W_{-k}(N)/W_{-k-1}(N)
$$
is an isomorphism. We write 
$$
W_k=W_{k-2}(N)
$$
 for the (shifted) weight filtration of $N$. Since $N^3=0$, this filtration satisfies $W_{-1}=0$ and $W_4=\mathcal{V}_{\mathbb{Z}s} \otimes \mathbb{Q}$. Abusing notation, we denote by $W_\bullet$ the corresponding filtration of the local system $\mathcal{V}_\mathbb{Q}$: 
$$
0 = W_{-1} \subseteq W_0 \subseteq W_1 \subseteq W_2 \subseteq W_3 \subseteq W_4 = \mathcal{V}_\mathbb{Q}.
$$
Since $N=\log T \in \mathfrak{so}(\mathcal{V}_{\mathbb{Q}s},Q)$, we have
\begin{equation} \label{eq:polarization_and_quadratic_form}
Q(Nv,w) = - Q(v,Nw).
\end{equation}
It follows that the weight filtration $W_\bullet$ is self-dual: writing $W_k^\perp$ for the orthogonal complement of $W_k$ under $Q$, we have
$$
W_k^\perp = W_{3-k}.
$$
Moreover, the quotients
$$
\mathrm{Gr}_k^W \mathcal{V}_\mathbb{Q} := W_k/W_{k-1}
$$
carry canonical bilinear forms $Q_k$ defined as follows: if $k \geq 2$ and $\tilde{v},\tilde{w} \in \mathrm{Gr}_k^W \mathcal{V}_\mathbb{Q}$ are represented by $v,w \in W_k$, we define
\begin{equation} \label{eq:Q_k_definition}
Q_k(\tilde{v},\tilde{w}) = Q(v,N^{k-2}w).
\end{equation}
If $k <2$, then we define $Q_k$ so that the isomorphism $N^{2-k}: \mathrm{Gr}_{4-k}^W \mathcal{V}_\mathbb{Q} \to \mathrm{Gr}_{k}^W \mathcal{V}_\mathbb{Q}$ is an isometry (\cite[Lemma 6.4]{Schmid}). 
%Canonical extension and limit mixed Hodge structure
\subsubsection{Canonical extension and limit MHS} \label{section:canonical extension} The vector bundle $\mathcal{V} = \mathcal{V}_{\mathbb{Z}} \otimes \mathcal{O}_{\Delta^*}$ carries a canonical flat connection $\nabla$. Let us fix a flat multi-valued basis $v_1,\ldots,v_{n+2}$ of $\mathcal{V}_\mathbb{Q}$. We assume that this basis is chosen so as to provide a splitting of the weight filtration: that is,
$$
W_k = \langle v_1,\ldots, v_{\mathrm{dim} W_k} \rangle
$$
for $0 \leq k \leq 4$. Define a new basis $(\tilde{v}_i)$ of $\mathcal{V}$ by
$$
\tilde{v}_i(q) = \exp\left(\frac{i}{2\pi} \log t \cdot N \right)v_i(q)
$$
Note that parallel translation of along a positively oriented circle changes $v_i$ to $Tv_i$ and $\exp\left(\frac{i}{2\pi} \log t \cdot N \right)$ to
$$
\exp\left(\frac{i}{2\pi} (\log t + 2\pi i) \cdot N \right) = \exp\left(\frac{i}{2\pi} \log t \cdot N \right) \cdot T^{-1}.
$$
It follows that the basis $(\tilde{v}_i)$ is single-valued and so it defines a trivialization $\mathcal{O}_{\Delta^*}^{n+2} \simeq \mathcal{V}$ over $\Delta^*$. The canonical extension $\tilde{\mathcal{V}}$ of $\mathcal{V}$ is defined to be the extension of  $\mathcal{V}$ as a constant bundle over $\Delta$, that is, the extension corresponding to $\mathcal{O}^{n+2}_\Delta$ under the above isomorphism. We denote by $\tilde{\mathcal{V}}_0$ its fiber over $0 \in \Delta$. By \eqref{eq:polarization_and_quadratic_form}, we have
$$
Q(\tilde{v}_i(q),\tilde{v}_j(q)) = Q(v_i,v_j)
$$
and so the polarization $Q$ extends to a symmetric bilinear form on the fiber $\tilde{\mathcal{V}}_0$ that we still denote by $Q$.

Schmid's nilpotent orbit theorem \cite[Thm. 4.9]{Schmid} states that the Hodge filtration $\mathcal{F}^\bullet$ extends to a filtration $\tilde{\mathcal{F}}^\bullet$ of the canonical extension $\tilde{\mathcal{V}}$ by locally direct factors. We write
$
F_\mathrm{lim}^\bullet = \tilde{\mathcal{F}}_0
$
for the limit Hodge filtration, i.e. the corresponding filtration of $\tilde{\mathcal{V}}_0$. Then we have
\begin{equation} \label{eq:general_properties_limit_Hodge_filtration}
\begin{split}
Q(F^1_{\mathrm{lim}},F^2_{\mathrm{lim}}) &= 0 \\
N \cdot F^2_{\mathrm{lim}} & \subseteq F^1_{\mathrm{lim}}.
\end{split}
\end{equation}
Moreover, the basis $\tilde{v}_1(0),\ldots,\tilde{v}_{n+2}(0)$ defines a $\mathbb{Z}$-structure on the fiber $\tilde{\mathcal{V}}_0$ that we denote by $V_\mathbb{Z}$, and the weight and limit Hodge filtrations
$$
(W_\bullet, F^{\bullet}_{\mathrm{lim}})
$$
define a mixed $\mathbb{Q}$-Hodge structure on $V:=V_\mathbb{Z} \otimes \mathbb{Q}$.\mynote{I should check (and possibly mention?) that $V_\mathbb{C}$ is naturally identified with the space of flat multi--valued sections of $\mathcal{V}$, that is flat global sections of $\pi^*\mathcal{V}$ where $\pi: \mathbb{H} \to \Delta^*$ is the uniformising map. This is how Sabbah--Schnell phrase things.} 
Together with the action of $N$ and the extension of $Q$ to $V$, these filtrations define a polarized mixed $\mathbb{Q}$-Hodge structure. More precisely, %(\cite[Appendix to Lecture 10]{GriffithsLectures}),
we have (cf. \cite[Def. (2.26)]{CKS})
\begin{enumerate}
\item[i)] $(V,W_\bullet, F^{\bullet}_{\mathrm{lim}})$ is a mixed $\mathbb{Q}$-Hodge structure satisfying \eqref{eq:general_properties_limit_Hodge_filtration}.
\item[ii)] %The endomorphism $N \in \mathrm{End}(V)$ is a map of mixed Hodge structures of type $(-1,-1)$.
$W_\bullet=W_\bullet(N)[-2]$.
\item[iii)] Define 
$$
P_2 = \ker (N: \mathrm{Gr}_2^W V \to  \mathrm{Gr}_0^W V) \subseteq \mathrm{Gr}_2^W V
$$
and, for $k \neq 2$, set $P_k = \mathrm{Gr}_k^W V$. Then
\begin{equation}
\mathrm{Gr}_2^W V = P_2 \oplus N P_4
\end{equation}
and the restriction of the bilinear form $Q_k$ in \eqref{eq:Q_k_definition} to $P_k$ defines a polarized $\mathbb{Q}$-Hodge structure of weight $k$. 
\end{enumerate}

 We write
$$
h^{p,q}_{\mathrm{lim}} = \Gr^{p}_{F_{\mathrm{lim}}} \Gr_{p+q}^W V
$$
for the Hodge numbers of the limiting mixed Hodge structure.

%Nilpotent orbit
\subsubsection{Nilpotent orbit} Using the isomorphism $\mathcal{O}^{n+2}_\Delta \simeq \tilde{\mathcal{V}}$, we extend the filtration $F_{\mathrm{lim}}^{\bullet}$ of $\tilde{\mathcal{V}}_0$ to a filtration, still denoted by $F_{\mathrm{lim}}^\bullet$, of $\tilde{\mathcal{V}}$. The corresponding nilpotent orbit is then given by the filtration \mynote{This can be a little confusing. It becomes easier to understand using Deligne's connection
$$
\nabla^c = \nabla - \frac{i}{2\pi} N \frac{dt}{t}.
$$
The point is that $\nabla^c$ is a flat connection without monodromy, and the canonical extension is just the extension of $\tilde{\mathcal{V}}$ to a $\nabla^c$-flat vector bundle over $\Delta$. The expression ``Extending the filtration $F_\mathrm{lim}^\bullet \subset V$ to a filtration, still denoted by ...'' means just that $F_\mathrm{lim}^\bullet$ is a filtration of a fiber of $\tilde{\mathcal{V}}$, and so can be extended in a unique way to a filtration of $\tilde{\mathcal{V}}$ by $\nabla^c$-flat subbundles (this is what Deligne \cite{DeligneLocalBehaviour} calls ``constant''). Since the $F_\mathrm{lim}^\bullet$ are $\nabla^c$-flat, the $\mathcal{F}_{\mathrm{nilp}}^\bullet$ defined above are $\nabla$-flat, i.e. they are preserved by the monodromy and so they are well--defined bundles over $\Delta^*$.
}
\begin{equation}
\mathcal{F}_{\mathrm{nilp}}^{\bullet} := \exp \left(\frac{1}{2\pi i} \log t \cdot N \right) F^{\bullet}_{\mathrm{lim}} \subset \mathcal{V}.
\end{equation}
Using the uniformisation of $\Delta^*$ via the exponential map
$$
\pi: \mathbb{H} \to \Delta^*, \quad z \mapsto t:=e^{2\pi i z}.
$$
we may write
\begin{equation}
 \mathcal{F}_{\mathrm{nilp}}^k = e^{zN} F_\mathrm{lim}^k.
\end{equation}

In a small enough neighbourhood of $0$, the triple 
\begin{equation}
\mathbb{V}^{\mathrm{nilp}} := (\underline{V_\mathbb{Z}},Q,\mathcal{F}^{\bullet}_{\mathrm{nilp}})
\end{equation}
defines a variation of Hodge structures of weight two polarized by $Q$ with the same Hodge numbers as $\mathbb{V}$.
 
The variation $\mathbb{V}^{\mathrm{nilp}}$ approximates $\mathbb{V}$ in the following sense. Let us denote by $\mathbb{D}$ the period domain parametrising Hodge structures on $V_\mathbb{R}$ polarised by $Q$ with $h^{2,0}=1$: it is the hermitian symmetric domain attached to the Lie group 
$$
G_\mathbb{R} := \mathrm{Aut}(V_\mathbb{R},Q).
$$
We write $\mathbb{D}^\vee$ for the compact dual of $\mathbb{D}$; it contains $\mathbb{D}$ and is a homogeneous complex manifold for $G_\mathbb{C}$. The pullback $\pi^* \mathbb{V}$ to $\mathbb{H}$ of the PVHS $\mathbb{V}$ defines a holomorphic map
$$
\Phi_\mathbb{V}: \mathbb{H} \to \mathbb{D}
$$
satisfying $\Phi_\mathbb{V}(z+1) = e^N \Phi_\mathbb{V}(z)$. Since $e^{zN}$ belongs to $G_\mathbb{C}$, we have $e^{-zN}\Phi_\mathbb{V}(z) \in \mathbb{D}^\vee$ for every $z \in \mathbb{H}$. This gives a holomorphic map
$$
\tilde{\Psi}_\mathbb{V}: \mathbb{H} \to \mathbb{D}^\vee, \qquad \tilde{\Psi}_\mathbb{V}(z)= e^{-zN}\Phi_\mathbb{V}(z)
$$
that is invariant under $z \mapsto z+1$, and so $\tilde{\Psi}_{\mathbb{V}}(z)$ induces a holomorphic map
$$
\Psi_\mathbb{V} : \Delta^* \to \mathbb{D}^\vee, \quad \Psi_\mathbb{V}(e^{2\pi i z}) := e^{-zN}\Phi_\mathbb{V}(z).
$$
Schmid's nilpotent orbit theorem states that $\Psi_\mathbb{V}$ extends to a holomorphic map defined on $\Delta$ \cite[\textsection 2.3]{DeligneLocalBehaviour}. For the nilpotent orbit $\mathbb{V}^{\mathrm{nilp}}$ this map is constant with
$$
\Psi_{\mathbb{V}^{\mathrm{nilp}}}(t) = \Psi_\mathbb{V}(0) = F^\bullet_{\mathrm{lim}}.
$$
The map $\Psi_{\mathbb{V}}$ can be written as
$$
\Psi_\mathbb{V}(t) = \psi_{\mathbb{V}}(t) \cdot F^\bullet_{\mathrm{lim}}
$$
with 
$$
\psi_\mathbb{V}: \Delta \to G_\mathbb{C}
$$
a holomorphic map satisfying $\psi_{\mathbb{V}}(0)=1$ (for a canonical choice of $\psi_\mathbb{V}$, see \cite[(2.5)]{CattaniKaplanVHS}). Thus we may write 
$$
\Phi_\mathbb{V}(z) = e^{zN} \Psi_{\mathbb{V}}(e^{2\pi i z}) = e^{zN} \psi_\mathbb{V}(t) F^\bullet_{\mathrm{lim}} = e^{zN} \psi_\mathbb{V}(t) e^{-zN}\Phi_{\mathbb{V}^\mathrm{nilp}}(z).
$$
Equivalently, the Hodge filtrations of $\mathbb{V}$ and $\mathbb{V}^{\mathrm{nilp}}$ satisfy
\begin{equation} \label{eq:map_relating_V_and_nilp_orbit}
\mathcal{F}^\bullet_t = e^{zN}\psi_{\mathbb{V}}(t) e^{-zN} \mathcal{F}_{\mathrm{nilp},t}^\bullet, \qquad t \in \Delta^*.
\end{equation}
Given a norm $|\cdot|$ on $\mathrm{End}(V_\mathbb{C})$ and assuming that $|\mathrm{Re}(z)| \leq 1/2$, we have the trivial estimate
\begin{equation} \label{eq:dist_to_nilp_estimate}
|e^{zN} \psi_\mathbb{V}(t) e^{-zN}-1| = O(|t| (\log |t|)^{2k})
\end{equation}
for some positive integer $k$.

%Schmid's theorems and estimates
\subsubsection{Approximation and Hodge norm estimates} We also need Schmid's estimates for the Hodge norm: fix an angular sector
$$
U = U(t_0,\epsilon) = \{ t \in \Delta^* | 0 < \arg(t-t_0) < 2\pi-\epsilon \}
$$
of $\Delta^*$ and let $v \in \mathcal{V}_\mathbb{C}|_U$. If $v \in W_k - W_{k-1}$, then
$$
\|v\|^2_t \sim (-\log |t|)^{k-2},
$$
uniformly on $U$ (\cite[Thm. 6.6']{Schmid}). In \ref{subsection:Integrability_Theta''} we will state a  more precise version that also gives bounds for the derivatives of $\|v\|_t^2$.

\subsection{Type II degenerations} \label{subsection:type_ii_degenerations} We say that a degeneration is of type II if $N^2=0$ but $N$ is non--trivial. For this type of degeneration, the Hodge numbers of the limit mixed Hodge structure are
$$
h^{1,0}_{\mathrm{lim}}=h^{0,1}_{\mathrm{lim}}=h^{2,1}_{\mathrm{lim}}=h^{1,2}_{\mathrm{lim}}=1, \quad h^{1,1}_{\mathrm{lim}}=\mathrm{dim} V_\mathbb{C} - 4
$$
(all other Hodge numbers are zero) and the weight filtration is
\begin{align*}
W_0 &= 0 \\
W_1 &= \image N \\
W_2 &= \ker N \\
W_3 &=W_4=V.
\end{align*}

Below we compute explicitly the Hodge norms $\|v\|^2$ and Chern form $\Omega$ associated with the corresponding nilpotent orbit; these computations will allow us to derive explicit expressions for Kudla--Millson forms for such degenerations. Let us assume that the limit mixed Hodge structure is $\mathbb{R}$--split, i.e. that $V_\mathbb{R}$ is the direct sum of pure Hodge structures. This will suffice for our intended application.

\subsubsection{} Let $V_\mathbb{C} = \oplus_{a,b} I^{a,b}$ ($0 \leq a,b \leq 2)$ denote the canonical bigrading defined by Deligne \cite[(2.12)]{CKS}. Since we assume that $V$ is $\mathbb{R}$--split, this bigrading is simply given by
\begin{equation} \label{eq:type_II_Deligne_splitting}
I^{a,b} = F^a_{\mathrm{lim}} \cap \overline{F^b_\mathrm{lim}} \cap W_{a+b,\mathbb{C}}.
\end{equation}
Then we have
$$
W_{k,\mathbb{C}} = \oplus_{a+b \leq k} I^{a,b}, \quad F^p_\mathrm{lim} = \oplus_{a \geq p} I^{a,b}
$$
and $I^{b,a}=\overline{I^{a,b}}$. Let us define
$$
V_k = (\oplus_{a+b=k} I^{a,b}) \cap W_{k,\mathbb{R}},
$$
so that $V_\mathbb{R} =\oplus_{1 \leq k \leq 3} V_k$. Via the isomorphism
$$
V_k \simeq W_{k,\mathbb{R}}/W_{k-1,\mathbb{R}} = \mathrm{Gr}_k^W V_\mathbb{R}, 
$$
the Hodge filtration on $\mathrm{Gr}_k^W V_\mathbb{R}$ induces on $V_k$ a pure $\mathbb{R}$-Hodge structure of weight $k$ (with Hodge filtration $F^\bullet_{\mathrm{lim}} \cap V_{k,\mathbb{C}}$).

Since the form $Q(\cdot, N \cdot)$ polarizes the Hodge structure on $\mathrm{Gr}_3^W V_\mathbb{R} \simeq V_3$, we can find a vector $e^{2,1} \in I^{2,1}$ such that
$$
iQ(e^{2,1},N\overline{e^{2,1}}) =1.
$$
We fix such a vector and write $e^{1,2}=\overline{e^{2,1}} \in I^{1,2}$, $e^{1,0} = Ne^{2,1} \in I^{1,0}$ and $e^{0,1} = \overline{e^{1,0}} \in I^{0,1}$. Then $\{e^{2,1}, e^{1,2}\}$ is a basis for $V_3$ and $\{e^{1,0}, e^{0,1}\}$ is a basis for $V_1$. Using $Q(F^2_{\mathrm{lim}},F^1_{\mathrm{lim}})=0$ and $Q(W_1,W_2)=0$, one sees that $V_2$ is orthogonal to $V_1 \oplus V_3$ and that in the basis $\{e^{2,1},e^{1,2},e^{1,0},e^{0,1}\}$ the restriction of $Q$ to $V_1 \oplus V_3$ is given by the matrix
\begin{equation} \label{eq:Q_matrix_type_II}
\begin{pmatrix} 
0 & 0 & 0 & -i \\ 0 & 0 & i & 0 \\ 0 & i & 0 & 0 \\ -i & 0 & 0 & 0
\end{pmatrix}.
\end{equation}

\subsubsection{} We now consider the nilpotent orbit: for $z \in \mathbb{H}$, we write
$$
F_z^k= e^{zN}F^k_{\mathrm{lim}}.
$$
Since $N^2=0$ and $I^{1,1} \subset \ker N$, we have $e^{zN}=1+zN$ and hence
\begin{align*}
F_z^2 &= \langle e_z^{2,0} \rangle \\
F_z^1 &= e^{zN}(F^2_\mathrm{lim} \oplus I^{1,1})=F^2_z \oplus I^{1,1},
\end{align*}
where
\begin{align*}
e_z^{2,0} &:=e^{2,1}+ze^{1,0}.  \\
\end{align*}
The Hodge norm of $e_z^{2,0}$ can be computed explicitly using \eqref{eq:Q_matrix_type_II}:
\begin{equation} \label{eq:e_z20_Hodge_norm_type_II}
\begin{split}
    \|e_z^{2,0}\|_{\mathcal{V}}^2 &= -Q(e_z^{2,0},\overline{e_z^{2,0}}) \\
    &= -Q(e^{2,1}+z e^{1,0},e^{1,2}+\overline{z}e^{0,1}) \\
    &= i(\overline{z}-z) \\
    &= 2 \mathrm{Im}(z).
\end{split}
\end{equation}

We may think of $z \mapsto e_z^{2,0}$ as a holomorphic section of the hermitian line bundle $F_z^2$. Writing $\Omega$ for its Chern form, we find
\begin{equation} \label{eq:Chern_form_type_II}
\begin{split}
\Omega &= \frac{1}{2\pi i} \partial \overline{\partial} \log \|e_z^{2,0}\|_{\mathcal{V}}^2 
= \frac{i}{8\pi} \frac{dz \wedge d\overline{z}}{\mathrm{Im}(z)^2} %=  \frac{1}{4\pi} \frac{dx dy}{y^2} %= ? \cdot \frac{dt d\overline{t}}{|t|^2 (\log |t|)^2}. 
\end{split}
\end{equation}

Using \eqref{eq:e_z20_Hodge_norm_type_II} we can compute the Hodge norm $\|v\|_\mathcal{V}^2$ of vectors $v \in V_\mathbb{R}$. Fix such $v$ and let $v_z^{p,q}$ be the components of $v \in \mathcal{V}_z = \oplus_{p,q} \mathcal{V}_z^{p,q}$. Then 
$$
v_z^{2,0} = f(z) e_z^{2,0}
$$
for some holomorphic function $f:\mathbb{H} \to \mathbb{C}$. We have $v_z^{0,2} = \overline{v_z^{2,0}}$ and hence
$$
Q(v,e_z^{2,0}) = Q(v^{0,2}_z,e_z^{2,0}) = -2\overline{f(z)}\mathrm{Im}(z).
$$
This gives
\begin{equation} \label{eq:Hodge_norm_explicit_1}
\begin{split}
h(s_v) &= -2Q(v_z^{2,0},v_z^{0,2}) \\
&= 4|f(z)|^2 \mathrm{Im}(z) \\
&= \frac{|Q(v,e_z^{2,0})|^2}{\mathrm{Im}(z)}
\end{split}
\end{equation}
and, for the Hodge norm,
\begin{equation} 
\|v\|^2_{\mathcal{V},z} = Q(v,v) + 2 h(s_v) = Q(v,v) + \frac{2|Q(v,e_z^{2,0})|^2}{\mathrm{Im}(z)}.
\end{equation}
Let us consider the special case $v \in W_{2,\mathbb{R}}$. Such a vector can be written uniquely as 
\begin{equation}
v=v_2 + ae^{1,0} + \overline{a}e^{0,1}
\end{equation}
with $v_2 \in V_2$ and a complex number $a$. Using $Q(W_1,W_2)=Q(F^1_\mathrm{lim},F^2_\mathrm{lim})=0$ we find that
$$
Q(v_2,e_z^{2,0}) = Q(v_2,e^{2,1}) + z Q(v_2,e^{1,0}) = 0
$$
and hence
$$
Q(v,e_z^{2,0}) = Q(a e^{1,0} + \overline{a}e^{0,1},e^{2,1}+z e^{1,0}) = -i\overline{a}.
$$
We conclude that for $v \in W_{2,\mathbb{R}}$ we have
\begin{equation} \label{eq:Hodge_norm_explicit_1_type_II}
h(s_v) = \frac{|a|^2}{\mathrm{Im}(z)}
\end{equation}
and
\begin{equation} \label{eq:Hodge_norm_explicit_2_type_II}
\|v\|^2_{\mathcal{V},z} = Q(v_2,v_2) + \frac{2|a|^2}{\mathrm{Im}(z)}.
\end{equation}

\subsection{Type III degenerations} \label{subsection:type_iii_degenerations} We say that a degeneration is of type III if $N^2 \neq 0$. In this case we have
$$
h^{0,0}_{\mathrm{lim}} = h^{2,2}_{\mathrm{lim}} = 1, \quad h^{1,1}_{\mathrm{lim}}=\mathrm{rank} V_\mathbb{Z}-2,
$$
and all other Hodge numbers are zero, i.e. the real mixed Hodge structure $(V_\mathbb{R},W,F)$ is Hodge--Tate. The weight filtration is
\begin{align*}
W_0 &= W_1 =  \mathrm{Im} N^2 \\
W_2 &= W_3 = \mathrm{Im} N + \ker N.
\end{align*}

We will now do some computations analogous to the ones above in the type II case. We assume again that the limit mixed Hodge structure is $\mathbb{R}$--split.

\subsubsection{} Let $V_\mathbb{C}=\oplus_{a,b} I^{a,b}$ denote Deligne's canonical bigrading \cite[(2.12)]{CKS}. For $\mathbb{R}$--split type III degenerations we have $I^{p,q} = 0$ if $p \neq q$ and
$$
I^{p,p} = F_\mathrm{lim}^p \cap \overline{F_{\mathrm{lim}}^p} \cap W_{2p,\mathbb{C}}.
$$
The bigrading satisfies
$$
W_{2k,\mathbb{C}} = \oplus_{a \leq k} I^{a,a}, \quad F^p_\mathrm{lim} = \oplus_{a \geq p} I^{a.a}.
$$
We set
$$
V_{2k} = I^{k,k} \cap W_{k,\mathbb{R}}, \quad k = 0, 1, 2.
$$
Then $V_{2k}$ is a real Hodge structure of type $(k,k)$  and  $V_\mathbb{R} = \oplus V_{2k}$.

Since the form $Q(\cdot,N^2 \cdot)$ polarizes the Hodge structure on $\mathrm{Gr}_4^W V_\mathbb{R} \simeq V_4$, we can find a vector (unique up to multiplication by $\pm 1$) $e^{2,2} \in V_4$ such that
\begin{equation} \label{eq:type_II_e22_def}
Q(e^{2,2},N^2 e^{2,2})=1.
\end{equation}
Then $Ne^{2,2} \in V_2$ and $N^2 e^{2,2} \in V_0$ and both vectors are non-zero. Since $V_0$ and $V_4$ are one--dimensional we have
$$
V_4=\langle e^{2,2} \rangle, \quad V_0 = \langle N^2 e^{2,2} \rangle.
$$

Let us define
$$
U = \ker(N:V_2 \to V_0)
$$
Under the isomorphism $V_2 \simeq \mathrm{Gr}_2^W V_\mathbb{R}$ induced by the quotient map $W_{2,\mathbb{R}} \to \mathrm{Gr}_2^W V_\mathbb{R}$, the subspace $U \subset V_2$ corresponds to the primitive part $P_{2,\mathbb{R}} \subset \mathrm{Gr}_2^W V_\mathbb{R}$; in particular, the restriction of $Q$ to $U$ is positive definite. We have
$$
V_2 = U \oplus \langle Ne^{2,2} \rangle.
$$
This decomposition is orthogonal for $Q$ since $N \in \mathfrak{so}(V,Q)$. It follows that $V_\mathbb{R}$ can be written as
\begin{equation}
V_\mathbb{R} = U \oplus \langle e^{2,2},Ne^{2,2},N^2 e^{2,2} \rangle
\end{equation}
with $\langle e^{2,2}, Ne^{2,2}, N^2 e^{2,2} \rangle = U^\perp$ (in fact this is a decomposition as real mixed Hodge structures with the natural Hodge filtrations defined by intersecting $F^\bullet_\mathrm{lim}$ with each summand). By \eqref{eq:type_II_e22_def}, the matrix of the restriction of $Q$ to $U^\perp$ in the basis $\{e^{2,2},Ne^{2,2},N^2 e^{2,2}\}$ is
\begin{equation} \label{eq:Q_matrix_type_III}
    \begin{pmatrix}
    0 & 0 & 1 \\ 0 & -1 & 0 \\ 1 & 0 & 0
    \end{pmatrix}.
\end{equation}

\subsubsection{} Consider now the nilpotent orbit corresponding to an $\mathbb{R}$--split type III degeneration: for $z \in \mathbb{H}$, let
$$
F^k_z = e^{zN} F_{\mathrm{lim}}^k.
$$
Since the restriction of $N$ to $U$ vanishes, we have
\begin{align*}
F_z^2 &= \langle e_z^{2,0} \rangle \\
F_z^1 &= e^{zN}(V_{4,\mathbb{C}} \oplus V_{2,\mathbb{C}})= U_\mathbb{C} \oplus \langle e_z^{2,0}, N e_z^{2,0} \rangle,
\end{align*}
with
$$
e_z^{2,0}:= e^{zN}e^{2,2} = e^{2,2}+zNe^{2,2} + \frac{z^2}{2} N^2 e^{2,2}.
$$
Using \eqref{eq:Q_matrix_type_III} we can compute the Hodge norm of $e_z^{2,0}$:
\begin{equation} \label{eq:e_z20_Hodge_norm_type_III}
\begin{split}
    \|e_z^{2,0}\|_{\mathcal{V}}^2 &= -Q(e_z^{2,0},\overline{e_z^{2,0}}) \\
    &= -Q(e^{2,2}+zNe^{2,2}+\tfrac{z^2}{2} N^2e^{2,2},e^{2,2}+\overline{z}Ne^{2,2}+\tfrac{\overline{z}^2}{2} N^2e^{2,2}) \\
    &= -(\tfrac{z^2}{2}+\tfrac{\overline{z}^2}{2}-|z|^2) \\
    &= 2 \mathrm{Im}(z)^2.
\end{split}
\end{equation}
Writing $\Omega$ for the first Chern form of the hermitian line bundle $F^2_z$, this gives
\begin{equation} \label{eq:Chern_form_type_III}
\Omega = \frac{1}{2\pi i}  \partial \overline{\partial} \log \|e_z^{2,0}\|_\mathcal{V}^2 = \frac{i}{4\pi} \frac{dz \wedge d\overline{z}}{\mathrm{Im}(z)^2}.
\end{equation}
The argument we used in the case of type II degenerations shows that
\begin{equation}
h(s_v) = -2\frac{|Q(v,e_z^{2,0})|^2}{Q(e_z^{2,0},\overline{e_z^{2,0}})} = \frac{|Q(v,e_z^{2,0})|^2}{ \mathrm{Im}(z)^2}
\end{equation}
and, for the Hodge norm,
\begin{equation}
\|v\|_{\mathcal{V},z}^2 = Q(v,v)+2h(s_v) = Q(v,v) + \frac{2|Q(v,e_z^{2,0})|^2}{ \mathrm{Im}(z)^2}.
\end{equation}
Again we consider the special case $v \in W_{2,\mathbb{R}}=V_0 \oplus V_2$. Such a vector can be written as
\begin{equation} \label{eq:vectors_W_2_type_iii}
v = v_U + aNe^{2,2} + b N^2e^{2,2}
\end{equation}
for unique $v_U \in U$ and real numbers $a$ and $b$. Using $Q(F^1_\mathrm{lim},F^2_\mathrm{lim})=0$ and $Q(v,Nv')=-Q(Nv,v')$, we find that
$$
Q(v_U,e_z^{2,0}) = Q(v_U,e^{2,2}) + zQ(v_U,N e^{2,2}) + \tfrac{z^2}{2} Q(v_U,N^2 e^{2,2})=0
$$
and hence
\begin{equation}
\begin{split}
Q(v,e_z^{2,0}) &= Q(aN e^{2.2} + b N^2 e^{2,2}, e^{2,2}+z Ne^{2,2} + \tfrac{z^2}{2} N^2 e^{2,2}) \\
&= b-az.
\end{split}
\end{equation}
We conclude that for $v \in W_{2,\mathbb{R}}$ we have
\begin{equation} \label{eq:Hodge_norm_explicit_1_type_III}
h(s_v) = \frac{|b-az|^2}{\mathrm{Im}(z)^2} = a^2 + \left(\frac{b-a\mathrm{Re}(z)}{\mathrm{Im}(z)}\right)^2
\end{equation}
and
\begin{equation} \label{eq:Hodge_norm_explicit_2_type_III}
\begin{split}
\|v\|^2_{\mathcal{V},z} &= Q(v_U,v_U) -a^2+ \frac{2|b-az|^2}{\mathrm{Im}(z)^2} \\
&= Q(v_U,v_U) + a^2 + 2\left(\frac{b-a\mathrm{Re}(z)}{\mathrm{Im}(z)} \right)^2.
\end{split}
\end{equation}

\begin{remark}
The formulas for Hodge norms will be used in Section \ref{subsection:integrability_type_iii} to derive explicit expressions for $\varphi_{\mathbb{V}}(v)$ for type III degenerations that agree with those computed by Funke in \cite{Funke}.
\end{remark}

\section{Kudla--Millson theta series and degenerations of Hodge structure}
 
In this section we briefly review some of the needed background regarding the Weil representation and the construction of theta series attached to a $\mathbb{Z}$--PVHS $\mathbb{V}$ on $S$ by pulling back Kudla--Millson theta series via the period map associated with $\mathbb{V}$. We will also define certain theta series attached to limiting mixed Hodge structures.

\subsection{Weil representation} 

\subsubsection{} Let $L$ be an even lattice, that is, a free abelian group of finite rank endowed with a non--degenerate symmetric bilinear form $Q: L \times L \to \mathbb{Z}$ such that $Q(v,v)$ is even for every $v \in L$. We denote its signature by $(b^+, b^-)$ and write
$$
L^\vee = \{ v \in L \otimes \mathbb{Q} \ | \ Q(v,w) \in \mathbb{Z} \text{ for all } w \in L \}
$$
for the dual of $L$. Thus $L \subseteq L^\vee$, and the finite group $L^\vee/L$ is known as the discriminant group of $L$. We denote by
$$
\mathbb{C}[L^\vee/L]
$$
its group algebra and by $e^\mu$ ($\mu \in L^\vee/L$) its standard basis.

We write $\mathrm{Mp}_2(\mathbb{Z})$ for the metaplectic double 
cover of $\mathrm{SL}_2(\mathbb{Z})$. Its elements are pairs of the form
$$
\left( \begin{pmatrix} a & b \\ c & d \end{pmatrix}, \phi(\tau) \right),
$$
where $\left(\begin{smallmatrix} a & b \\ c & d \end{smallmatrix}\right) \in \mathrm{SL}_2(\mathbb{Z})$ and $\phi(\tau)$ satisfies $\phi(\tau)^2=c \tau + d$. It is generated by the elements
\begin{equation}
\begin{split}
T &= \left( \begin{pmatrix} 1 & 1 \\ 0 & 1 \end{pmatrix}, 1 \right), \\
S &= \left( \begin{pmatrix} 0 & -1 \\ 1 & 0 \end{pmatrix}, \sqrt{\tau} \right).
\end{split}
\end{equation}

There is a representation $\rho_L$ of $\mathrm{Mp}_2(\mathbb{Z})$ on $\mathbb{C}[L^\vee/L]$ determined by the formulas
\begin{equation}
\begin{split}
\rho_L(T)(e^\mu) &= e^{\pi i Q(\mu,\mu)} e^\mu \\
\rho_L(S)(e^\mu) &= \frac{e^{\pi i (b^- -b^+)/4}}{\sqrt{|L^\vee/L|}} \sum_{\lambda \in L^\vee/L} e^{-2\pi i Q(\mu,\lambda)} e^\lambda.
\end{split}
\end{equation}
The representation $\rho_L$ factors through a double cover of $\mathrm{SL}_2(\mathbb{Z}/N\mathbb{Z})$, where $N$ (sometimes called the level of $L$) is the smallest integer such that $NQ(\lambda,\lambda)/2$ is an integer for all $\lambda \in L^\vee$.

\subsubsection{} Let $V = L\otimes \mathbb{Q}$  and suppose given a filtration 
$$
0 \neq W_1 \subseteq W_2 \subseteq V
$$
of $V$ by $\mathbb{Q}$-vector spaces such that $W_1$ is isotropic and $W_2=W_1^\perp$. Define $L_k = W_k \cap L$ and set $\mathrm{Gr}_2^W L =L_2/L_1$. Then $\mathrm{Gr}_2^W L$ is an even lattice with respect to the bilinear form induced by $Q$. The group $\mathrm{Mp}_2(\mathbb{Z})$ acts on
$$
\mathbb{C}[(\mathrm{Gr}_2^W L)^\vee/\mathrm{Gr}_2^W L] =  \mathbb{C}[(L^\vee \cap W_2)/(L^\vee \cap W_1 + L_2)]
$$
via the corresponding Weil representation $\rho_{\mathrm{Gr}_2^W L}$. The map
\begin{equation} \label{eq:def_iota_map}
\iota: \rho_{\mathrm{Gr}_2^W L} \to \rho_L, \quad e^\mu \mapsto \sum_{\lambda \in (L^\vee \cap W_1+L)/L} e^{\lambda+\mu}
\end{equation}
intertwines the $\mathrm{Mp}_2(\mathbb{Z})$--actions \cite[Prop. 6.1]{Scheithauer}.

\subsubsection{} We briefly recall some notions of modular forms for $\mathrm{Mp}_2(\mathbb{Z})$ valued in $\rho_L$; see \cite[Chap. 1]{BruinierBook} for more details. Let 
$$
f: \mathbb{H} \to \rho_L
$$
be a smooth function and $k^+, k^- \in \tfrac{1}{2}\mathbb{Z}$. We say that $f$ is a non--holomorphic modular form of weight $(k^+,k^-)$ valued in $\rho_L$ if
\begin{equation} \label{eq:vv_mod_form}
f\left(\frac{a\tau + b}{c \tau + d} \right) = \phi(\tau)^{2k^+} \overline{\phi(\tau)}^{2k^-} \rho_L\left( \begin{pmatrix} a & b \\ c & d \end{pmatrix}, \phi(\tau) \right)f(\tau)
\end{equation}
for every $\left( \left( \begin{smallmatrix} a & b \\ c & d \end{smallmatrix}\right), \phi(\tau) \right) \in \mathrm{Mp}_2(\mathbb{Z})$.

If $f:\mathbb{H} \to \rho_L$ is holomorphic and satisfies \eqref{eq:vv_mod_form} with $(k^+,k^-)=(k,0)$, then we may write
$$
f(\tau) = \sum_{\mu} f_\mu \cdot e^\mu,
$$ 
where the components $f_\mu$ of $f$ are weakly modular forms of weight $k$. We say that $f$ is a modular (resp. cusp) form of weight $k$ valued in $\rho_L$ if it each component is a modular (resp. cusp) form of weight $k$.

The most important examples of modular forms valued in $\rho_L$ arise from theta series. For a positive definite even lattice $L$ and $\mu \in L^\vee$, define
$$
\Theta_L(\tau)_\mu = \sum_{\lambda \in \mu+L} e^{\pi i Q(\lambda,\lambda)\tau}
$$ 
and set
$$
\Theta_L(\tau) = \sum_\mu \Theta_L(\tau)_\mu \cdot e^\mu.
$$
Then $\Theta_L(\tau)$ is a modular form valued in $\rho_L$ of weight $\mathrm{rk}(L)/2$ (\cite[Thm 4.1]{Borcherds}).

%Mathai--Quillen and Kudla--Millson forms
\subsection{Kudla--Millson theta series}

\subsubsection{} Let $\mathbb{V} \to S$ be a $\mathbb{Z}$--PVHS satisfying \ref{hypothesis:hyp_on_V}. Associated with $\mathbb{V}$ there is a period map
$$
\Phi_\mathbb{V}: S \to \Gamma \backslash \mathbb{D}
$$
into a quotient of the hermitian symmetric space attached to $\mathrm{SO}(h^{1,1},2)$ (see e.g. \cite[p. 227-228]{Schmid}). More precisely, fix a point $s_0 \in S$ and let $\pi :\tilde{S} \to S$ be the universal cover of $S$. The pullback $\pi^*\mathcal{V}_\mathbb{Z}$ to $\tilde{S}$ is then a constant local system endowed with a constant bilinear form induced by $Q$, i.e. of the form $\underline{V_\mathbb{Z}}$ for some indefinite lattice $(V_\mathbb{Z},Q)$. It carries a canonical action
\begin{equation}
\label{eq:period_domain_local_system_action_1}
\pi_1(S,s_0) \to \mathrm{Aut}(V_\mathbb{Z},Q).
\end{equation}
Let now $V_\mathbb{R}=V_\mathbb{Z} \otimes \mathbb{R}$ and denote by $\mathbb{D}$ the space of all Hodge structures on $V_\mathbb{R}$ polarized by $Q$ with $h^{2,0}=1$; thus $\mathbb{D}$ is the hermitian symmetric domain attached to the orthogonal group $\mathrm{Aut}(V_\mathbb{R},Q)$. The pullback $\pi^*\mathbb{V}$ induces a holomorphic map $\Phi_{\pi^* \mathbb{V}}:\tilde{S} \to \mathbb{D}$. If $\Gamma \subseteq \mathrm{Aut}(V_\mathbb{Z},Q)$ is any subgroup containing the image of \eqref{eq:period_domain_local_system_action_1}, then the composite of $\Phi_{\pi^*\mathbb{V}}$ with the projection $\mathbb{D} \to \Gamma \backslash \mathbb{D}$ induces a holomorphic map $S \to \Gamma \backslash \mathbb{D}$. Under assumption \ref{hypothesis:hyp_on_V}, and identifying $\mathcal{V}_\mathbb{Z}^\vee/\mathcal{V}_\mathbb{Z} \simeq V_{\mathbb{Z}}^\vee/V_\mathbb{Z}$, we can take for $\Gamma$ the group
$$
\Gamma:=\Gamma_{V_\mathbb{Z}} = \left\{ \gamma \in \mathrm{Aut}(V_\mathbb{Z},Q) \ | \ \gamma \equiv \mathrm{id} \text{ on } V_\mathbb{Z}^\vee/V_\mathbb{Z} \right\},
$$
and denote the corresponding period map by
\begin{equation}
\label{eq:period_map_def_1}
\Phi_\mathbb{V} : S \to \Gamma \backslash \mathbb{D}.
\end{equation}

\subsubsection{} Let $\mathcal{S}(V_\mathbb{R})$ be the Schwartz space of $V_\mathbb{R}$. In their seminal works \cite{KudlaMillson1, KudlaMillson2, KudlaMillson3}, Kudla and Millson have introduced certain differential forms 
$$
\varphi_\mathrm{KM} \in (\Omega^{1,1}(\mathbb{D}) \otimes \mathcal{S}(V_\mathbb{R}))^{\mathrm{SO}(V_\mathbb{R},Q)}
$$
and associated theta series
$$
\Theta_{\mathrm{KM}}(\tau)_\mu= \sum_{v \in \mu + V_\mathbb{Z}} \varphi_{\mathrm{KM}}(y^{1/2}v) e^{\pi i x Q(v,v)} \in \Omega^{1,1}(\mathbb{D})^\Gamma \simeq \Omega^{1,1}(\Gamma \backslash \mathbb{D}).
$$
Using the period map $\Phi_\mathbb{V}$ we can define differential forms on $S$ canonically associated with $\mathbb{V}$ by pulling back the Kudla--Millson theta series.

\begin{definition}
For $\mu \in \mathcal{V}_\mathbb{Z}^\vee/\mathcal{V}_\mathbb{Z}$, define
$$
\Theta_\mathbb{V}(\tau)_\mu = \Phi_\mathbb{V}^*\Theta_{\mathrm{KM}}(\tau)_\mu.
$$
\end{definition}
The results of Kudla and Millson imply that the forms $\Theta_\mathbb{V}(\tau)_\mu$ have modularity properties that can be described most easily by saying that the differential form
\begin{equation}\label{eq:def_vv_Theta_V}
\Theta_\mathbb{V}(\tau) := \sum_{\mu \in \mathcal{V}^\vee_\mathbb{Z}/\mathcal{V}_\mathbb{Z}} \Theta_\mathbb{V}(\tau)_\mu \cdot e^\mu \in \Omega^{1,1}(S) \otimes \rho_{\mathcal{V}_\mathbb{Z}}.
\end{equation}
transforms under $\mathrm{Mp}_2(\mathbb{Z})$ like a non--holomorphic modular form of weight $\mathrm{rk}(\mathcal{V}_\mathbb{Z})/2$ valued in $\rho_{\mathcal{V}_\mathbb{Z}}$.

Using the formulas given by Kudla and Millson we can describe $\Theta_\mathbb{V}(\tau)_\mu$ as
\begin{equation}
\Theta_{\mathbb{V}}(\tau)_\mu = \sum_{\mu +\mathcal{V}_\mathbb{Z}} \varphi_\mathbb{V}(y^{1/2}v) e^{\pi i x Q(v,v)},
\end{equation}
with $\varphi_\mathbb{V}(v)$ (which is only locally defined, e.g. on a small disk around a given point in $S$) given by
\begin{equation}
\varphi_\mathbb{V}(v) = e^{-\pi \|v\|_\mathcal{V}^2}(-\Omega + ih(s_v)\theta \wedge \overline{\theta}), \quad \theta=\frac{\partial h(s_v)}{h(s_v)}.
\end{equation}
We briefly explain the terms in this formula; for more details, see also \cite{GarciaSPD}. The terms $\|v\|_\mathcal{V}^2$ and $h(s_v)$ have been defined in \ref{subsection:VHS_definitions}: $\|v\|_\mathcal{V}^2$ denotes the Hodge norm of $v$ and the value of $h(s_v)$ at $z \in S$ is
$$
h(s_v)_z = -2Q(v^{2,0}_z,v^{0,2}_z).
$$
The term $\Omega$ denotes the first Chern form of $\mathcal{L}$, i.e.
\begin{equation}
\Omega = (2\pi i)^{-1} \partial \overline{\partial}\log \|s\|_\mathcal{V}^2
\end{equation}
for any meromorphic section $s$ of $\mathcal{L}$.

We will sometimes write
\begin{equation}
\varphi_\mathbb{V}(v)=e^{-\pi Q(v,v)}\varphi_\mathbb{V}^\circ(v),
\end{equation}
with
\begin{equation}
\varphi^\circ_\mathbb{V}(v) = e^{-2\pi h(s_v)}(-\Omega + ih(s_v)\theta \wedge \overline{\theta}), \quad \theta=\frac{\partial h(s_v)}{h(s_v)}.
\end{equation}

%Theta series and MHS
\subsection{Theta series and limit MHS}

We now associate a vector--valued theta series to a limiting mixed Hodge structure of type II or III.

\subsubsection{Type II} For a type II degeneration, the polarization $Q$ induces a quadratic form on $\mathrm{Gr}_2^W V$ that we still denote by $Q$; note that in this case $\mathrm{Gr}_2^W V=P_2$ and hence this quadratic form is positive definite. The image of $V_\mathbb{Z} \cap W_2$ in $\mathrm{Gr}_2^W V$ defines a lattice that we denote by $\mathrm{Gr}_2^W V_\mathbb{Z}$. We write
\begin{equation}
\begin{split}
\rho_{\mathrm{Gr}_2^W V_\mathbb{Z}} &= \mathbb{C}[(\mathrm{Gr}_2^W V_\mathbb{Z})^\vee/\mathrm{Gr}_2^WV_\mathbb{Z}] \\
&= \mathbb{C}\left[ W_2 \cap V_\mathbb{Z}^\vee/(W_1 \cap V_\mathbb{Z}^\vee + W_2 \cap V_\mathbb{Z}) \right]
\end{split}
\end{equation}
for the corresponding Weil representation of $\mathrm{Mp}_2(\mathbb{Z})$. Associated to the positive--definite even lattice $(\mathrm{Gr}_2^W V_\mathbb{Z},Q)$ is the theta series 
$$
\Theta_{\mathrm{Gr}_2^W V_\mathbb{Z}}(\tau) = \sum_\mu \Theta_{\mathrm{Gr}_2^W V_\mathbb{Z}}(\tau)_\mu \cdot e^\mu.
$$
It is a modular form valued in $\rho_{\mathrm{Gr}_2^W \mathcal{V}_\mathbb{Z}}$ of weight $(\mathrm{rk} (\mathcal{V}_\mathbb{Z})-4)/2$.

More generally, the image of $V_\mathbb{Z} \cap W_k$ in $\mathrm{Gr}_k^W V$ is a lattice that we denote by $\mathrm{Gr}_k^W V_\mathbb{Z}$. Then
$$
N : \mathrm{Gr}_3^W V_\mathbb{Z} \to \mathrm{Gr}_1^W V_\mathbb{Z}
$$
is an injective map between lattices of the same rank; following \cite{StewartVologodsky} we write
$$
r_1(V_\mathbb{Z},N)
$$
for the size of its cokernel.\mynote{TO DO: explain why it is a perfect square, relation with monodromy pairing.} The form $Q$ also induces a non-degenerate bilinear pairing
\begin{equation}
\mathrm{Gr}^W_{3,1} Q: \mathrm{Gr}_3^W V_\mathbb{Z} \times \mathrm{Gr}_1^W V_\mathbb{Z} \to \mathbb{Z}.
\end{equation}
Let $\mathrm{disc}(\mathrm{Gr}_{3,1}^W Q)$ be its discriminant, that is
$$
\mathrm{disc}(\mathrm{Gr}_{3,1}^W Q) = |\det (\mathrm{Gr}_{3,1}^W Q(\tilde{v}_i,\tilde{w}_j))|
$$
for any bases $(\tilde{v}_i)$ of $\mathrm{Gr}_3^W V_\mathbb{Z}$ and $(\tilde{w}_j)$ of $\mathrm{Gr}_1^W V_\mathbb{Z}$. Note that the form $Q_3(v,w)=Q(v,Nw)$ is symplectic and takes integral values on $\mathrm{Gr}_3^W V_\mathbb{Z}$, and hence
$$
r_1(V_\mathbb{Z},N) \mathrm{disc}(\mathrm{Gr}_{3,1}^W Q) = |\det(Q_3(\tilde{v}_i,\tilde{v}_j))|=\deg(Q_3)^2
$$
for a positive integer $\deg(Q_3)$.

Let us write
\begin{equation} \label{eq:def_iota_map_1}
\iota : \rho_{\mathrm{Gr}_2^W V_\mathbb{Z}} \to \rho_{\mathcal{V}_\mathbb{Z}}
\end{equation}
for the $\mathrm{Mp}_2(\mathbb{Z})$--intertwining map defined in \eqref{eq:def_iota_map}; we recall that
$$
\iota(e^\lambda) = \sum_{\gamma \in (W_1 \cap V_\mathbb{Z}^\vee + V_\mathbb{Z})/V_\mathbb{Z}} e^{\gamma + \lambda}.
%\left\{ \begin{array}{cc}  \end{array} \right.
$$

For $\mu \in (\mathrm{Gr}_2^W V_\mathbb{Z})^\vee/\mathrm{Gr}_2^W V_\mathbb{Z}$, define
\begin{equation} \label{eq:def_theta_MHS_type_II}
Z_{\mathbb{V},P}^-(\tau)_\mu = \left(\frac{r_1(V_\mathbb{Z},N)}{\mathrm{disc}(\mathrm{Gr}_{3,1}^W Q)} \right)^{1/2} \frac{1}{4\pi 
y} \Theta_{\mathrm{Gr}_2^W V_\mathbb{Z}} (\tau)_\mu
\end{equation}
and set
\begin{equation} \label{eq:def_theta_MHS_type_II_2}
Z_{\mathbb{V},P}^-(\tau) = \sum_{\mu \in (\mathrm{Gr}_2^W V_\mathbb{Z})^\vee/\mathrm{Gr}_2^W V_\mathbb{Z}} Z_{\mathbb{V},P}^-(\tau)_\mu \cdot \iota(e^{\mu}).
\end{equation}

\subsubsection{Type III} \label{subsubsection:Theta_MHS_type_III} Let us now consider degenerations of type III. Let
\begin{equation} \label{eq:def_shadow_type_II_1}
\mathrm{Gr}_{2,\mathrm{prim}}^W V = \ker(N: W_2/W_1 \to W_0) \subset \mathrm{Gr}_2^{W} V.
\end{equation}
Then $\mathrm{Gr}_{2,\mathrm{prim}} V$ is a vector space over $\mathbb{Q}$ of dimension $n-1$. The subgroup 
$$
\mathrm{Gr}_{2,\mathrm{prim}}^W V_\mathbb{Z}:=\mathrm{Gr}_{2,\mathrm{prim}}^W V \cap \mathrm{Gr}^W_2 V_\mathbb{Z}
$$
 is a lattice in $\mathrm{Gr}_{2,\mathrm{prim}}^W V$. Since $W_1$ is $Q$-isotropic, the polarization $Q$ induces a quadratic form on $\mathrm{Gr}_{2,\mathrm{prim}}^W V$ that is positive definite and that we still denote by $Q$. Let us write
\begin{equation}
\begin{split}
\rho_{\mathrm{Gr}_{2,\mathrm{prim}}^W V_\mathbb{Z}} &= \mathbb{C}[(\mathrm{Gr}_{2,\mathrm{prim}}^W V_\mathbb{Z})^\vee / \mathrm{Gr}_{2,\mathrm{prim}}^W V_\mathbb{Z}] 
\end{split}
\end{equation}
for be the corresponding Weil representation of $\mathrm{Mp}_2(\mathbb{Z})$. Associated to the positive definite even lattice  $(\mathrm{Gr}_{2,\mathrm{prim}}^W V_\mathbb{Z},Q)$ is the theta series 
$$
\Theta_{\mathrm{Gr}_{2,\mathrm{prim}}^W V_\mathbb{Z}}(\tau) = \sum_{\mu} \Theta_{\mathrm{Gr}_{2,\mathrm{prim}}^W V_\mathbb{Z}}(\tau)_\mu \cdot e^\mu
$$
that transforms under $\mathrm{Mp}_2(\mathbb{Z})$ like a holomorphic modular form valued in $\rho_{\mathrm{Gr}_{2,\mathrm{prim}}^W V_\mathbb{Z}}$ of weight $(n-1)/2$.

The bilinear form $Q$ induces a pairing
$$
\mathrm{Gr}_{4,0}^W Q : \mathrm{Gr}_4^W V_\mathbb{Z} \times \mathrm{Gr}_0^W V_\mathbb{Z} \to \mathbb{Z}.
$$
Let $\mathrm{disc}(\mathrm{Gr}_{4,0}^W Q)$ be its discriminant, that is
$$
\mathrm{disc}(\mathrm{Gr}_{4,0}^W Q) = |\det (\mathrm{Gr}_{4,0}^W Q(\tilde{v}_i,\tilde{w}_j))|
$$
for any bases $(\tilde{v}_i)$ of $\mathrm{Gr}_4^W V_\mathbb{Z}$ and $(\tilde{w}_j)$ of $\mathrm{Gr}_0^W V_\mathbb{Z}$ respectively.

Let us now consider the rank one lattice
$$
\mathrm{Gr}_4^W V_\mathbb{Z} = \text{image of } V_\mathbb{Z} \text{ in } \mathrm{Gr}_4^W V,
$$
endowed with the positive--definite quadratic $Q_4(v,v)=Q(v,N^2 v)$ defined in \eqref{eq:Q_k_definition}. 

\begin{lemma} \label{lemma:integrality_N_type_iii}  Let $L$ be the rank one lattice $\mathrm{Gr}_4^W V_\mathbb{Z}$ endowed with the positive--definite quadratic form $v \mapsto Q_4(v,v)$.
\begin{enumerate}
\item[(i)]
The lattice $L$ is even.
\item[(ii)] The image of $L$ under $N: \mathrm{Gr}_4^W V \to \mathrm{Gr}_2^W V$ lies in $\mathrm{Gr}_2^W V_\mathbb{Z}$.
\item[(iii)] The image of $L^\vee$ under $N:\mathrm{Gr}_4^W V \to \mathrm{Gr}_2^W V$ contains 
$(\mathrm{Gr}_2^W V_\mathbb{Z})^\vee \cap N(\mathrm{Gr}_4^W V)
$.

\end{enumerate}
\end{lemma} 
\begin{proof}
For a degeneration of type III we have
$$
N= (T-1) - (T-1)^2/2
$$ 
and so $N^2=(T-1)^2$. Hence we can write $N=(T-1)-N^2/2$ and for $v \in V_\mathbb{Z}$ we have
$$
Q(v,N^2 v) = -Q(Nv,Nv) = -Q((T-1)v,(T-1)v)
$$
which is an even integer since $(T-1)v \in V_\mathbb{Z}$. This proves (i). Part (ii) follows from the identity $N=(T-1)-N^2/2$, which implies that $N \equiv (T-1) \mod W_0$. For part (iii), suppose that $w \in (\mathrm{Gr}_2^W V_\mathbb{Z})^\vee \cap N(\mathrm{Gr}_4^W V)$ and write $w=Nv$ with $v \in \mathrm{Gr}_4^W V$. For any $w' \in \mathrm{Gr}_2^W V_\mathbb{Z}$, we have
$$
Q(v,Nw')=-Q(w,w') \in \mathbb{Z},
$$
i.e. $Q(v,v') \in \mathbb{Z}$ for every $v' \in N(\mathrm{Gr}_2^W V_\mathbb{Z})$. Part (ii) implies that $N^2(L) \subseteq N(\mathrm{Gr}_2^W V_\mathbb{Z})$ and hence that $v \in L^\vee$.
\end{proof}

The proof the lemma shows that $N^2=(T-1)^2$ is integral. Thus
$$
\mathrm{Gr}^W N^2 : \mathrm{Gr}_4^W V_\mathbb{Z} \to \mathrm{Gr}_0^W V_\mathbb{Z} 
$$
is an injective map of lattices of the same rank; we denote its cokernel by
$$
r_2(V_\mathbb{Z},N).
$$
Writing $\mathrm{Vol}(\mathrm{Gr}^W_4 V_\mathbb{Z}):=Q_4(v_0,v_0) \in 2\mathbb{N}$ where $v_0$ is a generator of $\mathrm{Gr}^W_4 V_\mathbb{Z}$, we have
$$
\mathrm{Vol}(\mathrm{Gr}^W_4 V_\mathbb{Z}) = r_2(V_\mathbb{Z},N)\mathrm{disc}(\mathrm{Gr}^W_{4,0} Q).
$$

Associated to the positive--definite even lattice $(\mathrm{Gr}_4^W V_\mathbb{Z},Q_4)$ is the Weil representation
$$
\rho_{\mathrm{Gr}_4^W V_\mathbb{Z}} = \mathbb{C}[(\mathrm{Gr}_4^W V_\mathbb{Z})^\vee/\mathrm{Gr}_4^W V_\mathbb{Z}]
$$
of $\mathrm{Mp}_2(\mathbb{Z})$ 
and the non--holomorphic unary theta series
\begin{equation}
R_{\mathrm{Gr}_4^W V_\mathbb{Z}}(\tau)_\nu = \frac{1}{4\pi \sqrt{y}} \sum_{v \in \nu + \mathrm{Gr}_4^W V_\mathbb{Z}} \beta_{3/2}(2\pi y Q_4(v,v)) q^{-Q_4(v,v)/2},
\end{equation}
where
$$
\beta_{3/2}(t) = \int_1^\infty u^{-3/2}e^{-tu}du.
$$

Let us write $(\mathrm{Gr}_4^W V_\mathbb{Z})^-$ for the negative--definite even lattice defined by $-Q_4$. By the above lemma, the map
$$
\mathrm{Gr}_{2,\mathrm{prim}}^W V_\mathbb{Z} \hat{\oplus} (\mathrm{Gr}_4^W V_\mathbb{Z})^- \to \mathrm{Gr}_2^W V_\mathbb{Z}, \quad (v,w) \mapsto v+Nw
$$
identifies $\mathrm{Gr}_{2,\mathrm{prim}} V_\mathbb{Z} \hat{\oplus} (\mathrm{Gr}_4^W V_\mathbb{Z})^-$ with a sublattice of $\mathrm{Gr}_2^W V_\mathbb{Z}$ of finite index. This induces a natural $\mathrm{Mp}_2(\mathbb{Z})$--intertwining map
\begin{equation} \label{eq:def_iota_map_2}
\rho_{\mathrm{Gr}_{2,\mathrm{prim}} V_\mathbb{Z}} \otimes \rho_{(\mathrm{Gr}_4^W V_\mathbb{Z})^-} \simeq \rho_{\mathrm{Gr}_{2,\mathrm{prim}} V_\mathbb{Z} \oplus (\mathrm{Gr}_4^W V_\mathbb{Z})^-} \to \rho_{\mathrm{Gr}_2^W V_\mathbb{Z}}
\end{equation}
$$
e^\lambda \otimes e^\nu \mapsto \left\{\begin{array}{rr} 0, & \text{ if } \lambda+ N \nu \notin (\mathrm{Gr}_2^W V_\mathbb{Z})^\vee, \\ \lambda + N\nu + \mathrm{Gr}_2^W V_\mathbb{Z}, & \text{ otherwise.}  \end{array} \right.
$$
Let
\begin{equation}
\iota: \rho_{\mathrm{Gr}_{2,\mathrm{prim}} V_\mathbb{Z}} \otimes \rho_{(\mathrm{Gr}_4^W V_\mathbb{Z})^-} \to \rho_{\mathcal{V}_\mathbb{Z}}
\end{equation}
be the map obtained by composing \eqref{eq:def_iota_map_2} with the map  \eqref{eq:def_iota_map_1}. For $\mu \in (\mathrm{Gr}_2^W V_\mathbb{Z})^\vee/\mathrm{Gr}_2^W V_\mathbb{Z}$, define
\begin{equation} \label{eq:def_theta_MHS_type_III}
\begin{split}
Z_{\mathbb{V},P}^-(\tau)_\mu = &  \left(\frac{r_2(V_\mathbb{Z},N)}{2\mathrm{disc}(\mathrm{Gr}^{W}_{4,0} Q)}\right)^{1/2} \\ & \times \sum_{\substack{\lambda + N\nu  \equiv \mu \\ \mod \mathrm{Gr}_2^W V_\mathbb{Z}}} R_{\mathrm{Gr}_4^W V_\mathbb{Z}}(\tau)_\nu \cdot \Theta_{\mathrm{Gr}_{2,\mathrm{prim}}^{W} V_\mathbb{Z}}(\tau)_\lambda
\end{split}
\end{equation}
and set
\begin{equation} \label{eq:def_theta_MHS_type_III_2}
Z_{\mathbb{V},P}^-(\tau) = \sum_{\mu \in (\mathrm{Gr}_2^W V_\mathbb{Z})^\vee/\mathrm{Gr}_2^W V_\mathbb{Z}} Z_{\mathbb{V},P}^-(\tau)_\mu \cdot \iota(e^{\mu}).
\end{equation}

\section{Integrability of Kudla--Millson theta series}
\label{section:Integrability}

\subsection{A convergence result} \label{section:main_convergence} Let $\overline{S}$ be a compact Riemann surface and let $S$ be obtained by removing a finite set of points from $\overline{S}$. Consider a polarized variation of Hodge structure $\mathbb{V} \to S$ of weight two with $h^{2,0}=1$ and $h^{1,1}=n$ satisfying \ref{hypothesis:hyp_on_V}.  In the previous section we have attached to $\mathbb{V}$ a collection of closed differential forms
$$
\Theta_\mathbb{V}(\tau)_\mu \in \Omega^{1,1}(S), \quad \mu \in \mathcal{V}_\mathbb{Z}^\vee/\mathcal{V}_\mathbb{Z},
$$
that vary smoothly in $\tau \in \mathbb{H}$ and transform like  (non-holomorphic) modular forms of weight $\mathrm{rk}(\mathcal{V}_\mathbb{Z})/2$. If $S=\overline{S}$, i.e. when $S$ is compact, the integral
\begin{equation} \label{eq:I_V_def_integrability_section}
Z_\mathbb{V}(\tau)_\mu := \int_S \Theta_{\mathbb{V}}(\tau)_\mu
\end{equation}
is obviously convergent, and the results of Kudla and Millson \cite{KudlaMillson3} show that $Z_\mathbb{V}(\tau)_\mu$ is a holomorphic modular form of weight $1+n/2$ with $q$-expansion
$$
-\mathrm{deg}(\mathcal{L}) \delta_{\mu,0} + \sum_{m >0} \mathrm{deg} \ \mathrm{NL}_\mathbb{V}(m)_\mu \cdot q^m.
$$
A little more precisely: the $Z_\mathbb{V}(\tau)_\mu$ are the components of a modular form of weight $1+n/2$ valued in $\rho_{\mathcal{V}_\mathbb{Z}}$.

Now suppose that $S$ is not compact; in that case, the  differential form $\Theta_\mathbb{V}(\tau)$ might not extend to a smooth form on $\overline{S}$. Fortunately, as we will show below, the singularities of $\Theta_{\mathbb{V}}(\tau)$ around the points in $\overline{S} \, \backslash \, S$ are very mild. In particular, $\Theta_\mathbb{V}(\tau)$ is always integrable over $S$ and so one can define $Z_\mathbb{V}(\tau)$ as in \eqref{eq:I_V_def_integrability_section} for arbitrary $S$. This is the content of the following proposition, whose proof will comprise most of this section. Let us write
$$
\Theta_\mathbb{V}(\tau)_\mu = \sum_{m \in \tfrac{1}{2}Q(\mu,\mu) +\mathbb{Z}} \Theta^\circ_\mathbb{V}(y)_{m,\mu} \cdot q^m,
$$
with
$$
\Theta^\circ_\mathbb{V}(y)_{m,\mu} = \sum_{\substack{v \in \mu+ \mathcal{V}_\mathbb{Z} \\ Q(v,v)=2m}} \varphi^\circ_\mathbb{V}(y^{1/2} v).
$$

\begin{theorem}  \label{thm:convergence_main}
Let $\mathbb{V}$ be a $\mathbb{Z}$--PVHS over $S$ of weight two with $h^{2,0}=1$. Then the integral
$$
Z_\mathbb{V}(\tau)_\mu:= \int_S \Theta_\mathbb{V}(\tau)_\mu
$$
converges for all $\mu \in \mathcal{V}_\mathbb{Z}^\vee/\mathcal{V}_\mathbb{Z}$. The forms $\Theta_{\mathbb{V}}^\circ(y)_{m,\mu}$ are also integrable over $S$ for any $m$ and $\mu$ and
we have
$$
Z_\mathbb{V}(\tau)_\mu = \sum_{m \in \tfrac{Q(\mu,\mu)}{2}+\mathbb{Z}} \left( \int_S \Theta^\circ_\mathbb{V}(y)_{m,\mu} \right) q^m.
$$
The expression
$$
Z_\mathbb{V}(\tau):= \sum_{\mu \in \mathcal{V}_\mathbb{Z}^\vee/\mathcal{V}_\mathbb{Z}} Z_\mathbb{V}(\tau)_\mu \cdot e^\mu
$$
defines a (possibly non-holomorphic) modular form of weight $1+n/2$ valued in $\rho_{\mathcal{V}_\mathbb{Z}}$.
\end{theorem}

The following remarks reduce the proof of this theorem to the analogous local question around each cusp. Moreover, they show that when addressing the local question we may assume the local monodromy to be unipotent and non-trivial.

\begin{enumerate}
\item Fix a point $s_0 \in S$ and a simply connected neighbourhood $U$ of $s_0$ and choose a trivialization of $\mathcal{V}_\mathbb{Z}|_{U}$. The Hodge metric on $\mathcal{V}|_{U}$ can then be identified with a smooth map from $U$ to the space of hermitian metrics on $\mathbb{C}^{n+2}$. Hence, after possibly shrinking $U$, the Hodge metric on $\mathcal{V}|_{U}$ is uniformly bounded below by some constant metric $v \mapsto |v|_0$ on $\mathbb{C}^{n+2}$. It follows that on such a neighbourhood we can find $\epsilon>0$ so that
$$
|\varphi_{\mathbb{V}}(v)|_s < e^{-\epsilon |v|_0^2}
$$
for every $s \in U$ and every flat section $v \in \mathcal{V}_\mathbb{Z}$ over $U$. By dominated convergence, this implies that the above proposition holds locally, that is, replacing $S$ by a small enough neighbourhood of any given point $s_0 \in S$. Taking a finite covering of $\overline{S}$ shows that it suffices to prove the proposition for a coordinate neighbourhood of each cusp, i.e. for $S \simeq \Delta^*$. In the rest of Section \ref{section:Integrability} we will assume that $S= \Delta^*$ and denote by $T$ the local monodromy as in Section \ref{subsection:local_monodromy}.

\item Recall that $T$ is quasi--unipotent: there exist positive integers $e$ and $m$ such that $(T^e-\mathrm{id})^m=0$. Let $\pi: \Delta^* \to \Delta^*$ be the covering map of degree $e$. Then $\pi^*\mathbb{V}$ is a PVHS over $\Delta^*$ with unipotent monodromy $T^e$. Since
$$
\Theta_{\pi^*\mathbb{V}}(\tau) = \pi^*\Theta_{\mathbb{V}}(\tau),
$$
the integrability of $\Theta_{\pi^*\mathbb{V}}(\tau)$ over $\Delta^*$ implies that of $\Theta_{\mathbb{V}}(\tau)$. It follows that we may assume that $T$ is unipotent.

\item Finally, recall from Section \ref{subsection:local_monodromy} that if $T=\mathrm{id}$, then the period map associated to the PVHS $\mathbb{V}$ extends to $\Delta$; in this case the argument in (1) proves the proposition. So in the rest of Section \ref{section:Integrability} we will assume that $T$ is unipotent and $T \neq 1$.

\end{enumerate}

So we will prove the proposition for $S=\Delta^*$ and a PVHS $\mathbb{V} \to \Delta^*$ with unipotent non-trivial monodromy. In order to ensure that the only degeneration of $\mathbb{V}$ happens as $t \to 0$, we may and do assume that $\mathbb{V}$ extends to a punctured disk centered at $0$ of radius strictly larger than one.

The proof is based on Schmid's Hodge norm estimates and his nilpotent orbit and $\mathrm{SL}_2$--orbit theorems. As a first step, let us write 
$$
\Theta_\mathbb{V}(\tau)_\mu = \Theta_\mathbb{V}(\tau)'_\mu + \Theta_\mathbb{V}(\tau)''_\mu,
$$ 
with
\begin{equation}
\begin{split}
\Theta_\mathbb{V}(\tau)'_\mu &= \sum_{\substack{ v \in \mu + \mathcal{V}_\mathbb{Z} \\ v \in W_2}} \varphi_\mathbb{V}(y^{1/2} v) \ e^{\pi i Q(v,v)x} \\
\Theta_\mathbb{V}(\tau)''_\mu &= \sum_{\substack{ v \in \mu + \mathcal{V}_\mathbb{Z} \\ v \notin W_2}} \varphi_\mathbb{V}(y^{1/2}v) \ e^{\pi i Q(v,v)x},
\end{split}
\end{equation}
and similarly we write $\Theta_\mathbb{V}^\circ(y)_{m,\mu} = \Theta_\mathbb{V}^\circ(y)'_{m,\mu} + \Theta_\mathbb{V}^\circ(y)''_{m,\mu}$.

It turns out that the integrability of $\Theta_\mathbb{V}(\tau)''_\mu$ and $\Theta^\circ_\mathbb{V}(y)''_{m,\mu}$ is easier to prove: it is a straightforward consequence of the Hodge norm estimates. The integrability of $\Theta_\mathbb{V}(\tau)'_\mu$ and $\Theta^\circ_\mathbb{V}(y)'_{m,\mu}$ is more delicate: our proof proceeds by bounding the difference $\varphi_{\mathbb{V}}(v)-\varphi_{\mathbb{V}^{\mathrm{nilp}}}(v)$ and computing $\varphi_{\mathbb{V}^{\mathrm{nilp}}}(v)$ explicitly.

\subsection{Integrability of $\Theta_{\mathbb{V}}''$} \label{subsection:Integrability_Theta''} Let us fix a $\mathbb{Z}$-PVHS $\mathbb{V}$ with $h^{2,0}=1$ over the punctured unit disk 
$$
\Delta^* = \{ t \in \mathbb{C} \ | \ 0 < |t|<1\}
$$
with unipotent non-trivial monodromy. We will first establish the integrability of $\Theta_\mathbb{V}(\tau)''$ and $\Theta^\circ_\mathbb{V}(\tau)''_m$. 

For this it suffices to work on a fixed angular sector 
\begin{equation}
U:= \{t \in \Delta^*| \epsilon < \arg(t) <2\pi-\epsilon\}  \subset \Delta^*.    
\end{equation}
In order to estimate the size of differential forms on $\Delta^*$, we work with the Poincar\'e metric, defined by declaring that the coframe $\tfrac{dt}{t \log |t|}$ and $\tfrac{d\overline{t}}{\overline{t}\log|t|}$ is unitary. In particular, a form $\alpha \in \Omega^{1,1}(\Delta^*)$ can be written as
$$
\alpha=\alpha_{11} \frac{dt d\overline{t}}{|t|^2 (\log |t|)^2}
$$
for a unique smooth function $\alpha_{11}$ on $\Delta^*$, and for $t \in \Delta^*$ we set
$$
|\alpha|_t := |\alpha_{11}(t)|.
$$
We say that $\alpha$ is rapidly decreasing if $|\alpha|_t=O(t^\epsilon)$ for some $\epsilon>0$ and $t$ in a given $U$.

Fix a basis $v_1,\ldots,v_{n+2}$ adapted to the weight filtration as in \ref{section:canonical extension} giving a trivialization $\mathcal{V}|_{U} \simeq \underline{\mathbb{C}}^{n+2}$  and denote by $h(t)=(h_{ij}(t))$ the matrix of the Hodge metric in this basis: for a flat section $v=a_1 v_1 + \ldots a_{n+2} v_{n+2}$ and $t \in U$ we have
\begin{equation} \label{eq:Hodge_norm_matrix_form}
\|v\|^2_t = a^* h(t) a =  \sum_{i,j} \overline{a_i}a_j h_{ij}(t).
\end{equation}

The basis $v_1,\ldots, v_{n+2}$ gives a splitting of the complexified weight filtration: writing
\begin{equation} \label{eq:weight_filt_splitting}
Y_k = \langle  v_{1+\dim W_{k-1}}, \ldots v_{\dim W_k} \rangle \subset V_\mathbb{C},
\end{equation}
we have $W_{k,\mathbb{C}} = Y_k \oplus W_{{k-1},\mathbb{C}}$ for all $k \geq 0$. Following K\'ollar \cite[Definition 5.3.(v)]{Kollar}, denote by 
$$
e: \mathcal{V}|_U \to \mathcal{V}|_U
$$ 
the endomorphism of the vector bundle $\mathcal{V}|_U$ that acts on the fiber $\mathcal{V}_t$ by
$$
v \mapsto (-\log|t|)^{(k-2)/2} v \quad \text{ if } v \in Y_k
$$
and set
$$
\tilde{h}={^t}e^{-1} h e^{-1}.
$$
It follows from the Hodge norm estimates that the entries $\tilde{h}_{ij}$ of $\tilde{h}$ and $(\det \tilde{h})^{-1}$ are bounded. More precisely, write $\mathcal{C}^\omega(\Delta)$ for the set of real analytic functions on $\Delta$ and $L$ for the set of Laurent polynomials in $(-\log |t|)^{1/2}$ with complex coefficients and define
$$
B\Delta = \{ f \in \mathcal{C}^\omega(\Delta) \otimes L \ | \ f \text{ bounded} \}.
$$
Then
\begin{equation} \label{eq:tilde_h_is_bounded}
\tilde{h}_{ij}, \ (\det \tilde{h})^{-1} \in B\Delta
\end{equation}
(cf. \cite[Prop. 5.4]{Kollar}). Moreover, $B\Delta$ is closed under the operators $t \log |t| \tfrac{d}{dt}$ and $\overline{t} \log |t| \tfrac{d}{d\overline{t}}$. Hence, in the coframe given by $\tfrac{dt}{t \log |t|}$ and $\tfrac{d\overline{t}}{\overline{t}\log|t|}$, the forms
\begin{equation} \label{eq:Hodge_components_nearly_bdd}
\partial \tilde{h}_{ij}, \overline{\partial} \tilde{h}_{ij}, \partial \overline{\partial} \tilde{h}_{ij}
\end{equation}
have components that belong to $B\Delta$; we will refer to forms with this property as nearly bounded (loc. cit., Def. 5.3). Note that the product of two nearly bounded forms is nearly bounded.

We can use the bounds \eqref{eq:tilde_h_is_bounded} and \eqref{eq:Hodge_components_nearly_bdd} to give an estimate for the form $\varphi_{\mathbb{V}}(v)$.

\begin{lemma}  \label{lemma:bound_varphi_V}
There exists a positive constant $C$ such that
$$
|\varphi_{\mathbb{V}}(v)|_t < C e^{-\pi \|v\|_t^2} \left(1+\|v\|_t^2 \right)
$$
for any $t \in U$ and any $v \in V_\mathbb{R}$.
\end{lemma}
\begin{proof}
Let 
\begin{equation} \label{eq:def_poly_p_V}
p_\mathbb{V}(v) = \left(-\Omega+ i h(s_v)\theta \wedge \overline{\theta} \right), \quad \theta := \frac{\partial h(s_v)}{h(s_v)}.
\end{equation}
Since $\varphi_\mathbb{V}(v) = e^{-\pi \|v\|_t^2} p_\mathbb{V}(v)$, it suffices to show that
$$
|p_\mathbb{V}(v)|_t < C(1+\|v\|_t^2)
$$
for $t \in U$ and $v \in V_\mathbb{R}$. Now the form $\Omega$ is the first Chern form of the hermitian line bundle $F^2$. It is known to be bounded when the monodromy is unipotent \cite[Prop. 1.11]{Zucker}; that is, $|\Omega|_t$ is bounded on $\Delta^*$. To estimate the term $h(s_v) \theta \wedge \overline{\theta}$, recall that
$$
-2\pi i \Omega = \partial \overline{\partial} \log h(s_v) = \frac{\partial \overline{\partial}h(s_v)}{h(s_v)} - \theta \wedge \overline{\theta}.
$$
Multiplying by $h(s_v)$ gives
\begin{equation} \label{eq:simplify_p_V}
h(s_v) \theta \wedge \overline{\theta} = 2\pi i h(s_v) \Omega + \partial \overline{\partial} h(s_v).
\end{equation}
Since $h(s_v)=2 \|v^{2,0}\|_t^2 \leq 2 \|v\|^2_t$, we have
$
|h(s_v)\Omega|_t \leq 2\|v\|_t^2 |\Omega|_t
$,
and so it remains to estimate the term $\partial \overline{\partial} h(s_v)$. Writing $v_t=\Sigma v^{p,q}_t$ for the Hodge decomposition of $v_t \in \mathcal{V}_t$, we  have $h(s_v)=-2Q(v_t^{2,0},v_t^{0,2})$ and hence
\begin{equation} \label{eq:relation_Q_h_and_majorant}
\|v\|_t^2 = Q(v^{1,1}_t,v^{1,1}_t)- 2Q(v_t^{2,0},v_t^{0,2}) = Q(v,v) + 2h(s_v)
\end{equation}
and
$$
\partial \overline{\partial}h(s_v) = 2^{-1} \partial \overline{\partial} \|v\|^2_t = 2^{-1} \sum \overline{a_i} a_j \partial \overline{\partial}h_{ij}(t).
$$
Let us next give a bound for the forms $\partial \overline{\partial}h_{ij}(t)$. We have $\tilde{h}_{ij}=e^{-1}_{ij}h_{ij}$, where $e_{ij}(t)=(-\log |t|)^{a_{ij}/2}$ and $a_{ij}$ is the integer defined by $a_{ij}=(m-2)+(l-2)$ if $v_i \in W_{m}-W_{m-1}$ and $v_j \in W_{l}-W_{l-1}$. A direct computation shows that the forms $e_{ij}^{-1} \partial e_{ij}$, $e_{ij}^{-1} \overline{\partial} e_{ij}$ and $e_{ij}^{-1} \partial \overline{\partial} e_{ij}$ are nearly bounded; writing
\begin{equation*}
\begin{split}
\partial \overline{\partial} h_{ij} = \tilde{h}_{ij} \partial \overline{\partial} e_{ij} + \partial \tilde{h}_{ij} \overline{\partial} e_{ij} - \overline{\partial}\tilde{h}_{ij} \partial e_{ij} + e_{ij} \partial \overline{\partial} \tilde{h}_{ij}
\end{split}
\end{equation*}
and applying \eqref{eq:Hodge_components_nearly_bdd} shows that $e_{ij}^{-1} \partial \overline{\partial}h_{ij}$ is nearly bounded too. Thus
$$
|\partial \overline{\partial} h(s_v)|_t = O(\sum |a_i  a_j| e_{ij}(t)),
$$
and by \cite[Lemma 5.6]{Kollar}, we have 
\begin{equation} \label{eq:Kollar_bound_1}
\sum |a_i a_j| e_{ij}(t) = O(\|v\|_t^2).
\end{equation}
This finishes the proof.
\end{proof}

The lemma implies the (very coarse) bound
\begin{equation} \label{eq:theta_no_w_2_bound_1}
\begin{split}
|\Theta_\mathbb{V}(\tau)''_\mu|_t & \leq C  \sum_{\substack{v \in \mu + \mathcal{V}_\mathbb{Z} \\ v \notin W_2}}
 e^{-\pi y  \|v\|_t^2} (1+ y \|v\|_t^2) \\
& \leq C' \sum_{ \substack{v \in \mathcal{V}_\mathbb{Z}^\vee \\ v \notin W_2}} e^{-\pi y \|v\|_t^2/2},
\end{split}
\end{equation}
for some constant $C'>0$ depending only on $\mathbb{V}$ and the basis $(v_i)$. Using this bound we can prove the integrability of $\Theta_\mathbb{V}(\tau)''_\mu$.
\begin{proposition} \label{prop:theta_no_w_2_bound_1}
For any $m$ and $\mu$, the forms $\Theta_\mathbb{V}(\tau)''_\mu$ and $\Theta_\mathbb{V}^\circ(y)_{m,\mu}''$ are rapidly decreasing as $t \to 0$, uniformly on any angular sector $U$. We have
$$
\int_{\Delta^*} \Theta_\mathbb{V}(\tau)''_\mu = \sum_{m} \left( \int_{\Delta^*} \Theta_\mathbb{V}^\circ(y)''_{m,\mu} \right) \cdot q^m.
$$
\end{proposition}
\begin{proof}
Fix a basis $v_1,\ldots,v_{n+2}$ of $\mathcal{V}^\vee_\mathbb{Z}|_U$ adapted to the weight filtration and denote by $|\cdot|$ the metric on $\mathcal{V}$ obtained from the standard metric on $\mathbb{C}^{n+2}$ via the corresponding trivialization $\mathcal{V}|_U \simeq \underline{\mathbb{C}}^{n+2}$. Define $Y_k$ as in \eqref{eq:weight_filt_splitting}. For $v \in \mathcal{V}_\mathbb{Q}$, let us write $v=\sum v_k$ with $v_k \in Y_k$. It follows from \eqref{eq:tilde_h_is_bounded} that there is a positive constant $c$, depending only on $\mathbb{V}$ and the basis $(v_i)$, such that for all $v \in \mathcal{V}_\mathbb{Q}$ we have
\begin{equation} \label{eq:Hodge_metric_lower_bound_1}
\|v\|_t^2 > 2c \sum_k |v_k|^2 (-\log |t|)^{k-2}.
\end{equation}

Combined with \eqref{eq:theta_no_w_2_bound_1}, this immediately implies the following bound: let us write $Y_k^\mathbb{Z} = Y_k \cap \mathcal{V}_\mathbb{Z}^\vee$ and let $l=1$ if $\mathbb{V}$ is of type II and $l=0$ if $\mathbb{V}$ is of type III; then  
$$
W_2 \cap \mathcal{V}_\mathbb{Z}^\vee = Y_2^\mathbb{Z} \oplus (W_1 \cap \mathcal{V}_\mathbb{Z}^\vee) = Y_2^\mathbb{Z} \oplus Y_{l}^\mathbb{Z}
$$
and
$$
\mathcal{V}_\mathbb{Z}^\vee = Y_{4-l}^\mathbb{Z} \oplus (W_2 \cap \mathcal{V}_\mathbb{Z}^\vee) = Y_{4-l}^\mathbb{Z} \oplus Y_2^\mathbb{Z}  \oplus Y_{l}^\mathbb{Z}
$$
and we have
\begin{equation} \label{eq:theta_no_w_2_bound_2}
\begin{split}
|\Theta_\mathbb{V}(\tau)''_\mu|_t & \leq C' \sum_{ \substack{v \in \mathcal{V}_\mathbb{Z}^\vee \\ v \notin W_2}}  e^{-\pi y \|v\|_t^2/2} \\
& = C' \sum_{\substack{0 \neq u \in Y_{4-l}^\mathbb{Z} \\ v \in Y_2^\mathbb{Z} \\ w \in Y_{l}^\mathbb{Z}}} e^{-\pi y \|u+v+w\|_t^2/2} \\
& <  C' \sum_{0 \neq u \in Y_{4-l}^\mathbb{Z}} e^{-\pi c y |u|^2 (-\log |t|)^{2-l}} \\ & \quad \times \sum_{v \in Y_2^\mathbb{Z}} e^{-\pi c y |v|^2} \\ & \quad \times \sum_{w \in Y_{l}^\mathbb{Z}} e^{-\pi c y |w|^2 (-\log |t|)^{l-2}}.
\end{split}
\end{equation}
In the last expression, the sum over $u$ is clearly rapidly decreasing as $t \to 0$, and the sum over $v$ is independent of $t$. It remains to estimate the sum over $w$. This can be done by Poisson summation: since $Y_{l}^\mathbb{Z} = W_1 \cap \mathcal{V}_\mathbb{Z}^\vee$ is a lattice of rank $l+1$, we have
\begin{equation} \label{eq:trivial_Poisson_estimate}
\sum_{w \in Y_{l}^\mathbb{Z}} e^{-\pi c y |w|^2 (-\log |t|)^{l-2}} = O\left(-\log |t| \right).
\end{equation}
This shows that $\Theta_\mathbb{V}(\tau)''_\mu$ and $\Theta_\mathbb{V}^\circ(y)''_{m,\mu}$ are rapidly decreasing as $t \to 0$. The identity in the statement follows by dominated convergence.
\end{proof}

Note that \eqref{eq:theta_no_w_2_bound_2} and \eqref{eq:trivial_Poisson_estimate} give the bound
\begin{equation}\label{eq:trivial_estimate_Theta'}
\begin{split}
|\Theta_\mathbb{V}(\tau)'_\mu|_t &= O\left( \sum_{v \in \mathcal{V}^\vee_\mathbb{Z} \cap W_2} e^{-\pi y \|v\|_t^2/2} \right) \\
&= O(-\log|t|).
\end{split}
\end{equation}
This estimate will be useful later but it is not enough to guarantee the integrability of $\Theta_\mathbb{V}(\tau)'_\mu$.

\subsection{Reduction to nilpotent orbits}
\label{subsection:reduction_to_nilpotent_orbits}
Following the strategy outlined at the end of \ref{section:main_convergence}, we must now consider the integrability of $\Theta_{\mathbb{V}}(\tau)'$ and $\Theta^\circ_{\mathbb{V}}(y)'_m$. Our next goal is to prove the following Proposition, which shows that it is enough to consider the case where $\mathbb{V}$ is a nilpotent orbit.

\begin{proposition} \label{prop:theta'_difference_bound}
Let $\mathbb{V} \to \Delta^*$ be a weight two polarized variation of Hodge structure with $h^{2,0}=1$. Assume that the monodromy is unipotent and non--trivial and let $\mathbb{V}^\mathrm{nilp}$ be the corresponding nilpotent orbit. For any $\tau \in \mathbb{H}$ and any $m$ and $\mu$, the forms
$$
\Theta_{\mathbb{V}}(\tau)'_\mu - \Theta_{\mathbb{V}^\mathrm{nilp}}(\tau)'_\mu \ 
\text{ and } \ \Theta_{\mathbb{V}}^\circ(y)'_{m,\mu} - \Theta_{\mathbb{V}^\mathrm{nilp}}^\circ(y)'_{m,\mu}  \in \Omega^{1,1}(\Delta^*)
$$
are rapidly decreasing as $t \to 0$. We have
$$
\int_{\Delta^*} \Theta_{\mathbb{V}}(\tau)'_\mu - \Theta_{\mathbb{V}^\mathrm{nilp}}(\tau)'_\mu = \sum_m \left( \int_{\Delta^*} \Theta_{\mathbb{V}}^\circ(y)'_{m,\mu} - \Theta_{\mathbb{V}^\mathrm{nilp}}^\circ(y)'_{m,\mu} \right) \cdot q^m.
$$
\end{proposition}

As a first step towards the proof let us establish an estimate for the difference between the Hodge norms $\|v\|_{\mathcal{V},t}$ and $\|v\|_{\mathcal{V}^\mathrm{nilp},t}$ of a flat section $v \in V_\mathbb{R}$ as $t \to 0$.

\begin{lemma}  \label{lemma:dist_V_to_V_nilp}
There are positive constants $A$, $B$ such that
$$
\left| \|v\|_{\mathcal{V},t}^2 - \|v\|_{\mathcal{V}^{\mathrm{nilp}},t}^2 \right| \leq  A|t|^B \|v\|_{\mathcal{V}^{\mathrm{nilp}},t}^2
$$
for every $t \in U$ and $v \in V_\mathbb{R}$.
\end{lemma}
\begin{proof}

Let $\pi:\mathbb{H} \to \Delta^*$ be the uniformizing map $z \mapsto t=e^{2\pi i z}$ and let $\Phi_1:\mathbb{H} \to \mathbb{D}$ and $\Phi_2:\mathbb{H} \to \mathbb{D}$ be the period maps induced by $\pi^*\mathbb{V}$ and $\pi^*\mathbb{V}^{\mathrm{nilp}}$ respectively. Let us write $G_\mathbb{R}=\mathrm{SO}(V_\mathbb{R},Q)$ and fix $z_0 \in \mathbb{D}$. Pick differentiable lifts
$$
\phi_1, \phi_2: \mathbb{H} \to G_\mathbb{R}
$$
of $\Phi_1$ and $\Phi_2$, i.e. maps $\phi_i$ that satisfy
$$
\Phi_i(z) = \phi_i(z)x_0, \quad z \in \mathbb{H}, \quad i=1,2.
$$
(For example, use the Iwasawa decomposition of $G_\mathbb{R}$.) Writing $\|\cdot\|_x$ for the euclidean metric on $V_\mathbb{R}$ corresponding to $x \in \mathbb{D}$, we have the equivariance property $\|gv\|_{gx}=\|v\|_
x$ for all $g \in G_\mathbb{R}$. Hence
$$
\|v\|^2_{\mathcal{V},t} = \|v\|^2_{\Phi_1(t)} = \|\phi_1(t)^{-1}v\|^2_{x_0}.
$$
and similarly $\|v\|^2_{\mathcal{V}^\mathrm{nilp},t}=\|\phi_2(t)^{-1}v\|_{x_0}^2$. Thus it suffices to prove that
$$
|\|\phi_1(t)^{-1}\phi_2(t) v\|_{x_0}^2-\|v\|^2_{x_0}| \leq A|t|^B \|v\|^2_{x_0}
$$
for all $v \in V_\mathbb{R}$, or equivalently to show that the norm $\|\phi_1(t)^{-1}\phi_2(t)\|$ of the operator $\phi_1(t)^{-1}\phi_2(t) \in \mathrm{End}(V_\mathbb{R})$ satisfies
$$
|\|\phi_1(t)^{-1}\phi_2(t)\|^2-1| \leq A|t|^B.
$$
This follows directly from Schmid's nilpotent orbit theorem \cite[Thm 4.12]{Schmid} (see also p. 244 of loc. cit. for a comparison between the operator norm and Riemannian distance on $\mathbb{D}$).
\end{proof}

Lemma \ref{lemma:dist_V_to_V_nilp} leads to the following upper bound for the difference between $\varphi_{\mathbb{V}}$ and $\varphi_{\mathbb{V}^\mathrm{nilp}}$.

\begin{lemma} \label{lemma:bound_diff_varphi_V_V_nilp}
There exist positive constants $A$, $B$ and $C$ such that
$$
|\varphi_\mathbb{V}(v) - \varphi_{\mathbb{V}^\mathrm{nilp}}(v)|_t < C |t|^B e^{-\pi(1-A|t|^B)\|v\|_{\mathcal{V},t}^2} (1+\|v\|_{\mathcal{V},t}^4)
$$
for any $t \in U$ and any $v \in V_\mathbb{R}$.
\end{lemma}
\begin{proof}
Let $p_\mathbb{V}$ be as in \eqref{eq:def_poly_p_V}. The proof of Lemma \ref{lemma:bound_varphi_V} shows that
\begin{equation} \label{eq:easy_bound_p_V}
|p_\mathbb{V}(v)|_t = O(1+\|v\|_{\mathcal{V},t}^2).
\end{equation}
By Lemma \ref{lemma:dist_V_to_V_nilp} the same upper bound holds for $|p_{\mathbb{V}^\mathrm{nilp}}(v)|_t$. Hence it suffices to establish the bounds
\begin{align}
\label{eq:exp_V_V_nilp_diff_bound}
|e^{-\pi \|v\|_{\mathcal{V},t}^2} - e^{-\pi \|v\|_{\mathcal{V}^\mathrm{nilp},t}^2}| &< C|t|^B e^{-\pi (1-A|t|^B) \|v\|_{\mathcal{V},t}^2}\|v\|_{\mathcal{V},t}^2 \\
|p_{\mathbb{V}}(v)-p_{\mathbb{V}^\mathrm{nilp}}(v)|_t &< C |t|^B (1+\|v\|_{\mathcal{V},t}^2).
\end{align}
The first bound is equivalent to
$$
|e^{-\pi(\|v\|^2_{\mathcal{V}^\mathrm{nilp},t} - \|v\|^2_{\mathcal{V},t})}-1| < C|t|^B e^{\pi A|t|^B \|v\|_{\mathcal{V},t}^2} \|v\|_{\mathcal{V},t}^2,
$$
which follows readily from Lemma \ref{lemma:dist_V_to_V_nilp} and the inequality $|e^x-1| \leq |x| e^{|x|}$ valid for all real $x$. As to the second bound, let us use \eqref{eq:simplify_p_V} and \eqref{eq:relation_Q_h_and_majorant} to write $p_\mathbb{V}(v)$ as
\begin{equation}
\begin{split}
p_\mathbb{V}(v) &= -(1+2\pi h(s_v))\Omega_\mathcal{L} +i\partial \overline{\partial}h(s_v) \\
&= -(1+\pi \|v\|_\mathcal{V}^2-\pi Q(v,v))\Omega_\mathcal{L} + i\partial \overline{\partial} \|v\|_\mathcal{V}^2/2
\end{split}
\end{equation}
and similarly
$$
p_{\mathbb{V}^{\mathrm{nilp}}}(v) = -(1+\pi \|v\|_{\mathcal{V}^\mathrm{nilp}}^2-\pi Q(v,v))\Omega_{\mathcal{L}^\mathrm{nilp}} + i\partial \overline{\partial} \|v\|_{\mathcal{V}^\mathrm{nilp}}^2/2.
$$
By Lemma \ref{lemma:dist_V_to_V_nilp} and the fact that $\Omega_\mathcal{L}$ and $\Omega_{\mathcal{L}^\mathrm{nilp}}$ are nearly bounded, it suffices to show that 
\begin{equation}\label{lemma:proof_bound_diff_varphi_V_V_nilp_identity_2}
|\Omega_{\mathcal{L}}-\Omega_{\mathcal{L}^\mathrm{nilp}}|_t=O(|t|^B)
\end{equation}
and 
\begin{equation}\label{lemma:proof_bound_diff_varphi_V_V_nilp_identity_1}
|\partial \overline{\partial}(\|v\|_{\mathcal{V},t}^2 - \|v\|_{\mathcal{V}^\mathrm{nilp},t}^2)|_t=O(|t|^B \|v\|^2_{\mathcal{V},t}).
\end{equation}
Both bounds follow from \eqref{eq:tilde_h_is_bounded}. Namely, let us write $\|v\|^2_{\mathcal{V},t}=a^*h_\mathcal{V}(t)a$ and $\|v\|^2_{\mathcal{V}^{\mathrm{nilp}},t}=a^*h_{\mathcal{V}^\mathrm{nilp}}(t)a$ as in \eqref{eq:Hodge_norm_matrix_form} and let $f_{ij}(t)=e_{ij}^{-1}(h_\mathcal{V}(t)_{ij}-h_{\mathcal{V}^\mathrm{nilp}}(t)_{ij})$; then
$$
\|v\|^2_{\mathcal{V},t}-\|v\|^2_{\mathcal{V}^{\mathrm{nilp}},t}=\sum_{i,j} a_i a_j e_{ij}(t) \cdot f_{ij}(t).
$$
Let us write $a^* e a = \Sigma_{i,j} a_i a_j e_{ij}$; then $\left|a^* e(t) a \right| = O(\|v\|^2_{\mathcal{V},t})$ by \eqref{eq:Kollar_bound_1}. Since the forms $e_{ij}^{-1} \partial e_{ij}$, $e_{ij}^{-1} \overline{\partial} e_{ij}$ and $e_{ij}^{-1} \partial \overline{\partial} e_{ij}$ are nearly bounded, the expressions $|\partial(a^* e a)|_t$, $|\overline{\partial}(a^* e a)|_t$ and $ |\partial \overline{\partial}(a^* e a)|_t$ are all $O(\|v\|^2_{\mathcal{V},t})$. Hence to establish the bound \eqref{lemma:proof_bound_diff_varphi_V_V_nilp_identity_1} it suffices to prove that the expressions $|f_{ij}(t)|$, $|\partial f_{ij}|_t$, $|\overline{\partial}f_{ij}|_t$ and $|\partial \overline{\partial}f_{ij}|_t$ are all $O(|t|^B)$ for some positive constant $B$.
% $$
% |\partial(a^* e a)|_t, \quad |\overline{\partial}(a^* e a)|_t, \quad |\partial \overline{\partial}(a^* e a)|_t = O(\|v\|^2_{\mathcal{V},t}).
% $$
Now Lemma \ref{lemma:dist_V_to_V_nilp} gives the bound
\begin{equation} \label{eq:h_V_h_V_nilp_bound}
t^{-B}(h_\mathcal{V}(t)_{ij}-h_{\mathcal{V}^\mathrm{nilp}}(t)_{ij})) \in B\Delta
\end{equation}
for some $B>0$, and so the required bounds on $f_{ij}$ and its derivatives follow from the fact that $B\Delta$ is closed under the operators $t\log|t|\tfrac{d}{dt}$ and $\overline{t}\log|t|\tfrac{d}{d\overline{t}}$.

It remains to prove the bound \eqref{lemma:proof_bound_diff_varphi_V_V_nilp_identity_2}. To see this, pick a non-zero element $v \in W_1$; then $Q(v,v)=0$ and hence $\|v\|^2_\mathcal{V}=2h_\mathcal{V}(s_v)$ and $\|v\|^2_{\mathcal{V}^{\mathrm{nilp}}}=2 h_{\mathcal{V}^{\mathrm{nilp}}}(s_v)$ by \eqref{eq:relation_Q_h_and_majorant}. We can then write $-2\pi i \Omega_{\mathcal{L}} = \partial \overline{\partial} \log \|v\|^2_{\mathcal{V}}$  and similarly $-2\pi i \Omega_{\mathcal{L}^{\mathrm{nilp}}} = \partial \overline{\partial} \log \|v\|^2_{\mathcal{V}^{\mathrm{nilp}}}$ and so to prove \eqref{lemma:proof_bound_diff_varphi_V_V_nilp_identity_2} it suffices to establish that
$$
\left| \frac{\partial \overline{\partial} \|v\|^2_{\mathcal{V}}}{\|v\|^2_{\mathcal{V}}} - \frac{\partial \overline{\partial} \|v\|^2_{\mathcal{V}^{\mathrm{nilp}}}}{\|v\|^2_{\mathcal{V}^{\mathrm{nilp}}}} \right|_t \text{ and } \left| \frac{\partial \|v\|^2_{\mathcal{V}}}{\|v\|^2_{\mathcal{V}}} - \frac{\partial  \|v\|^2_{\mathcal{V}^{\mathrm{nilp}}}}{\|v\|^2_{\mathcal{V}^{\mathrm{nilp}}}} \right|_t
$$
are of the form $O(|t|^B)$ for some $B>0$. This follows the Hodge norm estimates $\|v\|^2_{\mathcal{V}}$, $\|v\|^2_{\mathcal{V}^{\mathrm{nilp}}} \sim (-\log |t|)^{k-2}$ (for $v \in W_k-W_{k-1}$) together with the bounds provided by Lemma \ref{lemma:dist_V_to_V_nilp} and \eqref{lemma:proof_bound_diff_varphi_V_V_nilp_identity_1}.
\end{proof}

\begin{proof}[Proof of Proposition \ref{prop:theta'_difference_bound}] 

Lemma \ref{lemma:bound_diff_varphi_V_V_nilp} implies that for $|t|$ small enough we have
$$
|\varphi_\mathbb{V}(v)-\varphi_{\mathbb{V}^{\mathrm{nilp}}}(v)|_t < C |t|^B e^{-\pi \|v\|_{\mathcal{V},t}^2/2}
$$
for all $v \in V_\mathbb{R}$. Then
$$
|\Theta_\mathbb{V}(\tau)_\mu' - \Theta_{\mathbb{V}^\mathrm{nilp}}(\tau)_\mu'|_t < C |t|^B \sum_{v \in \mathcal{V}_\mathbb{Z}^\vee \cap W_2} e^{-\pi y \|v\|^2_{\mathcal{V},t}/2}
$$
With the notation of the proof of Proposition \ref{prop:theta_no_w_2_bound_1} we have
\begin{equation} \label{proof:bound_Theta'_V_to_V_nilp}
\begin{split}
\sum_{v \in \mathcal{V}_\mathbb{Z}^\vee \cap W_2} e^{-\pi y \|v\|^2_{\mathcal{V},t}/2} &< \sum_{v \in Y_2^\mathbb{Z}} e^{-\pi c y |v|^2} \\
& \ \times \sum_{w \in Y_l^\mathbb{Z}} e^{-\pi c y |w|^2 (-\log |t|)^{l-2}}
\end{split}
\end{equation}
and as argued in that proof the last expression is $O(-\log |t|)$. This gives 
$$
|\Theta_\mathbb{V}(\tau)'_\mu-\Theta_{\mathbb{V}^\mathrm{nilp}}(\tau)'_\mu|_t =O(-|t|^B \log |t|)
$$
and a similar upper bound for $|\Theta_{\mathbb{V}}^\circ(y)'_{m,\mu}-\Theta^\circ_{\mathbb{V}^\mathrm{nilp}}(y)'_{m,\mu}|_t$. The identity in the statement follows from dominated convergence.
\end{proof}

It will convenient to pass from an arbitrary nilpotent orbit $\mathbb{V}^\mathrm{nilp}$ to a special type of nilpotent orbit that we denote by $\tilde{\mathbb{V}}^{\mathrm{nilp}}$. The special feature of $\tilde{\mathbb{V}}^{\mathrm{nilp}}$ is that the corresponding limiting mixed Hodge structure splits over $\mathbb{R}$; one might refer to $\tilde{\mathbb{V}}^{\mathrm{nilp}}$ as an ``$\mathbb{R}$--split nilpotent orbit''.

To a nilpotent orbit $\mathbb{V}^\mathrm{nilp}$ one can canonically attach an $\mathbb{R}$--split nilpotent orbit $\tilde{\mathbb{V}}^\mathrm{nilp}$. The Hodge filtration $\tilde{\mathcal{F}}^\bullet$ of $\tilde{\mathbb{V}}^\mathrm{nilp}$ lives in the same complex vector space as the Hodge filtration $\mathcal{F}^\bullet$ of $\mathbb{V}^\mathrm{nilp}$, and both are related by
$$
\tilde{\mathcal{F}}^\bullet = e^{i\delta} \mathcal{F}^\bullet
$$
for a certain element $\delta$ defined by Deligne (see \cite[Prop. 2.20, p. 480]{CKS}). The two orbits $\mathbb{V}^{\mathrm{nilp}}$ and $\tilde{\mathbb{V}}^{\mathrm{nilp}}$ are close in a sense made precise by Schmid's $\mathrm{SL}_2$--orbit theorem \cite[Thm 3.25]{CKS}. As a result one obtains the following bound for the difference between the forms $\varphi_{\mathbb{V}^\mathrm{nilp}}(v)$ and $\varphi_{\tilde{\mathbb{V}}^\mathrm{nilp}}(v)$.

\begin{lemma} \label{lemma:bound_Hodge_norm_nilp_to_split_nilp_orbits}
There exists a positive constant $C$ such that
\begin{equation}
\label{eq:bound_Hodge_norm_nilp_to_split_nilp_orbits}
| \|v\|_{\mathcal{V}^{\mathrm{nilp}},t}^2 - \|v\|_{\tilde{\mathcal{V}}^{\mathrm{nilp}},t}^2 | \leq C(-\log|t|)^{-1}\|v\|_{\mathcal{V}^{\mathrm{nilp}},t}^2
\end{equation}
for all $t \in U$ and all $v \in V_\mathbb{R}$.
\end{lemma}
\begin{proof}
As in the proof of Lemma \ref{lemma:dist_V_to_V_nilp}, this bound is equivalent to a bound for operator norm of an element $g_z \in G_\mathbb{R}$ relating $\Phi_{\mathbb{V}^\mathrm{nilp}}(t)$ and $\Phi_{\tilde{\mathbb{V}}^{\mathrm{nilp}}}(t)$. The relevant bound is proved in \cite[pp.  480-481]{CKS}: in the notation of that paper, the element
$$
g_z=e^{xN}\tilde{g}(y)e^{-xN} \in G_\mathbb{R}
$$
(cf. loc. cit., eq. (3.19)) relates both filtrations, i.e. it satisfies
$$
\Phi_{\mathbb{V}^\mathrm{nilp}}(t) = g_z\Phi_{\tilde{\mathbb{V}}^{\mathrm{nilp}}}(t).
$$
The bound \eqref{eq:bound_Hodge_norm_nilp_to_split_nilp_orbits} is then equivalent to 
\begin{equation} \label{eq:SL_2_bound_1}
\sup_{v \in V_{\mathbb{R}}-0} \left|  \frac{\|g_z v\|_{\mathcal{V}^\mathrm{nilp},t}}{\|v\|_{\mathcal{V}^{\mathrm{nilp}},t}}-1 \right| \leq C (-\log |t|)^{-1}.
\end{equation}
Schmid's $\mathrm{SL}_2$--orbit theorem \cite[Thm 3.25]{CKS} shows that $\tilde{g}(y)$ admits a convergent expansion
$$
\tilde{g}(y)=\tilde{g}(\infty)(1+\tilde{g}_1 y^{-1}+\tilde{g}_2 y^{-2}+\cdots).
$$
To prove \eqref{eq:SL_2_bound_1} it suffices to establish that
$$
\sup_{v \in V_{\mathbb{R}}-0}   \frac{\|(g_z \tilde{g}(\infty)^{-1}-1) v\|_{\mathcal{V}^\mathrm{nilp},t}}{\|v\|_{\mathcal{V}^{\mathrm{nilp}},t}} \, \text{ and } \, \sup_{v \in V_{\mathbb{R}}-0}  \frac{\|(\tilde{g}(\infty)-1) v\|_{\mathcal{V}^\mathrm{nilp},t}}{\|v\|_{\mathcal{V}^{\mathrm{nilp}},t}}
$$
are $O((-\log |t|)^{-1})$. For the first expression this follows directly from the above expansion for $\tilde{g}(y)$ together with the fact that $\tilde{g}_k$ maps $W_{n,\mathbb{R}}$ to $W_{n+k-1,\mathbb{R}}$. For the second expression one uses that $\tilde{g}(\infty)-1$ maps $W_{n,\mathbb{R}}$ to $W_{n-2,\mathbb{R}}$ (cf. \cite[Thm. 3.25.(ii) and (iv)]{CKS}).
\end{proof}

\begin{lemma}
\label{lemma:bound_varphi_nilp_to_split_nilp_orbits}
There exists a positive constant $C$ such that
$$
|\varphi_{\mathbb{V}^{\mathrm{nilp}}}(v)-\varphi_{\tilde{\mathbb{V}}^{\mathrm{nilp}}}(v)|_t <  C(-\log |t|)^{-1}e^{-\frac{\pi}{2} \|v\|^2_{\mathcal{V}^\mathrm{nilp},t}}
$$
for any $v \in V_\mathbb{R}$ and any $t \in U$ with $|t|$ sufficiently small. 
\end{lemma}
\begin{proof}

The proof follows closely that of Lemma \ref{lemma:bound_diff_varphi_V_V_nilp}, replacing the use of Lemma \ref{lemma:dist_V_to_V_nilp} by \eqref{eq:bound_Hodge_norm_nilp_to_split_nilp_orbits}. The needed bounds
$$
|\Omega_{\mathcal{L}^{\mathrm{nilp}}}-\Omega_{\tilde{\mathcal{L}}^{\mathrm{nilp}}}|_t = O((-\log|t|)^{-1})
$$
and
$$
|\partial \overline{\partial}(\|v\|^2_{\mathcal{V}^{\mathrm{nilp}},t}-\|v\|^2_{\tilde{\mathcal{V}}^{\mathrm{nilp}},t})|_t = O((-\log |t|)^{-1} \|v\|^2_{\mathcal{V}^{\mathrm{nilp}},t} )
$$
follow in the same way as \eqref{lemma:proof_bound_diff_varphi_V_V_nilp_identity_2} and \eqref{lemma:proof_bound_diff_varphi_V_V_nilp_identity_1} using the fact that the subspace
$$
(-\log |t|)^{-1} B \Delta \subset B\Delta
$$
is stable under the operators $t \log|t| \tfrac{d}{dt}$ and $\overline{t} \log|t| \tfrac{d}{d\overline{t}}$.
\end{proof}

Combined with the estimate \eqref{eq:trivial_estimate_Theta'}, the lemma implies the bound
\begin{equation}
\begin{split}
|\Theta_{\mathbb{V}^\mathrm{nilp}}(\tau)'_\mu - \Theta_{\tilde{\mathbb{V}}^\mathrm{nilp}}(\tau)'_\mu|_t & \leq \sum_{v \in \mathcal{V}^\vee_\mathbb{Z} \cap W_2} |\varphi_{\mathbb{V}^{\mathrm{nilp}}}(v)-\varphi_{\tilde{\mathbb{V}}^{\mathrm{nilp}}}(v)|_t \\
&\leq C(-\log|t|)^{-1} \left( \sum_{v \in \mathcal{V}^\vee_\mathbb{Z} \cap W_2} e^{-\tfrac{\pi}{2}\|v\|^2_{\mathcal{V}^\mathrm{nilp},t}} \right) \\
&= O(1),
\end{split}    
\end{equation}
and similarly that $|\Theta_{\mathbb{V}^\mathrm{nilp}}^\circ(y)'_{m,\mu} - \Theta_{\tilde{\mathbb{V}}^\mathrm{nilp}}^\circ(y)'_{m,\mu}|_t$ is also bounded for all $m$ and $\mu$. 

It follows that to prove Theorem \ref{thm:convergence_main} it suffices to show that $\Theta_{\tilde{\mathbb{V}}^\mathrm{nilp}}(\tau)'_\mu$ and $\Theta_{\tilde{\mathbb{V}}^\mathrm{nilp}}^\circ(y)'_{m,\mu}$ are integrable over $\Delta^*$ for all $m$ and $\mu$ and satisfy
\begin{equation}\label{eq:main_convergence_split_nilp}
\int_{\Delta^*} \Theta_{\tilde{\mathbb{V}}^\mathrm{nilp}}(\tau)'_\mu = \sum_{m} \left( \int_{\Delta^*} \Theta_{\tilde{\mathbb{V}}^\mathrm{nilp}}^\circ(y)'_{m,\mu} \right) \cdot q^m.
\end{equation}

We will prove \eqref{eq:main_convergence_split_nilp} in the next two sections by distinguishing the nilpotent orbits of types II and III, using the explicit nature of $\varphi_{\tilde{\mathbb{V}}^\mathrm{nilp}}(v)$ in each case.

\subsection{Integrability for type II nilpotent orbits} \label{subsection:integrability_type_ii}

\subsubsection{} Let us first determine explicitly the form $\varphi_\mathbb{V}(v)$ corresponding to a type II nilpotent orbit $\mathbb{V}=\tilde{\mathbb{V}}^{\mathrm{nilp}}$. The setting is that of \autoref{subsection:type_ii_degenerations}; in particular, we assume that the associated limiting mixed Hodge structure is $\mathbb{R}$--split.

For a vector $v \in W_{2,\mathbb{R}}$, we can write $v=v_2+ae^{1,0}+\overline{a}e^{0,1}$ with $v_2 \in V_2$, and by \eqref{eq:Hodge_norm_explicit_2_type_II} we have
\begin{equation} \label{eq:phi_type_I_explicit_0}
\varphi_\mathbb{V}(v)=e^{-\pi Q(v_2,v_2)} \varphi_\mathbb{V}(a e^{1,0}+\overline{a} e^{0,1})
\end{equation} 
and
\begin{equation} \label{eq:type_I_explicit_metric}
\|ae^{1,0}+\overline{a} e^{0,1}\|_z^2=2|a|^2/\mathrm{Im}(z).
\end{equation}
For $v \in W_1$, we have 
$$
Q(v,v)=0 \Rightarrow \|v_z^{1,1}\|_z^2=2\|v_z^{2,0}\|_z^2
$$ 
and hence
$$
h(s_v)=2\|v^{2,0}\|_z^2 = (\|v^{1,1}\|_z^2+2\|v^{2,0}\|_z^2)/2=\|v\|_z^2/2.
$$
For $v \in W_1$, we conclude using \eqref{eq:Chern_form_type_II} that
\begin{align*}
\theta &= \frac{\partial h(s_v)}{h(s_v)} = \frac{\partial \|v\|_z^2}{\|v\|_z^2} = -\frac{\partial 
\mathrm{Im}(z)}{\mathrm{Im}(z)} = \frac{i dz}{2\mathrm{Im}(z)} \\
%\Omega &= \frac{\partial \overline{\partial} h(s_v)}{h(s_v)} + \theta \wedge \overline{\theta} \\
\theta \wedge \overline{\theta} &= \frac{dz \wedge d\overline{z}}{4\mathrm{Im}(z)^2}=-2\pi i \Omega
\end{align*}
and hence
\begin{equation} \label{eq:phi_type_I_explicit_1}
\begin{split}
\varphi_\mathbb{V}(v)&=e^{-\pi \|v\|_z^2}(-\Omega + ih(s_v) \theta \wedge \overline \theta) \\
%&= e^{-\pi \|v\|_z^2}(-\frac{1}{2\pi}+ h(s_v)) i\theta \wedge \overline \theta \\
&= e^{-\pi \|v\|_z^2}(\pi \|v\|_z^2 -1)\Omega
%&=e^{-(a^2+b^2)h_z^{-1}}(-1 + \frac{a^2+b^2}{4h_z})\frac{i}{2\pi} \frac{dz \wedge d\overline{z}}{h_z^2} \\
\end{split}
\end{equation}
For $v=ae^{1,0} + \overline{a}e^{0,1}$, by \eqref{eq:type_I_explicit_metric} we obtain
\begin{equation} \label{eq:phi_type_I_explicit_2}
\varphi_\mathbb{V}(ae^{1,0}+\overline{a}e^{0,1}) = \phi\left(\frac{a}{(\mathrm{Im}(z)/2)^{1/2}}\right)\Omega,
\end{equation}
where $\phi:\mathbb{C} \to \mathbb{R}$ is the Schwartz form defined by
\begin{equation} \label{eq:phi_type_II_explicit_3}
\phi(a) = e^{-\pi |a|^2}(\pi |a|^2-1).
\end{equation}
Let us write $\mathcal{F}\phi$ for the Fourier transform of $\phi$. In order to estimate $\Theta_{\mathbb{V}}(\tau)'_\mu$ we will need to compute $\mathcal{F}\phi(0)$. Using polar coordinates we find
\begin{equation} \label{eq:Four_transf_const_term_type_II}
\begin{split}
\mathcal{F}\phi(0) &= \int_{\mathbb{C}} \phi(a)da \\ &= \int_0^\infty \int_0^{2\pi} \phi(r \cos \theta, r \sin\theta) r dr d\theta \\
&= 2\pi \int_0^\infty e^{-\pi r^2 } (\pi r^2  -1) r dr \\
&= 2\pi \left(-e^{-\pi r^{2}} \frac{r^2}{2} \right|_{0}^\infty = 0.
\end{split}
\end{equation}

\subsubsection{} Using the above description of $\varphi_{\mathbb{V}}$ we can compute $\Theta_{\mathbb{V}}(\tau)'$ explicitly. We write $W_j^{\mathbb{Z}} = V_\mathbb{Z} \cap W_j$ and obtain a filtration
$$
0 = W_0^\mathbb{Z} \subset W_1^\mathbb{Z} \subset W_2^{\mathbb{Z}} \subset W_3^{\mathbb{Z}}=V_\mathbb{Z}
$$
of the local system $V_\mathbb{Z}$. The associated quotients
$$
\mathrm{Gr}^W_j V_\mathbb{Z} := W_{j}^{\mathbb{Z}}/W_{j-1}^{\mathbb{Z}}, \quad j=1,2,3,
$$
are local systems of free abelian groups of ranks $2$, $n-2$ and $2$ respectively. Note that if $\mu \notin W_2+V_\mathbb{Z}$, then $(\mu + V_\mathbb{Z})\cap W_2=\emptyset$ and $\Theta_\mathbb{V}(\tau)'_\mu = 0$.

Recall the Deligne splitting introduced in  \eqref{subsection:type_ii_degenerations}: we have
$$
W_1 \otimes \mathbb{C} = I^{1,0} \oplus I^{0,1}, \quad W_2 \otimes \mathbb{C} = I^{1,0} \oplus I^{0,1} \oplus I^{1,1}.
$$
Since $I^{1,1}$ is stable under complex conjugation, this induces a splitting of the filtration $W_1 \otimes \mathbb{R} \subset W_2 \otimes \mathbb{R}$:
$$
W_2 \otimes \mathbb{R} = W_1 \otimes \mathbb{R} \oplus (I^{1,1} \cap V_\mathbb{R}).
$$
We denote by $\pi_1: W_2 \otimes \mathbb{R} \to W_1 \otimes \mathbb{R}$ and $\pi_2: W_2 \otimes \mathbb{R} \to (I^{1,1} \cap V_{\mathbb{R}})$ the resulting projections. By \eqref{eq:phi_type_I_explicit_0}, for $v \in W_2^\mathbb{Z}$ we have
\begin{equation} 
\varphi_{\mathbb{V}}(v) = e^{-\pi Q(\pi_2(v),\pi_2(v))} \varphi_\mathbb{V}(\pi_1(v)).
\end{equation}
Since $Q(W_1,W_2)=0$, we can rewrite this as
\begin{equation}\label{eq:V_split_phi_identity_1}
\varphi_{\mathbb{V}}(v) = e^{-\pi Q(v,v)} \varphi_\mathbb{V}(\pi_1(v)).
\end{equation}
For the theta series $\Theta_{\mathbb{V}}(\tau)'_\mu$ with $\mu \in W_2 + V_\mathbb{Z}$ we have $(\mu+V_\mathbb{Z}) \cap W_2=\mu+W_2^\mathbb{Z}$ and hence
\begin{equation}
\begin{split}
\Theta_{\mathbb{V}}(\tau)'_\mu &= \sum_{v \in \mu
+ W_2^\mathbb{Z}} \varphi_{\mathbb{V}}(y^{1/2} v) e^{\pi i Q(v,v)x} \\
&= \sum_{v \in (\mu + W_2^\mathbb{Z})/W_1^\mathbb{Z}} q^{Q(v,v)/2} \sum_{v_1 \in W_1^\mathbb{Z}} \varphi_{\mathbb{V}}(y^{1/2}(v_1 + \pi_1(v))).
\end{split}
\end{equation}
We will now use Poisson summation to give an upper bound for the inner sum. Let us fix flat sections $\lambda_1$, $\lambda_{2}$ of $W_1^\mathbb{Z}$ giving a trivialization
\begin{equation} \label{eq:trivialization_W_1_type_II}
\Phi: \underline{\mathbb{Z}}^{2} \overset{\sim}{\to} W_1^\mathbb{Z}, \quad \Phi((a_1,a_2))=a_1 \lambda_1 + a_2 \lambda_2.
\end{equation}
Write
\begin{equation}
\begin{split}
\lambda_i &= \alpha_i e^{1,0} + \overline{\alpha_i} e^{0,1}, \quad i=1,2,
\end{split}
\end{equation}
for some complex numbers $\alpha_1$, $\alpha_2$. By \eqref{eq:Hodge_norm_explicit_2_type_II}, for a vector $a=(a_1,a_2) \in \mathbb{R}^2$, the Hodge metric of $\Phi(a)$ is given by
$$
\|\Phi(a)\|_z =  2|a_1 \alpha_1+a_2 \alpha_2|^2/ \mathrm{Im}(z),
$$
Let us define a Schwartz function $\tilde{\phi}:\mathbb{R}^2 \to \mathbb{C}$ by
$$
\tilde{\phi}(x_1,x_2) = \phi \left( x_1 \alpha_1 + x_2 \alpha_2 \right).
$$
Writing $\pi_1(v) = \Phi(a)$ for some $a \in \mathbb{R}^2$ and using \eqref{eq:phi_type_I_explicit_2} and \eqref{eq:phi_type_II_explicit_3} we have
\begin{equation}
\begin{split}
\sum_{v_1 \in W_1^\mathbb{Z}} \varphi_\mathbb{V}(y^{1/2}(v_1+\pi_1(v))) &= \left(  \sum_{n \in \mathbb{Z}^2} \tilde{\phi} \left(\frac{a+n}{\sqrt{\mathrm{Im}(z)/2y}} \right)\right)\Omega.
%(a_1+n_1)\alpha_1 + (a_2+n_2)\alpha_2)
\end{split}
\end{equation}

Writing $\mathcal{F}\tilde{\phi}$ for the Fourier transform of $\tilde{\phi}$, an application of Poisson summation gives
\begin{equation}
\sum_{n \in \mathbb{Z}^2} \tilde{\phi}\left(\frac{a+n}{\sqrt{\mathrm{Im}(z)/2y}}\right) = \frac{\mathrm{Im}(z)}{2y} \sum_{m \in \mathbb{Z}^2} e^{2\pi i m \cdot a} \mathcal{F}\tilde{\phi}( \sqrt{\mathrm{Im}(z)/2y} \cdot m ).
\end{equation}

% \begin{equation}
% \sum_{n \in \mathbb{Z}^2} \phi(h_z^{-1/2}A(a+n)) = \frac{h_z}{|\det A|} \sum_{m \in \mathbb{Z}^2} e^{2\pi i m \cdot (h_z^{-1/2}A a)} \mathcal{F}\phi(h_z {^t}A^{-1}m).
% \end{equation}
By \eqref{eq:Four_transf_const_term_type_II}, the term corresponding to $m=0$ vanishes, and we obtain the upper bound, uniform in $a$,
\begin{equation}
\left| \sum_{v_1 \in W_1^\mathbb{Z}} \varphi_\mathbb{V}(\sqrt{y}(v_1+\pi_1(v))) \right|_t \leq \frac{\mathrm{Im}(z)}{2y} \sum_{m \in \mathbb{Z}^2-0} |\mathcal{F}\tilde{\phi}(\sqrt{\mathrm{Im}(z)/2y} \cdot m)| \cdot |\Omega|_t
\end{equation}
This expression decreases rapidly as  $\mathrm{Im}(z) \to \infty$, which implies the integrability of $\Theta_{\mathbb{V}}(\tau)'_\mu$. Similarly, using $Q(W_1,W_2)=0$, we may write
\begin{equation}
\begin{split}
\Theta_{\mathbb{V}}^\circ(y)'_{m,\mu} &= \sum_{\substack{ v \in \mu + W_2^\mathbb{Z} \\ Q(v,v)=2m}} \varphi^\circ_\mathbb{V}(y^{1/2}v) \\
&= \sum_{\substack{ v \in (\mu + W_2^\mathbb{Z})/W_1^\mathbb{Z} \\ Q(v,v)=2m}} \sum_{v_1 \in W_1^\mathbb{Z}} \varphi^\circ_\mathbb{V}(y^{1/2}(v_1 + \pi_1(v))).
\end{split}
\end{equation}
Note that the outer sum is finite since $Q$ polarizes the pure Hodge structure $\mathrm{Gr}_2^W \mathbb{V}$, which is purely of type $(1,1)$. The above bound establishes that $\Theta_{\mathbb{V}}^\circ(y)'_{m,\mu}$ is integrable for all $m$ and implies the identity
$$
\int_{\Delta^*} \Theta_{\mathbb{V}}(\tau)'_\mu =\sum_{m} \left( \int_{\Delta^*} \Theta_{\mathbb{V}}^\circ(y)'_{m,\mu} \right) \cdot q^m
$$
by dominated convergence.

\subsection{Integrability for type III nilpotent orbits} \label{subsection:integrability_type_iii} We now determine the form $\varphi_\mathbb{V}(v)$ and theta series $\Theta_{\mathbb{V}}(\tau)'_\mu$ for a type III nilpotent orbit $\mathbb{V}=\tilde{\mathbb{V}}^{\mathrm{nilp}}$. We will work in the setting of \autoref{subsection:type_iii_degenerations}; in particular, we assume that the associated limiting mixed Hodge structure is $\mathbb{R}$--split.

\subsubsection{} Let $v \in W_{2,\mathbb{R}}$. As in \eqref{eq:vectors_W_2_type_iii}, we may write
$v=v_U+aNe^{2,2}+bN^2 e^{2,2}$ with $v_U \in U$ and real numbers $a$ and $b$. By \eqref{eq:Hodge_norm_explicit_1_type_III} and \eqref{eq:Hodge_norm_explicit_2_type_III}, we have
\begin{equation} \label{eq:phi_type_III_explicit_0}
\varphi_\mathbb{V}(v) = e^{-\pi Q(v_U,v_U)} \varphi_\mathbb{V}(aNe^{2,2}+bN^2e^{2,2}).
\end{equation}
Differentiating $h(s_v)=|b-az|^2/\mathrm{Im}(z)^2$ gives
$$
\theta(s_v) = \frac{\partial h(s_v)}{h(s_v)} =  \left(- \frac{a}{b-az} + \frac{i}{\mathrm{Im}(z)} \right)dz
$$
and hence
$$
\theta \wedge \overline{\theta} = \left| \frac{b-a\mathrm{Re}(z)}{(b-az)\mathrm{Im}(z)} \right|^2 dz \wedge d\overline{z}.
$$
Using \eqref{eq:Chern_form_type_III} we obtain
\begin{equation}
\begin{split}
i h(s_v) \theta \wedge \overline{\theta} &= \frac{|b-a\mathrm{Re}(z)|^2}{\mathrm{Im}(z)^2} \cdot \frac{i dz \wedge d\overline{z}}{\mathrm{Im}(z)^2} \\
&= \left(\frac{b-a\mathrm{Re}(z)}{\mathrm{Im}(z)}\right)^2 4\pi \Omega.
\end{split}
\end{equation}

For $v=aNe^{2,2}+bN^2e^{2,2}$ we obtain the explicit formula
\begin{equation} \label{eq:phi_type_III_explicit_1}
\varphi_\mathbb{V}(a Ne^{2,2}+b N^2 e^{2,2}) = e^{-\pi a^2} \phi\left(\frac{b-a\mathrm{Re}(z)}{\mathrm{Im}(z)} \right) \Omega,
\end{equation}
where $\phi:\mathbb{R} \to \mathbb{R}$ is the Schwartz function defined by
\begin{equation} \label{eq:phi_type_III_explicit_2}
\phi(b) = e^{-2\pi b^2}(4\pi b^2-1).
\end{equation}
The most important consequence of this explicit description of $\phi$ is that its Fourier transform $\mathcal{F}\phi$ satisfies 
\begin{equation}\label{eq:Four_transf_const_term_type_III}
\mathcal{F}\phi(0)=0,
\end{equation}
as follows from the identity $\tfrac{d}{db}(-b e^{-2\pi b^2}) = e^{-2\pi b^2}(4\pi b^2-1)$.

\subsubsection{} Let us now compute $\Theta_\mathbb{V}(\tau)'$ explicitly using the above description of $\varphi_\mathbb{V}$. As in \ref{subsection:integrability_type_ii}, we define $W_j^\mathbb{Z}=W_j \cap V_\mathbb{Z}$ and obtain a filtration
$$
0 \subset W_0^\mathbb{Z}=W_1^\mathbb{Z} \subset W_2^\mathbb{Z}=W_3^\mathbb{Z} \subset W_4^\mathbb{Z} = V_\mathbb{Z}
$$
whose associated quotients $\mathrm{Gr}_k^W V_\mathbb{Z}:=W_k^\mathbb{Z}/W_{k-1}^\mathbb{Z}$ are free abelian groups (of rank one in the case of $\mathrm{Gr}_0^W V_\mathbb{Z}$ and $\mathrm{Gr}_4^W V_\mathbb{Z}$). Let us define
$$
W_{2,\mathrm{prim}}^\mathbb{Z} = W_2^\mathbb{Z} \cap \ker N.
$$
Fix a generator $v_0$ of $W_0^\mathbb{Z}$ and vectors $v_1,\ldots,v_{n-1}$ of $W_{2,\mathrm{prim}}^\mathbb{Z}$ such that 
$$
W_{2,\mathrm{prim}}^\mathbb{Z} = \langle v_0, v_1,\ldots, v_{n-1} \rangle,
$$
i.e. so that $Y_{2,\mathrm{prim}}^\mathbb{Z}:=\langle v_1,\ldots,v_{n-1} \rangle$ is a complement of $W_0^\mathbb{Z}$ in $W_{2,\mathrm{prim}}^\mathbb{Z}$. We also fix a vector $v_n' \in V_\mathbb{Z}$ mapping to a generator of $\mathrm{Gr}_4^W V_\mathbb{Z}$ and a vector $v_n \in W_2^\mathbb{Z}$ such that $v_n \equiv Nv_n' \mod W_0$; then the image of $v_n$ in $\mathrm{Gr}_2^W V_\mathbb{Z}$ generates the rank one lattice
$$
N(\mathrm{Gr}_4^W V_\mathbb{Z}) = \mathrm{im}(N:\mathrm{Gr}_4^W V_\mathbb{Z} \to \mathrm{Gr}_2^W V_\mathbb{Z})
$$
(cf. Lemma \ref{lemma:integrality_N_type_iii}.(ii)). Define
$$
Y_2^\mathbb{Z} := \langle v_1,\ldots, v_n \rangle = Y_{2,\mathrm{prim}}^\mathbb{Z} \oplus \langle v_n \rangle
$$
(orthogonal sum). Then $W_0^\mathbb{Z} + Y_2^\mathbb{Z}$ is a (finite index) sublattice of $W_2^\mathbb{Z}$ and the quotient map $W_2 \to \mathrm{Gr}_2^W V$ induces an isometry
\begin{equation} \label{eq:graded_isometry_type_iii}
Y_2^\mathbb{Z} \simeq \mathrm{Gr}_{2,\mathrm{prim}}^W V_\mathbb{Z} \oplus N(\mathrm{Gr}_4^W V_\mathbb{Z})
\end{equation}
onto a (finite-index) sublattice of $\mathrm{Gr}_2^W V_\mathbb{Z}$. With the notation of \ref{subsubsection:Theta_MHS_type_III} we may then write $\mu=\mu_2+\mu_0$ with $\mu_2 \in Y_2^\mathbb{Z} \otimes \mathbb{Q}$ and $\mu_0 \in W_0$ and
\begin{equation} \label{eq:Theta'_decomposition_type_III_1}
\Theta_\mathbb{V}(\tau)'_\mu =  \sum_{\substack{N\lambda+\nu  \equiv \mu_2 \\ \mod W_2^\mathbb{Z}}} \Theta_{\mathbb{V}}(\tau)'_{\lambda \otimes \nu}
\end{equation}
with
\begin{equation} \label{eq:Theta'_decomposition_type_III_2}
\Theta_{\mathbb{V}}(\tau)'_{\lambda \otimes \nu} := \sum_{\substack{v' \in N\lambda + \langle v_n \rangle \\ v \in \nu + Y_{2,\mathrm{prim}}^\mathbb{Z} \\ w \in \mu_0 + W_0^\mathbb{Z}} } \varphi_\mathbb{V}(y^{1/2}(v'+v+w))e^{\pi i x (Q(v',v')+Q(v,v))}.
\end{equation}
To estimate this sum, we use \eqref{eq:vectors_W_2_type_iii} to write
\begin{equation}
    \begin{split}
        v' &= a(v')Ne^{2,2} + b(v')N^2e^{2,2} \\
        v &= \pi_U(v) + b(v)N^2 e^{2,2} \\
        w &= b(w) N^2e^{2,2},
    \end{split}
\end{equation}
where $\pi_U:W_{2,\mathbb{R}} \to U$ is the projection to $U$ and $a, b$ are linear functionals on $W_{2,\mathbb{R}}$. Since $W_0$ is anisotropic and $Q(W_0,W_2)=0$, we have $Q(\pi_U(v),\pi_U(v))=Q(v,v)$ and $Q(v',v')=-a(v')^2$ and hence, by \eqref{eq:phi_type_III_explicit_0} and \eqref{eq:phi_type_III_explicit_1}, 
\begin{equation}
    \begin{split}
\varphi_{\mathbb{V}}(v'+v+w) &= e^{-\pi Q(v,v)}  \\
& \quad \times \varphi_\mathbb{V}(a(v')Ne^{2,2}+b(v+v')N^2e^{2,2}+w). \\ 
&= e^{-\pi (Q(v,v)-Q(v',v'))} \\
& \quad \times \phi\left(\frac{b(v+v'+w)-a(v')\mathrm{Re}(z)}{\mathrm{Im}(z)}\right)\Omega,
    \end{split}
\end{equation}
with $\phi$ given by \eqref{eq:phi_type_III_explicit_2}.

For the theta series $\Theta_\mathbb{V}(\tau)'_{\lambda \otimes \nu}$ this gives
\begin{equation} \label{eq:theta_nilp_'_type_III_explicit_1}
\begin{split}
\Theta_{\mathbb{V}}(\tau)'_{\lambda \otimes \nu} &= \sum_{\substack{v' \in N\lambda+\langle v_n \rangle \\ v \in \nu + Y_{2,\mathrm{prim}}^\mathbb{Z}}} q^{Q(v,v)/2} \overline{q}^{(-Q(v',v')/2)} \\
& \qquad \times \sum_{w \in \mu_0 +  W_0^\mathbb{Z}} \phi\left(\sqrt{y}\frac{b(v+v'+w)-a(v')\mathrm{Re}(z)}{\mathrm{Im}(z)}\right) \Omega.
\end{split}
\end{equation}
To estimate the sum over $\mu_0 + W_0^\mathbb{Z}=\mathbb{Z}v_0$ we apply Poisson summation: writing $\mathcal{F}\phi$ for the Fourier transform of $\phi$ and 
$$
A = \frac{b(v+v' + \mu_0)-a(v')\mathrm{Re}(z)}{\mathrm{Im}(z)/\sqrt{y}},
$$ 
we have
$$
\sum_{n \in \mathbb{Z}} \phi\left(A+\frac{n b(v_0)}{\mathrm{Im}(z)/\sqrt{y}} \right) = \frac{\mathrm{Im}(z)}{|b(v_0)|\sqrt{y}} \sum_{m \in \mathbb{Z}} e^{2\pi i mA} \mathcal{F}\phi\left(\frac{\mathrm{Im}(z)}{b(v_0)\sqrt{y}} m \right).
$$
The term corresponding to $m=0$ vanishes by \eqref{eq:Four_transf_const_term_type_III}; this gives the upper bound, uniform in $A$,

\begin{equation} \label{eq:bound_phi_type_III}
    \begin{split}
& \left| \sum_{w \in W_0^\mathbb{Z}} \phi  \left(\sqrt{y}\frac{b(v+v'+w)-a(v')\mathrm{Re}(z)}{\mathrm{Im}(z)}\right) \right|_t \\
 & \qquad \qquad \leq \frac{\mathrm{Im}(z)}{|b(v_0)|\sqrt{y}} \sum_{m \in \mathbb{Z}-0} \left| \mathcal{F}\phi\left(\frac{\mathrm{Im}(z)}{b(v_0)\sqrt{y}} m \right)\right| |\Omega|_t
    \end{split}
\end{equation}
The right hand side is rapidly decreasing as $\mathrm{Im}(z) \to \infty$. This implies the integrability of $\Theta_\mathbb{V}(\tau)'_\mu$. Similarly, using $Q(W_0,W_2)=0$, we have
$
\Theta^\circ_{\mathbb{V}}(y)'_{m,\mu} = \sum_{\substack{N\lambda+\nu  \equiv \mu_2 \\ \mod W_2^\mathbb{Z}}} \Theta^\circ_{\mathbb{V}}(y)'_{m,\lambda \otimes \nu}
$
with
\begin{equation} \label{eq:Theta'_decomposition_type_III_3}
\begin{split}
\Theta^\circ_{\mathbb{V}}(y)'_{m,\lambda \otimes \nu} \cdot q^m = & \sum_{\substack{v' \in N\lambda+\langle v_n \rangle \\ v \in \nu + Y_{2,\mathrm{prim}}^\mathbb{Z} \\ Q(v+v',v+v')=2m}} q^{Q(v,v)/2} \overline{q}^{(-Q(v',v')/2)} \\
& \qquad \times \sum_{w \in \mu_0 + W_0^\mathbb{Z}} \phi\left(\sqrt{y}\frac{b(v+v'+w)-a(v')\mathrm{Re}(z)}{\mathrm{Im}(z)}\right) \Omega.
\end{split}
\end{equation}
The bound \eqref{eq:bound_phi_type_III} and the fact that $Q$ is positive definite on $Y_{2,\mathrm{prim}}^\mathbb{Z}$ and negative definite on $N(\mathrm{Gr}_4^W V_\mathbb{Z})$ show that $\Theta_\mathbb{V}^\circ(y)'_{m,\lambda \otimes \nu}$ is integrable for all $m$. The identity
$$
\int_{\Delta^*} \Theta_{\mathbb{V}}(\tau)'_\mu =\sum_{m} \left( \int_{\Delta^*} \Theta_{\mathbb{V}}^\circ(y)'_{m,\mu} \right) \cdot q^m
$$
follows by dominated convergence.

\section{Generating series of Noether--Lefschetz numbers}

The goal of this section is to determine the Fourier expansion of the non--holomorphic modular forms $Z_\mathbb{V}(\tau)_\mu$ in Theorem \ref{thm:convergence_main}. We will see that their Fourier coefficients can be expressed in terms of the degrees of the Noether--Lefschetz loci $\mathrm{NL}_\mathbb{V}(m)_\mu$ defined below and some discrete invariants of the limiting mixed Hodge structures arising from the degeneration of $\mathbb{V}$ around each point $P$ in $\overline{S}-S$.

More precisely, to $\mathbb{V}$ one can attach the $q$-series
$$
Z_\mathbb{V}^+(\tau)_\mu := -\deg(\overline{\mathcal{L}}) \delta_{\mu,0} + \sum_{m>0} \mathrm{deg} \  \mathrm{NL}_\mathbb{V}(m)_\mu \cdot q^m, \quad q = e^{2\pi i \tau}
$$
as well as theta series $Z_{\mathbb{V},P}^-(\tau)$ for each $P \in \overline{S} \, \backslash \, S$ (see \eqref{eq:def_theta_MHS_type_II_2} and \eqref{eq:def_theta_MHS_type_III_2}). We will prove the following theorem which implies Theorem \ref{thm:main_thm_intro}.

\begin{theorem} \label{thm:Z_Four_expansion}
Assume that $\mathbb{V}$ satisfies \ref{hypothesis:hyp_on_V}. For all $\tau \in \mathbb{H}$ and $\mu \in \mathcal{V}_\mathbb{Z}^\vee/\mathcal{V}_\mathbb{Z}$,
$$
Z_\mathbb{V}(\tau)_\mu = Z_\mathbb{V}^+(\tau)_\mu + \sum_{P \in \overline{S}\backslash S} Z^-_{\mathbb{V},P}(\tau)_\mu.
$$
\end{theorem}

The proof proceeds by checking that both sides have the same Fourier coefficients. That is, let
$$
Z^-_{\mathbb{V},P}(\tau)_\mu = \sum_{m} Z^-_{\mathbb{V},P}(y)_{m,\mu} \cdot q^m
$$
be the Fourier expansion of $Z^-_{\mathbb{V},P}(\tau)_\mu$ and write similarly
$$
Z^+_\mathbb{V}(\tau)_{m,\mu} = \left\{ \begin{array}{rr}
     \mathrm{deg} \ \mathrm{NL}_\mathbb{V}(m)_\mu & \text{ if } m>0, \\
     -\mathrm{deg}(\overline{\mathcal{L}}), & \text{ if } (m,\mu) = (0,0), \\
     0, & \text{ otherwise,}
\end{array} \right.
$$
for the Fourier coefficients of $Z^+_\mathbb{V}(\tau)_\mu$. Theorem \ref{thm:Z_Four_expansion} is then equivalent to the identity
\begin{equation} \label{thm:Z_Four_coeffs_comparison}
\int_S \Theta^\circ_{\mathbb{V}}(y)_{m,\mu} = Z^+_\mathbb{V}(\tau)_{m,\mu} + \sum_{P \in \overline{S} \, \backslash \, S} Z^-_{\mathbb{V},P}(y)_{m,\mu}
\end{equation}
for all $m$ and $\mu$.

\subsection{Kudla--Millson forms and Noether--Lefschetz loci}

The main input needed to prove Theorem \ref{thm:Z_Four_expansion} is the computation of the residues at the boundary of certain Green functions $\mathfrak{g}^\circ(y)_{m,\mu}$ for the Noether--Lefschetz loci obtained by pulling back Green functions for special divisors on orthogonal Shimura varieties. The latter Green functions were introduced by Kudla in \cite{KudlaAnnals}. Let us briefly recall their definition. Consider the Kudla--Millson theta series
$$
\Theta_{\mathrm{KM}}(\tau)_\mu =  \sum_{v \in \mu + V_\mathbb{Z}} \varphi_{\mathrm{KM}}(y^{1/2} v) e^{\pi i x Q(v,v)}
$$
and let us write 
$$
\Theta_{\mathrm{KM}}(\tau)_\mu = \sum_{m \in \tfrac{1}{2}Q(\mu,\mu)+\mathbb{Z}} \Theta_{\mathrm{KM}}^\circ(y)_{m,
\mu} \cdot q^m
$$
for its Fourier expansion. One of the main properties of  $\Theta_{\mathrm{KM}}(\tau)_\mu$ is that it defines a closed differential form and its Fourier coefficients $\Theta^\circ_{\mathrm{KM}}(y)_{m,\mu}$ are Poincar\'e dual to a certain special divisor $Z(m,\mu)$ (see \cite{KudlaDuke}) whose intersection with $S$ gives the Noether--Lefschetz locus $\mathrm{NL}_\mathbb{V}(m)_\mu$. 

In \cite{KudlaAnnals}, Kudla introduced a Green function $\mathfrak{g}^\circ(y)_{m,\mu}$ for the special divisor $Z(m,\mu)$, i.e. a smooth function on $\Gamma \backslash \mathbb{D}-|Z(m,\mu)|$ satisfying Green's equation
\begin{equation}
\mathrm{dd^c}[\mathfrak{g}^\circ(y)_{m,\mu}] + \delta_{Z(m,\mu)} = [\Theta_\mathrm{KM}^\circ(y)_{m,\mu}-\varphi_{\mathrm{KM}}(0)\delta_{(m,\mu)=(0,0)}].
\end{equation}
Here $\mathrm{d^c}=(4\pi i)^{-1}(\partial - \overline{\partial})$, so that $\mathrm{dd^c}=-(2\pi i)^{-1} \partial\overline{\partial}$, and the term $\delta_{(m,\mu) = (0,0)}$ equals one if $(m,\mu) = (0,0)$ and vanishes otherwise.

Pulling back $\mathfrak{g}^\circ(y)_{m,\mu}$ by the period map $\Phi_\mathbb{V}$ associated with $\mathbb{V}$, we obtain a function $\mathfrak{g}^\circ_\mathbb{V}(y)_{m,\mu}$ whose main properties are summarized in the following Proposition.

\begin{proposition}
For $v$ a local section of $\mathcal{V}_\mathbb{R}$, define
\begin{equation}
\nu_\mathbb{V}^\circ(v) = e^{-2\pi h(s_v)}
\end{equation}
and $\nu_\mathbb{V}(v)=e^{-\pi Q(v,v)} \nu_\mathbb{V}^\circ(v)$. Then
\begin{equation} \label{def:Green_form}
\mathfrak{g}^\circ_\mathbb{V}(y)_{m,\mu} := \int_1^
\infty \left( \sum_{\substack{0 \neq v \in \mu + \mathcal{V}_\mathbb{Z} \\ Q(v,v)=2m }} \nu^\circ_{\mathbb{V}}((yu)^{1/2} v) \right) \frac{du}{u}
\end{equation}
defines a smooth function on $S-|\mathrm{NL}_\mathbb{V}(m)_\mu|$ that satisfies the differential equation
\begin{equation}
\mathrm{dd^c}[\mathfrak{g}^\circ_{\mathbb{V}}(y)_{m,\mu}] + \delta_{\mathrm{NL}_\mathbb{V}(m)_\mu} = [\Theta_\mathbb{V}^\circ(y)_{m,\mu} + \Omega \delta_{(m,\mu)=(0,0)}]
\end{equation}
as currents on $S$. (Here $\delta_{\mathrm{NL}_\mathbb{V}(m)_\mu}$ denotes the current of integration against the divisor associated with $\mathrm{NL}_\mathbb{V}(m)_\mu$, understood to vanish if $m \leq 0$.)
\end{proposition}

Let us fix $m$ and $\mu$ and choose small disks $D_{P,\epsilon}$ around each point $P$ in the support of $\mathrm{NL}_\mathbb{V}(m)_\mu$ as well as in $\overline{S} \, \backslash \, S$ whose radii tend to zero as $\epsilon \to 0$. By the above proposition and the integrability of $\Theta_{\mathbb{V}}^\circ(y)_{m,\mu}$ we have
\begin{equation}\label{eq:Section_5_reduction_to_local_residue}
\begin{split}
\int_S \Theta^\circ_\mathbb{V}(y)_{m,\mu} &+ \mathrm{deg}(\overline{\mathcal{L}})\delta_{(m,\mu)=(0,0)} \\ &= \lim_{\epsilon \to 0} \int_{S-\cup D_{P,\epsilon}} (\Theta^\circ_\mathbb{V}(y)_{m,\mu}+\Omega \delta_{(m,\mu)=(0,0)}) \\
&= \lim_{\epsilon \to 0} \int_{\partial(S-\cup D_{P,\epsilon})} \mathrm{d^c} \mathfrak{g}^\circ_{\mathbb{V}}(y)_{m,\mu} \\
&= \mathrm{deg} \ \mathrm{NL}_\mathbb{V}(m)_\mu - \sum_{P \in \overline{S} \, \backslash \, S} \lim_{\epsilon \to 0} \int_{\partial D_{P,\epsilon}}\mathrm{d^c}\mathfrak{g}^\circ_\mathbb{V}(y)_{m,\mu} \\
&= \mathrm{deg} \ \mathrm{NL}_\mathbb{V}(m)_\mu - \sum_{P \in \overline{S} \, \backslash \, S} \mathrm{res}_P \ \partial \mathfrak{g}^\circ_\mathbb{V}(y)_{m,\mu},
\end{split}
\end{equation}
where in the last line $\mathrm{res}_P$ denotes the residue at $P$ of the $(1,0)$--form $\partial \mathfrak{g}_\mathbb{V}^\circ(y)_{m,\mu}$. Thus to establish \eqref{thm:Z_Four_coeffs_comparison} it suffices to prove the identity
\begin{equation}
\label{thm:Z_Four_coeffs_comparison_2}
-\mathrm{res}_P \ \partial \mathfrak{g}^\circ_\mathbb{V}(y)_{m,\mu} =  Z^-_{\mathbb{V},P}(y)_{m,\mu}
\end{equation}
for all $m$ and $\mu$ and all $P \in \overline{S} \, \backslash \, S$. Note that the residue $\mathrm{res}_P$ depends only on the restriction of $\mathbb{V}$ to a small disk centered at $P$. 

\subsection{Local residue computations} It follows from \eqref{thm:Z_Four_coeffs_comparison_2} that to prove Theorem \ref{thm:Z_Four_expansion} it suffices to prove the following three lemmas. In their statements we assume that $\mathbb{V}$ is an arbitrary $\mathbb{Z}$-PVHS of weight two with $h^{2,0}=1$ satisfying \ref{hypothesis:hyp_on_V} on the punctured unit disk $S=\Delta^*$. With the notation of Section \ref{subsection:local_monodromy} we define
$$
\mathfrak{g}^\circ_{\mathbb{V}}(y)'_{m,\mu} = \int_1^\infty \left( \sum_{\substack{0 \neq v \in (\mu +\mathcal{V}_\mathbb{Z}) \cap W_2 \\ Q(v,v)=2m }} \nu^\circ_{\mathbb{V}}((yu)^{1/2}v) \right) \frac{du}{u}.
$$
The strategy is now the same as for the proof of Theorem \ref{thm:convergence_main}: one first shows that the residue of $\partial \mathfrak{g}^\circ_{\mathbb{V}}(y)_{m,\mu}$ agrees with that of $\partial \mathfrak{g}^\circ_{\tilde{\mathbb{V}}^\mathrm{nilp}}(y)'_{m,\mu}$ and then one computes the latter residue  using the explicit formulas for the ``$\mathbb{R}$--split nilpotent orbit'' $\tilde{\mathbb{V}}^\mathrm{nilp}$ in Sections \ref{subsection:type_ii_degenerations} and \ref{subsection:type_iii_degenerations}.

\begin{lemma}\label{lemma:residue_current_1} For any $m$ and $\mu$, we have
$$
\mathrm{res}_{t=0} \ ( \partial \mathfrak{g}^\circ_{\mathbb{V}}(y)_{m,\mu} - \partial \mathfrak{g}^\circ_{\mathbb{V}^\mathrm{nilp}}(y)'_{m,\mu} )= 0. 
$$
\end{lemma}

\begin{lemma}\label{lemma:residue_current_2} For any $m$ and $\mu$, we have
$$
\mathrm{res}_{t=0} \ ( \partial \mathfrak{g}^\circ_{\mathbb{V}^\mathrm{nilp}}(y)'_{m,\mu} - \partial \mathfrak{g}^\circ_{\tilde{\mathbb{V}}^\mathrm{nilp}}(y)'_{m,\mu} )= 0. 
$$
\end{lemma}

\begin{lemma}\label{lemma:residue_current_3} For any $m$ and $\mu$, we have
$$
-\mathrm{res}_{t=0} \  \partial \mathfrak{g}^\circ_{\tilde{\mathbb{V}}^\mathrm{nilp}}(y)'_{m,\mu} =  Z_{\mathbb{V},P}^-(\tau)_{m,\mu}.
$$
\end{lemma}

The proof of these lemmas is analogous to the proofs of similar lemmas in Sections \ref{subsection:Integrability_Theta''}, \ref{subsection:reduction_to_nilpotent_orbits}, \ref{subsection:integrability_type_ii} and \ref{subsection:integrability_type_iii}. It will be convenient to define
\begin{equation}
\tilde{\Theta}_{\mathbb{V}}(y)_{m,\mu} = \sum_{\substack{v \in \mu
+ \mathcal{V}_\mathbb{Z} \\ Q(v,v) = 2m}} \nu^\circ_{\mathbb{V}}(y^{1/2} v)
\end{equation}
and write 
$$
\tilde{\Theta}_{\mathbb{V}}(y)_{m,\mu}=\tilde{\Theta}_{\mathbb{V}}(y)'_{m,\mu}+\tilde{\Theta}_{\mathbb{V}}(y)''_{m,\mu},
$$
where in $\tilde{\Theta}_{\mathbb{V}}(y)'_{m,\mu}$ the sum runs over vectors in $W_2$ while in $\tilde{\Theta}_{\mathbb{V}}(y)''_{m,\mu}$ it runs over vectors not in $W_2$. Since $\nu^\circ_\mathbb{V}(0)=1$ and hence $\partial \nu^\circ_\mathbb{V}(0)=0$, we can drop the condition $v \neq 0$ in \eqref{def:Green_form} when computing $\partial \mathfrak{g}^\circ_{\mathbb{V}}(y)_{m,\mu}$ ; that is, we have
\begin{equation}
\partial \mathfrak{g}^\circ_{\mathbb{V}}(y)_{m,\mu}= \int_1^\infty \partial \tilde{\Theta}_{\mathbb{V}}(uy)_{m,\mu} \frac{du}{u}
\end{equation}
and 
$$
\partial \mathfrak{g}^\circ_{\mathbb{V}}(y)_{m,\mu} = \partial \mathfrak{g}^\circ_{\mathbb{V}}(y)'_{m,\mu} + \partial \mathfrak{g}^\circ_{\mathbb{V}}(y)_{m,\mu}''
$$
with
\begin{equation}\label{eq:def_transgression_1}
\begin{split}
\partial \mathfrak{g}^\circ_{\mathbb{V}}(y)'_{m,\mu} &= \int_1^\infty \partial \tilde{\Theta}_{\mathbb{V}}(uy)'_{m,\mu} \frac{du}{u} \\
\partial \mathfrak{g}^\circ_{\mathbb{V}}(y)''_{m,\mu}  &= \int_1^\infty \partial \tilde{\Theta}_{\mathbb{V}}(uy)''_{m,\mu} \frac{du}{u}.
\end{split}
\end{equation}

\begin{proof}[Proof of Lemma \ref{lemma:residue_current_1}]
This reduces to
\begin{align}\label{lemma:Green_form_estimate_1_identity_1}
\mathrm{res}_{t=0} \ \partial \mathfrak{g}_{\mathbb{V}}^\circ(y)_{m,\mu}'' &= 0 \\
\label{lemma:Green_form_estimate_1_identity_2}
\mathrm{res}_{t=0} \ ( \partial \mathfrak{g}^\circ_{\mathbb{V}}(y)'_{m,\mu} - \partial \mathfrak{g}^\circ_{\mathbb{V}^\mathrm{nilp}}(y)'_{m,\mu} ) &= 0. 
\end{align}
To prove \eqref{lemma:Green_form_estimate_1_identity_1} we can use \eqref{eq:def_transgression_1} and the explicit expression 
\begin{equation} \label{eq:nu_0_explicit}
\partial \nu^\circ_\mathbb{V}(y^{1/2}v) = \partial(e^{-2\pi y h(s_v)})= e^{-2\pi y h(s_v)} \cdot (-\pi y \partial \|v\|^2_\mathcal{V})
\end{equation}
(recall that $\|v\|^2_{\mathcal{V}}=Q(v,v)+2h(s_v)$ and hence $2 \partial h(s_v) = \partial \|v\|_\mathcal{V}^2$). With the notation of \eqref{eq:Hodge_norm_matrix_form}, we have
$$
\partial \|v\|^2_\mathcal{V} = \sum_{i,j} \overline{a_i}a_j \partial h_{ij}(t).
$$
As in the proof of Lemma \ref{lemma:bound_varphi_V} one shows that the forms $e_{ij}^{-1} \partial h_{ij}$ are nearly bounded and hence that
\begin{equation} \label{lemma:current_identity_2}
|\partial \|v\|_\mathcal{V}^2|_t  \leq C \|v\|^2_{\mathcal{V},t}
\end{equation}
for some positive constant $C$,
giving the bound
$$
|\partial \tilde{\Theta}_\mathbb{V}(uy)''_{m,\mu}|_t \leq C \cdot \sum_{\substack{v \in \mathcal{V}^\vee_\mathbb{Z} \\ v \notin W_2 \\ Q(v,v)=2m}} e^{-2\pi uyh(s_v)}\pi uy\|v\|^2_t
$$
and hence
\begin{equation}
\begin{split}
|\partial \mathfrak{g}^\circ_\mathbb{V}(y)''_{m,\mu}|_t &\leq C \cdot \sum_{\substack{v \in \mathcal{V}^\vee_\mathbb{Z} \\ v \notin W_2 \\ Q(v,v)=2m}} \int_1^\infty e^{-2\pi uyh(s_v)}\pi y\|v\|^2_t du \\
&= C \cdot \sum_{\substack{v \in \mathcal{V}^\vee_\mathbb{Z} \\ v \notin W_2 \\ Q(v,v)=2m}} e^{-2\pi y h(s_v)} \frac{\|v\|_t^2}{2h(s_v)} \\
&= e^{2\pi y m}C \cdot \sum_{\substack{v \in \mathcal{V}^\vee_\mathbb{Z} \\ v \notin W_2 \\ Q(v,v)=2m}} e^{-\pi y \|v\|^2_t} \left(1- \frac{2m}{\|v\|_t^2}\right)^{-1}.
\end{split}
\end{equation}
By \eqref{eq:Hodge_metric_lower_bound_1}, the factor $(1-2m\|v\|_t^{-2})^{-1}$ in the last expression is bounded above by an expression of the form $(1-A(-\log|t|)^{-1})^{-1}$ for some $A>0$. The argument in the proof of Proposition \ref{prop:theta_no_w_2_bound_1} now shows that $|\partial \mathfrak{g}^\circ_\mathbb{V}(y)''_{m,\mu}|_t$ is rapidly decreasing as $t \to 0$, proving \eqref{lemma:Green_form_estimate_1_identity_1}.

To show that \eqref{lemma:Green_form_estimate_1_identity_2} holds one can follow closely the arguments proving Lemma \ref{lemma:bound_diff_varphi_V_V_nilp}  and Proposition \ref{prop:theta'_difference_bound}. One first shows by differentiating \eqref{eq:h_V_h_V_nilp_bound} that
\begin{equation} \label{lemma:current_identity_1}
|\partial\|v\|^2_{\mathcal{V},t}-\partial \|v\|^2_{\mathcal{V}^\mathrm{nilp},t}|_t = O(|t|^B \|v\|^2_{\mathcal{V},t})
\end{equation}
for some positive constant $B$. Multiplying \eqref{eq:exp_V_V_nilp_diff_bound} by $e^{\pi Q(v,v)}$ yields
\begin{equation}
|e^{-2\pi  h_\mathcal{V}(s_v)}-e^{-2\pi  h_{\mathcal{V}^\mathrm{nilp}}(s_v)}|_t < C |t|^B e^{-\pi  (2h_\mathcal{V}(s_v)-A|t|^B \|v\|^2_{\mathcal{V},t})} \|v\|^2_{\mathcal{V},t}.
\end{equation}
Combined with \eqref{lemma:current_identity_2} and \eqref{lemma:current_identity_1}, this gives the bound
\begin{equation}
\begin{split}
|\partial \nu^\circ_{\mathbb{V}}(v) - \partial \nu^\circ_{\mathbb{V}^\mathrm{nilp}}(v)|_t &< C |t|^B e^{-\pi (2h_\mathcal{V}(s_v)-A|t|^B \|v\|^2_{\mathcal{V}})} \cdot  \|v\|^2_{\mathcal{V}}(1+\|v\|^2_{\mathcal{V}}) %\\
%& < C' |t|^B e^{-\pi h_\mathcal{V}(s_v)}
\end{split}
\end{equation}
for some positive constants $A$, $B$ and $C$. 

Writing $f(v,t)=2h_\mathcal{V}(s_v)-A|t|^B \|v\|^2_{\mathcal{V}}$, we have
$$
\int_1^\infty e^{-\pi u f(v,t)} u \|v\|^2_{\mathcal{V}} \frac{du}{u} =  \|v\|^2_{\mathcal{V}} \frac{e^{-\pi  f(v,t)}}{\pi f(v,t)}
$$
and
$$
\int_1^\infty e^{-\pi u f(v,t)} u^2 \|v\|^4_{\mathcal{V}} \frac{du}{u} =  \|v\|^4_{\mathcal{V}} \left(\frac{e^{-\pi  f(v,t)}}{\pi f(v,t)} + \frac{e^{-\pi  f(v,t)}}{(-\pi f(v,t))^2} \right).
$$
By \eqref{eq:Hodge_metric_lower_bound_1}, there exist $A>0$ and $k \in \mathbb{N}$ such that $f(y^{1/2}v,t)^{-1} < Ay^{-1} ((-\log |t|)^k)$ for all non--zero $v \in \mathcal{V}^\vee_\mathbb{Z}$, and so for $v \in \mathcal{V}^\vee_\mathbb{Z}$ with $Q(v,v)=2m$ we obtain
\begin{equation}
\begin{split}
\int_1^\infty |\partial \nu^\circ_\mathbb{V}((yu)^{1/2}v) & -\partial \nu^\circ_{\mathbb{V}^\mathrm{nilp}}((yu)^{1/2}v)|_t \frac{du}{u} \\
& < C (y^{-1}+y^{-2}) |t|^B (-\log|t|)^K e^{2\pi y m} e^{-\pi y \|v\|^2_{\mathcal{V}}/2}
\end{split}
\end{equation}
for positive constants $B$, $C$ and $K$. Property \eqref{lemma:Green_form_estimate_1_identity_2} now follows as in the proof of Proposition \ref{prop:theta'_difference_bound}.
\end{proof}

\begin{proof}[Proof of Lemma \ref{lemma:residue_current_2}]
It suffices to show that
$$
|\partial \mathfrak{g}^\circ_{\mathbb{V}^\mathrm{nilp}}(y)'_{m,\mu} - \partial \mathfrak{g}^\circ_{\tilde{\mathbb{V}}^\mathrm{nilp}}(y)'_{m,\mu} |_t
$$
is bounded for $t$ in a fixed angular sector where $|\cdot|_t$ denotes the Poincar\'e metric, i.e. that the form $ \mathfrak{g}^\circ_{\mathbb{V}^\mathrm{nilp}}(y)'_{m,\mu} - \partial \mathfrak{g}^\circ_{\tilde{\mathbb{V}}^\mathrm{nilp}}(y)'_{m,\mu}$ is nearly bounded.

To see this, let us write $\mathcal{V}=\mathcal{V}^{\mathrm{nilp}}$ and $\tilde{\mathcal{V}} = \tilde{\mathcal{V}}^{\mathrm{nilp}}$. Using \eqref{eq:nu_0_explicit}, \eqref{eq:bound_Hodge_norm_nilp_to_split_nilp_orbits} and the elementary inequality $|e^x-1|\leq |x|e^{|x|}$ we estimate
\begin{equation}
\begin{split}
\pi^{-1}|\partial \nu^\circ_{\mathbb{V}^\mathrm{nilp}}(v)& - \partial \nu^\circ_{\tilde{\mathbb{V}}^\mathrm{nilp}}(v)|_t \\
&
= |e^{-2\pi h_\mathcal{V}(s_v)} \partial\|v\|^2_{\mathcal{V}}- e^{-2\pi h_{\tilde{\mathcal{V}}}(s_v)} \partial\|v\|^2_{\tilde{\mathcal{V}}}|_t \\
&\leq |e^{-2\pi h_\mathcal{V}(s_v)}-e^{-2\pi h_{\tilde{\mathcal{V}}}(s_v)}| \cdot |\partial \|v\|^2_{\tilde{\mathcal{V}}}|_t \\
& \quad + e^{-2\pi h_{\tilde{\mathcal{V}}}(s_v)} \cdot |\partial \|v\|^2_{\mathcal{V}} - \partial\|v\|^2_{\tilde{\mathcal{V}}}|_t \\
%&\leq e^{-2\pi h_{\tilde{\mathcal{V}}}(s_v)} \cdot \left( \pi |\|v\|^2_{\mathcal{V},t} - \|v\|^2_{\tilde{\mathcal{V}},t}| \cdot \|v\|^2_{\tilde{\mathcal{V}},t} \right. \\
% & \quad \left. + |\partial \|v\|^2_{\mathcal{V}} - \partial\|v\|^2_{\tilde{\mathcal{V}}}|_t \right)
& \leq  C e^{-\pi h_{\tilde{\mathcal{V}}}(s_v)} \cdot (-\log |t|)^{-1},
\end{split}
\end{equation}
where $C$ depends only on $m$. This gives
\begin{equation}
\begin{split}
|\partial \mathfrak{g}^\circ_{\mathbb{V}^{\mathrm{nilp}},m}(y)' - \partial \mathfrak{g}^\circ_{\tilde{\mathbb{V}}^{\mathrm{nilp}},m}(y)'|_t & \leq \int_1^\infty \sum_{\substack{v \in \mu + \mathcal{V}_\mathbb{Z} \\ Q(v,v)=2m}} |\nu^\circ_{\mathbb{V}^{\mathrm{nilp}}}((yu)^{1/2}v)-\nu^\circ_{\tilde{\mathbb{V}}^{\mathrm{nilp}}}((yu)^{1/2}v)|_t \frac{du}{u} \\
&\leq C (-\log |t|)^{-1} \int_1^\infty \sum_{\substack{v \in \mu + \mathcal{V}_\mathbb{Z} \\ Q(v,v)=2m}} e^{-\pi yu h_{\tilde{\mathcal{V}}}(s_v)} \frac{du}{u}.
\end{split}
\end{equation}
The proof of Lemma \ref{lemma:residue_current_3} will show that the integrand in the last expression is $O((-\log |t|)/\sqrt{u})$. It follows that $\partial \mathfrak{g}^\circ_{\mathbb{V}^{\mathrm{nilp}},m}(y)' - \partial \mathfrak{g}^\circ_{\tilde{\mathbb{V}}^{\mathrm{nilp}},m}(y)'$ is nearly bounded.
\end{proof}

\begin{proof}[Proof of Lemma \ref{lemma:residue_current_3}]
We consider the type II and Type III cases separately.

Assume first that $\tilde{\mathbb{V}}^\mathrm{nilp}$ has a degeneration of type II at $P$; this is the setting of Section \ref{subsection:integrability_type_ii}. Arguing as in that section, and with the same notation, we note first that for $v \in W_2^\mathbb{Z}$ we have
\begin{equation} \label{eq:V_split_psi_identity_1}
\nu^\circ_{\tilde{\mathbb{V}}^\mathrm{nilp}}(v) =  e^{-2\pi h(s_v)} = e^{-\pi  \|\pi_1(v)\|^2_\mathcal{V}}
\end{equation}
(cf. \eqref{eq:V_split_phi_identity_1}). For $\tilde{\Theta}_{\tilde{\mathbb{V}}^\mathrm{nilp}}(uy)'_{m,\mu}$, this gives
\begin{equation} \label{eq:residue_computation_nilp_type_II_1}
\begin{split}
\tilde{\Theta}_{\tilde{\mathbb{V}}^\mathrm{nilp}}(uy)'_{m,\mu} &= \sum_{\substack{v \in \mu
+ W_2^\mathbb{Z} \\ Q(v,v) = 2m}} \nu^\circ_{\tilde{\mathbb{V}}^\mathrm{nilp}}((uy)^{1/2} v)  \\
&= \sum_{\substack{v \in (\mu + W_2^\mathbb{Z})/W_1^\mathbb{Z} \\ Q(v,v)=2m}} \sum_{v_1 \in W_1^\mathbb{Z}} e^{-\pi yu \|v_1 + \pi_1(v)\|_\mathcal{V}^2}.
\end{split}
\end{equation}
The singularity of the inner sum as $t \to 0$ can be determined using Poisson summation: with $\Phi$ as in \eqref{eq:trivialization_W_1_type_II} we can write
\begin{equation}
\begin{split}
\sum_{v_1 \in W_1^\mathbb{Z}} e^{-\pi yu \|v_1 + \pi_1(v)\|_\mathcal{V}^2} &= \sum_{n \in \mathbb{Z}^2} e^{-2\pi \frac{yu}{\mathrm{Im}(z)} |(a_1+n_1)\alpha_1 + (a_1+n_2)\alpha_2|^2 }  \\
&= \left| \det\begin{pmatrix}
\alpha_1 & \overline{\alpha_1} \\ \alpha_2 & \overline{\alpha_2}
\end{pmatrix} \right|^{-1} \frac{\mathrm{Im}(z)}{yu}  + \mathrm{o}(\mathrm{Im}(z)/yu) \\
&= \frac{1}{2\pi yu}\left|\det\begin{pmatrix}
\alpha_1 & \overline{\alpha_1} \\ \alpha_2 & \overline{\alpha_2}
\end{pmatrix} \right|^{-1} \cdot (-\log|t|) +o((-\log |t|)/(yu)).
\end{split}
\end{equation}
%(CAREFUL WITH FACTORS OF $2$! CHECK AGAIN.) 
To compute the determinant, recall that $\alpha_1$, $\alpha_2$ are defined by
$$
\lambda_j = \alpha_j e^{1,0} + \overline{\alpha_j} e^{0,1}, \qquad j=1,2,
$$
where $\lambda_1, \lambda_2$ are a fixed basis of $W_1^\mathbb{Z}$. Pick $\tilde{\lambda_1}, \tilde{\lambda_2} \in V_\mathbb{Q}$ such that $N \tilde{\lambda_j} = \lambda_j$; then 
$$
\tilde{\lambda_j} \equiv \alpha_j e^{2,1} + \overline{\alpha_j} e^{1,2} \mod W_{2,\mathbb{R}}, \qquad j=1,2,
$$
and by \eqref{eq:Q_matrix_type_II} we have
$$
Q(\tilde{\lambda_j},\lambda_k) = \begin{pmatrix} 0 & i \overline{\alpha_1} \alpha_2 -i\alpha_1 \overline{\alpha_2} \\ i\alpha_1 \overline{\alpha_2} -i\overline{\alpha_1}\alpha_2  & 0 \end{pmatrix}.
$$
It follows that
\begin{equation}
\begin{split}
\left|\det\begin{pmatrix}
\alpha_1 & \overline{\alpha_1} \\ \alpha_2 & \overline{\alpha_2}
\end{pmatrix} \right| &= 2|\mathrm{Im}(\alpha_1\overline{\alpha_2})| \\
& =|\det(Q(\tilde{\lambda_j},\lambda_k))|^{1/2} \\
&= \left(\frac{\mathrm{disc} (\mathrm{Gr}_{3,1}^W Q)}{r_1(V_\mathbb{Z},N)}\right)^{1/2}
\end{split}
\end{equation}
and hence

\begin{equation*}
%\begin{split}
\int_1^\infty \tilde{\Theta}_{\tilde{\mathbb{V}}^\mathrm{nilp}}(uy)'_{m,\mu} \frac{du}{u}
=  Z_{\mathbb{V},P}^-(y)_{m,\mu} \cdot (-\log |t|^2) + \mathrm{o}((-\log |t|)/\sqrt{y}),
%\end{split}
\end{equation*}
which implies the statement for type II degenerations.

Let us now consider the case when $\tilde{\mathbb{V}}^\mathrm{nilp}$ has a degeneration of type III at $P$; this case was considered in Section \ref{subsection:integrability_type_iii}. Arguing as in that section, and with same notation, note first that \eqref{eq:Hodge_norm_explicit_1_type_III} and \eqref{eq:Hodge_norm_explicit_2_type_III} imply that for 
$$
v=v_U+aNe^{2,2}+bN^2e^{2,2} \in W_{2,\mathbb{Z}}
$$
we have
$$
\nu^\circ_{\tilde{\mathbb{V}}^\mathrm{nilp}}(v)=e^{-2\pi  h(s_v)}=e^{-2\pi a^2} e^{-2\pi (b-a\mathrm{Re}(z))^2/\mathrm{Im}(z)^2}
$$

(cf. \eqref{eq:phi_type_III_explicit_0} and \eqref{eq:phi_type_III_explicit_1}). The same argument that led to \eqref{eq:theta_nilp_'_type_III_explicit_1} shows that
$$
\tilde{\Theta}_{\tilde{\mathbb{V}}^\mathrm{nilp}}(yu)'_{m,\mu} = \sum_{\substack{ \lambda + N\nu \, \equiv \, \mu \\ \mod \mathrm{Gr}_2^W V_\mathbb{Z}} } \tilde{\Theta}_{\tilde{\mathbb{V}}^\mathrm{nilp}}(yu)'_{m,\lambda \otimes \nu}
$$
with
\begin{equation} \label{eq:tildetheta_nilp_'_type_III_explicit_1}
\begin{split}
\tilde{\Theta}_{\tilde{\mathbb{V}}^\mathrm{nilp}}(yu)'_{m,\lambda \otimes \nu} = & \sum_{\substack{v' \in N\lambda+\langle v_n \rangle \\ v \in \nu + Y_{2,\mathrm{prim}}^\mathbb{Z} \\ Q(v+v',v+v')=2m}}  e^{2\pi yu Q(v',v')} \\
& \qquad \cdot \sum_{w \in \mu_0+ W_0^\mathbb{Z}} e^{-2\pi \tfrac{yu}{\mathrm{Im}(z)^2}(b(v+v'+w)-a(v')\mathrm{Re}(z))^2}.
\end{split}
\end{equation}

Again the leading term of the inner sum as $t \to 0$ can be determined using Poisson summation: writing $v_0$ for a generator of $W_0^\mathbb{Z}$, we have
\begin{equation*}
\begin{split}
\sum_{w \in \mu_0 + W_0^\mathbb{Z}} e^{-2\pi \tfrac{yu}{\mathrm{Im}(z)^2}(b(v+v'+w)-a(v')\mathrm{Re}(z))^2} &= \frac{\mathrm{Im}(z)}{|b(v_0)|\sqrt{2yu}}+\mathrm{o}(\mathrm{Im}(z)/\sqrt{yu}) \\
&=\frac{-\log |t|}{|b(v_0)| 2\pi \sqrt{2yu}} + \mathrm{o}((-\log|t|)/\sqrt{yu}).
\end{split}
\end{equation*}
To compute $b(v_0)$, pick $\tilde{v}_0 \in V_\mathbb{Q}$ such that $N^2\tilde{v}_0 = v_0$; then
$$
Q(\tilde{v}_0,v_0) = \frac{\mathrm{disc}(\mathrm{Gr}^W_{4,0} Q)}{r_2(V_\mathbb{Z},N)}.
$$
On the other hand, since $v_0=b(v_0) N^2e^{2,2}$, we have
$
\tilde{v}_0 \equiv b(v_0)e^{2,2} \mod W_{2,\mathbb{R}}$ and hence $Q(\tilde{v}_0,v_0)=b(v_0)^2$. Using the notation 
$$
r_L(m)_\mu = \{ v \in \mu+L \ | \ Q(v,v)=2m \}
$$
for the representation numbers of a definite lattice $L$, this shows that
\begin{equation*}
\begin{split}
\tilde{\Theta}_{\tilde{\mathbb{V}}^\mathrm{nilp}}(uy)'_{m,\mu}
\sim & \left(\frac{r_2(V_\mathbb{Z},N)}{2 \mathrm{disc}(\mathrm{Gr}_{4,0}^W Q)} \right)^{1/2} \\
& \times \left( \sum_{a+b=m} r_{Y^\mathbb{Z}_{2,\mathrm{prim}}}(a)_\nu  \cdot r_{\langle v_n \rangle}(b)_{N\lambda} \frac{e^{4\pi y ub}}{4\pi \sqrt{yu}}\right) \cdot (-\log|t|^2),
% \sum_{\iota(\lambda)=\mu} Z_{\mathbb{V},P}^-(\tau)_{m,\lambda} \cdot (-\log |t|^2) + \mathrm{o}((-\log |t|)/\sqrt{y})
\end{split}
\end{equation*}
as $t \to 0$. As remarked in \eqref{eq:graded_isometry_type_iii}, the quotient map $W_2 \to \mathrm{Gr}_2^W V$ induces isometries $Y_{2,\mathrm{prim}}^\mathbb{Z} \simeq (\mathrm{Gr}_{2,\mathrm{prim}}^W V_\mathbb{Z},Q)$ and $\langle v_n \rangle \simeq (\mathrm{Gr}_4^W V_\mathbb{Z},-Q_4)$.
The statement in case III follows.
\end{proof}

\bibliographystyle{amsplain}
\bibliography{refs} 

\providecommand{\bysame}{\leavevmode\hbox to3em{\hrulefill}\thinspace}
\providecommand{\MR}{\relax\ifhmode\unskip\space\fi MR }
\providecommand{\MRhref}[2]{%
  \href{http://www.ams.org/mathscinet-getitem?mr=#1}{#2}
}
\providecommand{\href}[2]{#2}
\begin{thebibliography}{10}

\bibitem{Borcherds}
Richard~E. Borcherds, \emph{Automorphic forms with singularities on
  {G}rassmannians}, Invent. Math. \textbf{132} (1998), no.~3, 491--562.
  \MR{1625724}

\bibitem{BruinierBook}
Jan~H. Bruinier, \emph{Borcherds products on {O}(2, {$l$}) and {C}hern classes
  of {H}eegner divisors}, Lecture Notes in Mathematics, vol. 1780,
  Springer-Verlag, Berlin, 2002. \MR{1903920}

\bibitem{KudlaHoward1}
Jan~H. Bruinier, Benjamin Howard, Stephen~S. Kudla, Michael Rapoport, and
  Tonghai Yang, \emph{Modularity of generating series of divisors on unitary
  {S}himura varieties}, Ast\'{e}risque (2020), no.~421, Diviseurs
  arithm\'{e}tiques sur les vari\'{e}t\'{e}s orthogonales et unitaires de
  Shimura, 7--125. \MR{4183376}

\bibitem{BruinierZemel}
Jan~Hendrik Bruinier and Shaul Zemel, \emph{Special cycles on toroidal
  compactifications of orthogonal {S}himura varieties}, Math. Ann. \textbf{384}
  (2022), no.~1-2, 309--371. \MR{4476226}

\bibitem{CattaniKaplanVHS}
Eduardo Cattani and Aroldo Kaplan, \emph{Degenerating variations of {H}odge
  structure}, no. 179-180, 1989, Actes du Colloque de Th\'{e}orie de Hodge
  (Luminy, 1987), pp.~9, 67--96. \MR{1042802}

\bibitem{CKS}
Eduardo Cattani, Aroldo Kaplan, and Wilfried Schmid, \emph{Degeneration of
  {H}odge structures}, Ann. of Math. (2) \textbf{123} (1986), no.~3, 457--535.
  \MR{840721}

\bibitem{DeligneLocalBehaviour}
P.~Deligne, \emph{Local behavior of {H}odge structures at infinity}, Mirror
  symmetry, {II}, AMS/IP Stud. Adv. Math., vol.~1, Amer. Math. Soc.,
  Providence, RI, 1997, pp.~683--699. \MR{1416353}

\bibitem{https://doi.org/10.48550/arxiv.2301.05982}
Philip Engel, François Greer, and Salim Tayou, \emph{Mixed mock modularity of
  special divisors}, 2023.

\bibitem{https://doi.org/10.48550/arxiv.2007.03037}
Soheyla Feyzbakhsh and Richard~P. Thomas, \emph{Curve counting and s-duality},
  2020.

\bibitem{Funke}
Jens Funke, \emph{Heegner divisors and nonholomorphic modular forms},
  Compositio Math. \textbf{133} (2002), no.~3, 289--321. \MR{1930980}

\bibitem{GarciaSPD}
Luis~E. Garcia, \emph{Superconnections, theta series, and period domains}, Adv.
  Math. \textbf{329} (2018), 555--589. \MR{3783423}

\bibitem{Hain}
Richard Hain, \emph{Periods of limit mixed {H}odge structures}, Current
  developments in mathematics, 2002, Int. Press, Somerville, MA, 2003,
  pp.~113--133. \MR{2059020}

\bibitem{Kollar}
J\'{a}nos Koll\'{a}r, \emph{Subadditivity of the {K}odaira dimension: fibers of
  general type}, Algebraic geometry, {S}endai, 1985, Adv. Stud. Pure Math.,
  vol.~10, North-Holland, Amsterdam, 1987, pp.~361--398. \MR{946244}

\bibitem{KudlaDuke}
Stephen~S. Kudla, \emph{Algebraic cycles on {S}himura varieties of orthogonal
  type}, Duke Math. J. \textbf{86} (1997), no.~1, 39--78. \MR{1427845}

\bibitem{KudlaAnnals}
\bysame, \emph{Central derivatives of {E}isenstein series and height pairings},
  Ann. of Math. (2) \textbf{146} (1997), no.~3, 545--646. \MR{1491448}

\bibitem{KudlaMillson1}
Stephen~S. Kudla and John~J. Millson, \emph{The theta correspondence and
  harmonic forms. {I}}, Math. Ann. \textbf{274} (1986), no.~3, 353--378.
  \MR{842618}

\bibitem{KudlaMillson2}
\bysame, \emph{The theta correspondence and harmonic forms. {II}}, Math. Ann.
  \textbf{277} (1987), no.~2, 267--314. \MR{886423}

\bibitem{KudlaMillson3}
\bysame, \emph{Intersection numbers of cycles on locally symmetric spaces and
  {F}ourier coefficients of holomorphic modular forms in several complex
  variables}, Inst. Hautes \'{E}tudes Sci. Publ. Math. (1990), no.~71,
  121--172. \MR{1079646}

\bibitem{Morrison}
David~R. Morrison, \emph{The {C}lemens-{S}chmid exact sequence and
  applications}, Topics in transcendental algebraic geometry ({P}rinceton,
  {N}.{J}., 1981/1982), Ann. of Math. Stud., vol. 106, Princeton Univ. Press,
  Princeton, NJ, 1984, pp.~101--119. \MR{756848}

\bibitem{PetersSteenbrink}
C.~A.~M. Peters and J.~H.~M. Steenbrink, \emph{Monodromy of variations of
  {H}odge structure}, vol.~75, 2003, Monodromy and differential equations
  (Moscow, 2001), pp.~183--194. \MR{1975567}

\bibitem{Scheithauer}
Nils~R. Scheithauer, \emph{Some constructions of modular forms for the {W}eil
  representation of {${\rm SL}_2(\Bbb{Z})$}}, Nagoya Math. J. \textbf{220}
  (2015), 1--43. \MR{3429723}

\bibitem{Schmid}
Wilfried Schmid, \emph{Variation of {H}odge structure: the singularities of the
  period mapping}, Invent. Math. \textbf{22} (1973), 211--319. \MR{382272}

\bibitem{StewartVologodsky}
Allen~J. Stewart and Vadim Vologodsky, \emph{Motivic integral of {K}3 surfaces
  over a non-archimedean field}, Adv. Math. \textbf{228} (2011), no.~5,
  2688--2730. \MR{2838055}

\bibitem{VoisinII}
Claire Voisin, \emph{Hodge theory and complex algebraic geometry. {II}},
  Cambridge Studies in Advanced Mathematics, vol.~77, Cambridge University
  Press, Cambridge, 2003, Translated from the French by Leila Schneps.
  \MR{1997577}

\bibitem{Zagier}
Don Zagier, \emph{Nombres de classes et formes modulaires de poids {$3/2$}}, C.
  R. Acad. Sci. Paris S\'{e}r. A-B \textbf{281} (1975), no.~21, Ai, A883--A886.
  \MR{429750}

\bibitem{Zagier123}
\bysame, \emph{Elliptic modular forms and their applications}, The 1-2-3 of
  modular forms, Universitext, Springer, Berlin, 2008, pp.~1--103. \MR{2409678}

\bibitem{Zucker}
Steven Zucker, \emph{Remarks on a theorem of {F}ujita}, J. Math. Soc. Japan
  \textbf{34} (1982), no.~1, 47--54. \MR{639804}

\end{thebibliography}

\author{\noindent \small \textsc{Department of Mathematics, University College London, Gower Street, London WC1E 6BT,
  United Kingdom} \\ 
  {\it E-mail address:} \texttt{l.e.garcia@ucl.ac.uk}}

\end{document}